\documentclass[10pt]{article}

\usepackage{amsmath,amssymb,amsfonts,amsthm,amsbsy,mathtools}
\usepackage{graphicx,xcolor}
\usepackage{hyperref}
\usepackage{enumerate}
\hypersetup{
    colorlinks=true,
    linkcolor=blue,
    citecolor=darkgray
}
\usepackage[T1]{fontenc}
\usepackage{footnote,pifont,dsfont}
\usepackage[mathscr]{eucal}
\usepackage{fullpage}
\usepackage{thmtools}
 
\usepackage{pgf,tikz}
\usepackage{mathrsfs}  
\usetikzlibrary{arrows} 

\usepackage{algpseudocode}
\usepackage[algoruled,boxed,lined]{algorithm2e}

\newtheorem{theorem}{Theorem}

\newtheorem{proposition}[theorem]{Proposition}
\newtheorem{pr}[theorem]{Proposition}
\newtheorem{lemma}[theorem]{Lemma}

\numberwithin{theorem}{section}
\numberwithin{corollary}{section}
\numberwithin{proposition}{section}
\numberwithin{pr}{section}
\numberwithin{lemma}{section}

\theoremstyle{definition}
\newtheorem{defn}[theorem]{Definition}

\theoremstyle{remark}
\newtheorem{rem}[theorem]{Remark}

\newtheorem*{unremarks}{Remarks}

\newtheorem{examplex}[theorem]{Example}
\newenvironment{example}
  {\pushQED{\qed}\examplex}
  {\popQED\endexamplex}

\renewcommand\thmcontinues[1]{Continued}

\newcommand{\mA}{\ensuremath{\mathcal{A}}}
\newcommand{\A}{\ensuremath{\mathcal{A}}}

\newcommand{\mB}{\ensuremath{\mathcal{B}}}
\newcommand{\mC}{\ensuremath{\mathcal{C}}}
\newcommand{\CC}{\ensuremath{\mathcal{C}}}

\newcommand{\mH}{\ensuremath{\mathcal{H}}}

\newcommand{\mS}{\ensuremath{\mathcal{S}}}
\newcommand{\mT}{\ensuremath{\mathcal{T}}}

\newcommand{\mM}{\ensuremath{\mathcal{M}}}
\newcommand{\M}{\ensuremath{\mathcal{M}}}
\newcommand{\mV}{\ensuremath{\mathcal{V}}}
\newcommand{\sing}{\ensuremath{\mathcal{V}}}

\newcommand{\mJ}{\mathcal{J}}
\newcommand{\mm}{\ensuremath{\mathbf{m}}}

\newcommand{\bV}{\ensuremath{\mathbb{V}}}
\newcommand{\bU}{\ensuremath{\mathbb{U}}}

\newcommand{\mMR}{\ensuremath{\mathcal{M_{\R}}}}

\newcommand{\ba}{\ensuremath{\mathbf{a}}}
\newcommand{\bb}{\ensuremath{\mathbf{b}}}
\newcommand{\cc}{\ensuremath{\mathbf{c}}}
\newcommand{\dd}{\ensuremath{\mathbf{d}}}
\newcommand{\hh}{\ensuremath{\mathbf{h}}}
\newcommand{\bg}{\ensuremath{\mathbf{g}}}
\newcommand{\be}{\ensuremath{\mathbf{e}}}
\newcommand{\bff}{\ensuremath{\mathbf{f}}}
\newcommand{\bi}{\ensuremath{\mathbf{i}}}

\newcommand{\bl}{\ensuremath{\boldsymbol \ell}}

\newcommand{\bm}{\ensuremath{\mathbf{m}}}

\newcommand{\bpt}{\ensuremath{\mathbf{p}}}
\newcommand{\pp}{\ensuremath{\mathbf{p}}}
\newcommand{\br}{\ensuremath{\mathbf{r}}}
\newcommand{\rr}{\ensuremath{\mathbf{r}}}
\newcommand{\bv}{\ensuremath{\mathbf{v}}}
\newcommand{\uu}{\ensuremath{\mathbf{u}}}
\newcommand{\vv}{\ensuremath{\mathbf{v}}}
\newcommand{\ww}{\ensuremath{\mathbf{w}}}
\newcommand{\bq}{\ensuremath{\mathbf{q}}}
\newcommand{\bw}{\ensuremath{\mathbf{w}}}
\newcommand{\bx}{\ensuremath{\mathbf{x}}}
\newcommand{\xx}{\ensuremath{\mathbf{x}}}
\newcommand{\by}{\ensuremath{\mathbf{y}}}
\newcommand{\yy}{\ensuremath{\mathbf{y}}}
\newcommand{\bz}{\ensuremath{\mathbf{z}}}
\newcommand{\zz}{\ensuremath{\mathbf{z}}}
\newcommand{\one}{\ensuremath{\mathbf{1}}}

\newcommand{\lgrad}{\ensuremath{{\nabla_{\log}}}}
\newcommand{\grad}{\ensuremath{\nabla}}

\newcommand{\N}{\ensuremath{\mathbb{N}}}
\newcommand{\Z}{\ensuremath{\mathbb{Z}}}
\newcommand{\R}{\ensuremath{\mathbb{R}}}
\newcommand{\C}{\ensuremath{\mathbb{C}}}

\newcommand{\bbj}{\ensuremath{\mathbf{b}^{(j)}}}

\newcommand{\rhat}{\hat{\bf r}}

\newcommand{\brhat}{\hat{\br}}

\newcommand{\bss}{\mathbf{s}}

\newcommand{\bepsilon}{\ensuremath{\boldsymbol \epsilon}}

\newcommand{\bs}{{\ensuremath{\boldsymbol{\sigma}}}}
\newcommand{\brho}{\ensuremath{\boldsymbol \rho}}

\newcommand{\bt}{\ensuremath{\boldsymbol \theta}}
\newcommand{\btil}{\tilde{\bb}}

\newcommand{\btau}{\ensuremath{\boldsymbol \tau}}
\newcommand{\bzer}{\ensuremath{\mathbf{0}}}

\newcommand{\bone}{\ensuremath{\mathbf{1}}}
\newcommand{\obeta}{\ensuremath{\overline{\beta}}}
\newcommand{\tomega}{\ensuremath{\tilde{\omega}}}

\newcommand{\tG}{\tilde{G}}

\DeclareMathOperator\erf{\Psi}

\newcommand{\sgn}{\operatorname{sgn}}

\newcommand{\Comp}{\C}

\font\elevenss=cmss11

\font\eightss=cmss8

\font\sixss=cmss8 at 6pt

\newfam\ssfam
\textfont\ssfam=\elevenss \scriptfont\ssfam=\eightss
 \scriptscriptfont\ssfam=\sixss
\def\ss{\fam\ssfam \elevenss}%
\def\ssb{\fam\ssfam \eightss}%
\def\mO{{\mathcal O}}
\def\crit{\mbox{\ss crit}}
\def\contrib{\mbox{\ss contrib}}
\def\contribS{\mbox{\ssb contrib}}
\def\Cox{\hfill \Box}
\def\R{{\mathbb R}}
\def\C{{\mathbb C}}
\def\Z{{\mathbb Z}}
\def\torus{\btau}
\def\adj{{\rm Adj}\,}
\def\ee{\varepsilon}
\def\low{h_{\min}}

\newcommand{\Em}[1]{\textit{\textbf{#1}\index{#1|textbf}}}

\author{Yuliy Baryshnikov\thanks{Department of Mathematics, University of 
Illinois, Urbana, IL 61801, USA, {\tt ymb@illinois.edu}}, 
Stephen Melczer\thanks{Department of Combinatorics \& Optimization,
University of Waterloo, 200 University Avenue West,
Waterloo, ON N2L 3G1, Canada, {\tt smelczer@uwaterloo.ca} }, 
and 
Robin Pemantle\thanks{University of Pennsylvania, Department of 
Mathematics, 209 S. 33rd Street, Philadelphia, PA 19104, 
{\tt pemantle@math.upenn.edu}}}

\title{Asymptotics of multivariate sequences IV: generating functions 
with poles on a hyperplane arrangement}
\begin{document}
\maketitle

\begin{abstract}
Let $F(z_1,\dots,z_d)$ be the quotient of an analytic function with a 
product of linear functions.  Working in the framework of analytic
combinatorics in several variables, we compute asymptotic formulae for the 
Taylor coefficients of $F$ using multivariate residues and saddle-point 
approximations.  Because the singular set of $F$ is the union of 
hyperplanes, we are able to make explicit the topological decompositions
which arise in the multivariate singularity analysis. In addition to 
effective and explicit asymptotic results, we provide the first
results on transitions between different asymptotic regimes, and 
provide the first software package to verify and compute 
asymptotics in non-smooth cases of 
analytic combinatorics in several variables.
It is also our hope that this paper 
will serve as an entry to the more advanced corners of analytic 
combinatorics in several variables for combinatorialists. 
\end{abstract}

\section{Introduction} 

In this paper we study the coefficients of meromorphic functions 
\begin{equation} \label{eq:prototype}
F(\bz) = F(z_1,\dots,z_d) 
   = \frac{G(\bz)}{\prod_{j=1}^m \ell_j(\bz)^{p_j}}
\end{equation}
whose denominator is the product of positive integer powers of real linear 
functions $\ell_j$.  Such functions arise, among other places, in queuing theory.

\begin{example} \label{eg:queuing}
The so-called partition generating function for a closed multiclass
queuing network with one infinite server has the form
\begin{equation} \label{eq:queuing}
F(\bz) = \frac{e^{z_1 + z_2 + \cdots + z_d}}
{\prod_{j=1}^m\left(1-\sum_{i=1}^d\rho_{ij}z_j\right)}
\end{equation}
for real constants $\rho_{ij}>0$ depending on 
model parameters~\cite[Eq.~(2.26)]{BertozziMcKenna1993}.
\end{example}

Bertozzi and McKenna~\cite{BertozziMcKenna1993} approached the 
asymptotic analysis of the queuing system described in~\eqref{eq:queuing}
by noting that the multivariate Cauchy integral which evaluates these
coefficients can be represented as a sum of integrals over basic homology 
cycles in the domain of holomorphy of the complex $d$-form 
$\zz^{-\rr} F(\zz) d\zz$.  In certain low-dimensional cases, they 
were able to exploit linear relations among these cycles to determine 
dominant asymptotics via multivariate residues.
Since that time, the theory of multivariate coefficient extraction
via the field of \emph{analytic combinatorics in several variables 
(ACSV)}~\cite{PW-book,melczer-book} has grown substantially.  

The techniques of ACSV aim to characterize the asymptotic behaviour 
of the series coefficients $\{ a_\rr \}$ of a convergent
series expansion $F(\zz) = \sum_{\rr}a_\rr\zz^\rr$ as
$\rr \to \infty$ with the normalized vector 
$\rhat = \rr / |\rr|$ staying in a bounded set, where $|\rr|$ 
denotes the 1-norm $|\rr|=|r_1| + \cdots + |r_d|$.
Suppose that $F(\zz) = G(\zz)/H(\zz)$, 
where $H(\zz) = \prod_{j=1}^m h_j(\zz)^{p_j}$ for
polynomials $h_j$ each vanishing on a smooth variety (generalizing
the case when $H$ is a product of linear factors).  
The series coefficients $\{ a_\rr \}$ exhibit uniform asymptotic behaviour
in $\rr$ aside from certain degenerate cases having to do with
nontransverse intersections of these varieties or boundary directions $\rhat$
where there are transitions in asymptotic behaviour.
An asymptotic expansion for $a_\rr$ can usually 
be obtained by taking an integer sum of saddle-point
integrals $I(\bs)$ localized near certain \emph{contributing points} $\bs$ where
the denominator $H$ vanishes. Pemantle and Wilson~\cite{PW2} 
characterize the asymptotic behaviour of the types of 
local integrals that come up in this paper (so-called \emph{transverse
multiple points}); see also~\cite{RaichevWilson2011} for explicit
formulae.

Computing the set of contributing points $\bs$ over which to sum
local integrals to determine asymptotics can be easy in some applications, 
but is difficult (perhaps even undecidable) in general. When the denominator 
under consideration is a product
of linear factors, Bertozzi and McKenna~\cite{BertozziMcKenna1993} 
determine this set for a few examples and speculated on 
the existence of a theory to compute the set in general.  
Outside of the linear denominator case, which we cover
extensively in this paper, when the singular set of $F$
is a manifold then it is known how to compute the
contributing points both in the bivariate 
case~\cite{DeVriesHoevenPemantle2011} and when the 
contributing points lie on the boundary of the domain of
convergence of the power series~\cite{MelczerSalvy2021}.
There is currently no known algorithm to determine contributing singularities
for general meromorphic (or even rational) functions, 
and such an algorithm would need
to decide certain deep topological questions. However, the 
current state of knowledge is enough for many problems of 
combinatorial origin.

In addition to queuing theory, other areas where generating functions 
having a form similar to~\eqref{eq:prototype} arise include
Markov modeling~\cite{karloff},  
lattice point enumeration~\cite{deloera-sturmfels}, 
and discrete probability theory 
(see Example~\ref{eg:probs} below). 
Gaussian and other limit theorems also follow from the asymptotic 
extraction of coefficients of multivariate generating functions. 

When the denominator of $F$ has only linear factors, the singular
set of $F$ forms a hyperplane arrangement, which allows us to describe
which critical points $\bs$ are contributing points that
affect asymptotic behaviour of a coefficient sequence. 
In this paper we give an algorithm to determine dominant 
asymptotics for any such meromorphic function under two assumptions:
(1) $\br$ is {\bf generic}, meaning that $\rr  / |\rr|$ does not 
approach one of a codimension~1 set of bad boundary directions, and 
(2) the numerator $G(\zz)$ is polynomial.  Algorithm~\ref{alg:1} below
summarizes the procedure under the additional assumption that 
the functions $\ell_j$ are linearly independent, and a Maple implementation 
is detailed in Section~\ref{sec:implement}.  
Later, Algorithm~\ref{alg:SimpleDecomp} handles the more general
case where the $\{ \ell_j \}$ can be linearly dependent.  When $G$ is
allowed to be a more general entire function, the same two algorithms
give valid asymptotic formulae, with the proviso that the determination
of which among finitely many terms of the formula asymptotically
dominate may require further investigation.  In non-generic directions,
while we have no complete algorithm, we show how the desired estimates
can be written as integral transforms and compute these in a variety
of cases. We also give the first results discussing transitions in
asymptotics around non-generic directions.

The remainder of the paper is structured as follows.  General background
related to ACSV is given in Section~\ref{sec:ACSV}.  Results and motivating
examples are given in Section~\ref{sec:results}.  
Section~\ref{sec:simplegeneric} gives the heart of the analysis,
under the assumption of linear independence of the $\{ \ell_j \}$.
Here, the topological decomposition of the cycle of integration
in Cauchy's integral formula is decomposed into certain Morse-theoretically
determined cycles for which the integral has an easily computed
asymptotic form.  Section~\ref{sec:nonsimple-generic} extends the
results to allow for linear dependencies among the $\{ \ell_j \}$.
Finally, Section~\ref{sec:simple-nongeneric} studies a number of
cases where $\rr / |\rr|$ approaches a boundary direction, giving
complete results in a scaling window of width $|\rr|^{1/2}$ in the
case of an ordinary boundary direction 
in terms of negative moments of Gaussian random variables.  

\begin{rem}
Some of our exposition follows the textbook~\cite{melczer-book} of the
second author, which was written at the same time as much of this paper.
In addition to a self-contained presentation that makes clearer the
relationship between our arguments and more general results in Morse theory,
the new contributions of this paper include the first software implementation 
to verify and compute asymptotics in non-smooth cases of analytic 
combinatorics in several variables, the first results 
on transitions of behaviour around
non-generic directions, a proof that all critical points must be real
(giving a complete classification of when critical points can occur, and 
simplifying other arguments), and a classification of when
drops in the `neighbourhood' exponential growth of a sequence can occur.
\end{rem}

ACSV requires a number of techniques not always familiar to 
combinatorialists: Morse theory, computational algebraic geometry
and the theory of singular integral transforms.  Nevertheless, 
our aim is to provide an introductory exposition.  In order to 
help motivate the homological arguments taken in modern approaches 
to this topic, we illustrate our techniques on several examples and
provide a Maple implementation of our work at
\begin{center}
\url{https://github.com/ACSVMath/ACSVHyperplane}
\end{center} 

\section{ACSV background} \label{sec:ACSV}

The use of analytic techniques to derive asymptotic information about a 
sequence $(a_n)$ from properties of its generating function $f(z) = 
\sum_{n=0}^\infty a_nz^n$ is the domain of \emph{analytic combinatorics}.  
When the generating function $f(z)$ represents an analytic function at 
the origin, Cauchy's integral formula implies
\[ a_n = \frac{1}{2\pi i} \int_C f(z) \frac{dz}{z^{n+1}}, \]
where $C$ is any positively oriented circle sufficiently close to the 
origin.  By deforming the domain of integration $C$, one can typically 
use classical integral methods to obtain asymptotic results.  Standard 
references include Flajolet and Sedgewick~\cite{FlajoletSedgewick2009}, 
Odlyzko~\cite{Odlyzko1995}, and Henrici~\cite{Henrici1977}.

More recently, ACSV has developed tools for the multivariate asymptotic 
analysis of generating functions.  Fix a dimension $d\in\N$. 
Given a multi-dimensional vector $\bz = (z_1,\dots,z_d) \in \C^d$ 
and index $\bi \in \Z^d$, let $\bz^{\bi} := z_1^{i_1} \cdots 
z_d^{i_d} \in \C$.  Generalizing from the univariate case, if the 
multivariate generating function
$$ F(\bz) = \sum_{\rr \in \N^d} a_{\rr} \bz^{\rr} $$
represents a power series at the origin then the Cauchy integral formula gives 
an analytic representation
\begin{equation} \label{eq:intoCIF} 
a_{\rr} = \left(\frac{1}{2\pi i}\right)^d \int_{\mT} 
   F(\bz) \frac{d\bz}{\bz^{\rr + \bone}},  
\end{equation}
where $\mT$ is the product of sufficiently small positively oriented 
circles.  The aim is often to determine asymptotics of the sequence 
$[\bz^{\br}]F(\bz) = a_{\rr}$ as $\rr = n\rhat$ and 
$n \rightarrow \infty$, for some fixed \emph{direction} 
$\rhat \in \R_{>0}^d$. 
Although $a_{\br}$ is only non-zero when $\br$ has non-negative 
integer coordinates, the theory as detailed in this paper shows that 
asymptotics typically vary smoothly with $\br$, allowing one to make 
asymptotic statements about $a_{\br}$ for generic directions $\rhat$ 
or when a normalization of $\rr$ approaches $\rhat$ sufficiently 
quickly (this will be made precise below).

The earliest results in this 
area~\cite{BenderRichmond1983,GaoRichmond1992,Hwang1998,BenderRichmond1999} 
were based on showing that the \emph{sections} $\sum_{\br : r_1 = k} 
a_{\br}\bz^{\br}$ are well approximated by a quasi-power $C_kz_1^k 
g(z_2,\dots,z_d)^k$, leading to a central limit theorem and Gaussian 
behaviour of the generating function coefficients. Since the early 2000s, 
a systematic program to determine asymptotics of multivariate generating 
function coefficients has been developed by the \emph{analytic combinatorics
in several variables} project\footnote{See \url{acsvproject.com} for a
listing of papers in this project.}~\cite{PW1,PW2,BP-cones,PW-book,
BaryshnikovMelczerPemantle2021,melczer-book}.
In the earliest of these works~\cite{PW1,PW2} the deformations were accomplished
by {\em ad hoc} surgeries.  In~\cite{BP-cones}, existence of the 
relevant deformations was shown to follow from the theory of 
hyperbolic functions~\cite{ABG} and a Morse-theoretic structure
was provided to compute and interpret them; this Morse-theoretic
approach was recently made rigorous in~\cite{BaryshnikovMelczerPemantle2021}
by working around certain standard assumptions of Morse 
theory that don't hold in ACSV contexts.

Let $F(\bz)$ be a meromorphic function with poles along an algebraic
variety $\sing$, and let $\mM = \C_*^d 
\setminus \mV$ denote the domain on which the integrand of the Cauchy 
integral~\eqref{eq:intoCIF} is holomorphic.  The value of the 
Cauchy integral depends 
only on the homology class of $\mT$ in $H_d (\M)$ so, for instance, 
we can deform $\mT$ in $\mM$ without changing the value 
of the integral.  

As $\rr \to \infty$ the magnitude of the Cauchy integrand 
of~\eqref{eq:intoCIF} is controlled by the points $\zz$
for which $|\zz|^{-\rr}$ is maximized.  Integrals are usually easiest
to approximate when the contribution from integration away from
the maximum value of the integrand is negligible, leading us to try and
deform $\mT$ to a \emph{minimax contour}.  
The key to computing integral transforms
is that minimizing $\max_{\mT} |\zz|^{-\rr}$ will produce a contour $\mT$ where,
near the point where $|\zz|^{-\rr}$ is maximized, 
the integrand is in {\em stationary phase}, 
and may thus (typically) be approximated using standard analytic techniques.
This minimax problem is not affected by rescaling $\rr$, so 
the analysis depends only on the vector $\rhat := \rr / |\rr|$ 
where $|\rr| := r_1 + \cdots + r_d$
is $L^1$-norm of $\rr$ (the $L^1$-norm is natural in many contexts, 
but can be replaced by another norm if desired).

Taking logarithms, we need to look near points $\zz$ such that 
the {\bf height function} 
\begin{equation} \label{eq:h}
h_{\rhat} (\zz) := - \sum_{j=1}^d \hat{r}_j \log |z_j| \, .
\end{equation}
is maximized.
The decomposition of cycles into canonical height-minimizing cycles
is the province of Morse theory (when $\sing$ is smooth) and 
stratified Morse theory (when $\sing$ is a stratified space such
as an algebraic variety).  

The result of a Morse-theoretic analysis
is as follows.  
Any algebraic set $\sing$ can be decomposed into a finite set of 
smooth manifolds known as strata, each of which (under the assumptions of
this paper) contains 
a finite set of critical points for the height function (points
where the differential of the height function restricted to the strata
is zero).
Any cycle $C$ may be written as the sum of certain `attachment cycles' for
which height is maximized uniquely at one of these critical 
points\footnote{The fact that the original cycle is in $\M \subseteq
\sing^c$ while the Morse theory is done on $\sing$ will be 
reconciled in Section~\ref{sec:simplegeneric}.}.
These attachment cycles are known as {\bf linking tori}, the one near
a critical point $\bs$ being denoted by $\torus_{\bs}$.  The integral 
over one of these linking tori is a standard integral transform whose
asymptotics are sometimes easy to compute (e.g., in the generic
directions referenced above) or are in the province of singular
integral theory and are explicit to various degrees depending on
the geometry near the critical point.  Restricting to the case
where $\sing$ is the union of hyperplanes results in a bound
on the complexity of the singular integral, and therefore a reasonably
complete answer.

The case where the denominator of $F$ is the product of linear
functions with real coefficients simplifies for a second reason,
beyond the bounded complexity of the geometry of $\sing$.  The
analysis from above may be summarized by a homological computation
\begin{equation} \label{eq:homological}
[\mT] = \sum_{\bs \in \Omega} k_{\bs} \torus_{\bs} \, ,
\end{equation}
where $\Omega$ is the set of stratified critical points for $h$ on $\sing$
and the $k_{\bs}$ are integer invariants that typically hard to compute.
When $\bs$ lies on the boundary of the domain of convergence of $F(\zz)$
and the singular geometry at $\bs$ is sufficiently nice, it is known that
$k_{\bs}\in\{\pm1,0\}$. When this (rather strong and computationally expensive)
property cannot be established, the most promising method to date 
for determining the coefficients $k_{\bs}$ relies on the fact that,
when $\rhat$ is a rational direction, the coefficient sequences $a_{n 
\br}$ satisfy linear differential equations with polynomial 
coefficients~\cite{Christol1984,Lipshitz1988} which can be computed 
using so-called Creative Telescoping methods~\cite{Lairez2016}. 
Numeric analytic continuation~\cite{Mezzarobba2010} can then be used to 
rigorously determine the constants appearing in asymptotics up to any 
fixed accuracy, recovering the integers $k_{{\bs}}$. 
Representation of the denominator of $F$ as a product of real linear
functions allows explicit computation of $\{ k_{\bs} \}$, going through
the intermediate homology basis of {\bf imaginary fibers}, which relies
in turn on convexity properties following from the form of the
denominator of $F$.

\section{Notation and results} \label{sec:results}

In this section we state our main results, after some
basic definitions.

\subsection{Definitions} 
\label{ss:definitions}
We begin by introducing the quantities appearing in our analysis. 
We make use of terminology from ACSV, the study of hyperplane
arrangements, and topological constructions.

\subsubsection*{Hyperplane arrangements}
Fix a ratio 
$$F(\bz) = F(z_1,\dots,z_d) = \frac{G(\bz)}{H(\bz)} \,, $$
where $G(\bz)$ is an entire function and
$$ H(\bz) = \prod_{j=1}^m\ell_j(\bz)^{p_j} $$
for integers $p_j \geq 1$ and real linear functions
$$\ell_j(\bz) = 1 - (\bbj,\bz) = 1 - b^{(j)}_1z_1 - \cdots - b^{(j)}_d z_d \, .  $$
We always assume that $G$ and $H$ are coprime, which in this case simply 
means that $G$ does not identically vanish on the zero set of one of
the $\ell_j$ (when $G$ is a polynomial this means $G$ and $H$ are coprime
polynomials). The singular set
$$ \mV = \{\bz \in \C^d : H(\bz)=0 \} $$
is composed of a union of hyperplanes and thus defines a \emph{hyperplane 
arrangement}~\cite{OrlikTerao1992}, a well-studied combinatorial 
and topological object from which we borrow some of our terminology.  
The notation $\sing_{\R}$ denotes the real arrangement, a union of 
corresponding real hyperplanes in $\R^d$ that is the intersection of 
$\sing$ with $\R^d$.

Given complex-valued functions $g_1,\dots,g_s$, we let 
$\bV(g_1,\dots,g_s)$ denote their set of common complex zeroes. The 
hyperplane arrangement $\mV = \bV(\ell_1 \cdots\ell_m)$ can be decomposed into 
closed \Em{flats}: for each subset $\{k_1,\dots,k_s\} \subset 
\{1,\dots,m\}$ the flat $\mV_{k_1,\dots,k_s} := 
\bV(\ell_{k_1},\dots,\ell_{k_s})$ is the intersection of the appropriate 
hyperplanes.  Because two flats with distinct index sets could be
equal, we let $\mA$ be the collection of maximal subsets 
of $\{1,\dots,m\}$ corresponding to non-empty distinct flats: this is the 
\Em{hyperplane arrangement} defined by $\mV$.  Thus, by definition, $\mV_{S 
\cup \{k\}} \neq \mV_S$ for any $S \in \mA$ and $k \notin S$.

\begin{defn}[simple arrangements]
The arrangement $\A$ (or singular set $\mV$, or rational function $F(\bz)$) 
is said to be \Em{simple} if for any subset $\{ k_1 , \ldots , k_s \} 
\subseteq [m]$ of indices such that the flat $\sing_{k_1,\dots,k_s}$
is nonempty, the coefficient vectors $\bb^{(k_1)} , \dots , \bb^{(k_1)}$ 
are linearly independent.  In other words, $\A$ is simple if hyperplanes with common 
points have linearly independent normals.  
\end{defn}

The \Em{stratum} $\mS_S$ corresponding to any $S \in \mA$ is defined as 
the flat $\mV_S$ with all subflats removed,
$$ \mS_S := \mV_S \, \setminus \, \bigcup_{\mV_T \subsetneq \mV_S} 
   \mV_T \, . $$
The collection of strata form a \Em{stratification} of $\mV$, i.e., a 
partition of $\mV$ into disjoint smooth sets, and the flat $\mV_S$ is the 
closure of the stratum $\mS_S$.  The \Em{dimension} of $\mS_S$ 
is the dimension of $\mV_S$ as an algebraic set.  It is also the
dimension of $S$ as a complex manifold and half the dimension of
$S$ as a real manifold.  When $\mV$ is simple then $\mA$ is the 
collection of all subsets of $\{1 , \dots , m\}$ such that the
corresponding hyperplanes have non-empty intersection, and the dimension 
of a stratum $\mS_S$ is simply $d-|S|$.  
\begin{example}
Figure~\ref{fig:arrangement} shows an arrangement with two hyperplanes
and three non-empty flats.

\begin{figure}
\centering
\includegraphics[width=3in]{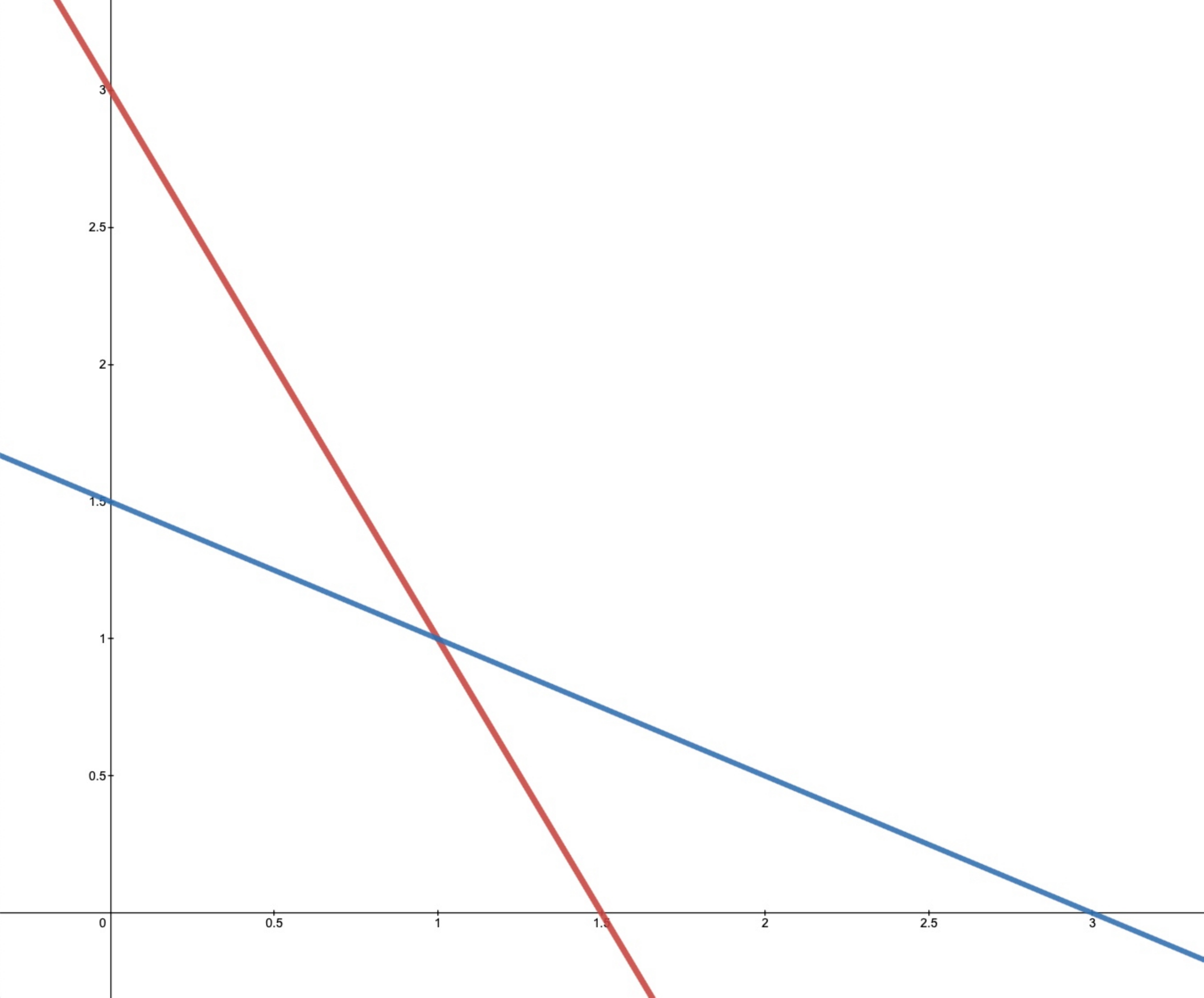}
\caption{The real part of a simple arrangement of two hyperplanes 
in two dimensions}
\label{fig:arrangement}
\end{figure}
\end{example}

\subsubsection*{Critical points}

Let $\rr\in\R_{>0}^d$. By definition, a \Em{critical point} 
of $h=h_{\rr}$ on a stratum
$\mS$ is a point where the differential of the restricted map $h|_\mS$
is zero. The critical points on a stratum satisfy an explicit set of 
algebraic equations, which we now describe.  Let $k_1, \ldots , k_s
\in [m]$ be any indices such that $\{ \bb^{(k_1)} , \ldots , \bb^{(k_s)} \}$
are linearly independent, defining a flat $\sing_S$.  
Because the differential of the unrestricted map 
$h_{\rr}$ is given by its gradient, a 
critical point on $\sing_S$ is characterized by
a rank deficiency in the matrix
\begin{equation} \label{eq:M}
M(\zz) := \begin{pmatrix} -\nabla\ell_{k_1} (\zz) \\ \vdots \\ 
   -\nabla\ell_{k_s} (\zz) \\ -\nabla h \end{pmatrix} 
   = \begin{pmatrix} &\bb^{(k_1)}& \\ &\vdots& \\
&\bb^{(k_s)}& \\ r_1/z_1 & \cdots & r_d/z_d \end{pmatrix} \, .  
\end{equation}
Up to a reordering of variables and the $\ell_{k_j}$, we may assume 
without loss of generality that $M$
contains pivots in its first $s$ diagonal entries.  For $j=s+1,\dots,d$,
let $M_j$ denote the $(s+1) \times (s+1)$ matrix constructed from the
first $s$ columns of $M$ together with its $(s+j)$th column.  
Clearing denominators, the critical points on $\mV_S$ are the 
real solutions of a polynomial system defined by
\begin{equation} \label{eq:critpoints} 
\ell_{k_1} = \cdots = \ell_{k_s} = \det M_1 = \cdots = \det M_{d-s} = 0 \, .
\end{equation}

As expected, the critical points are unchanged by scaling $\rr$,
and thus only depend on the normalized vector $\rhat=\rr/|\rr|$.

\begin{lemma}
If $\bs$ is a critical point of $h_\rr$ then $\bs \in \R^d$.
\end{lemma}

\begin{proof}
Assume, without loss of generality, that $\bs$ lies in the stratum $\mS_{1,\dots,s}$
. Then there exist constants $\lambda_1,\dots,\lambda_s\in\C$
such that
\[ \lambda_1 \bb^{(1)} + \cdots + \lambda_s \bb^{(s)} 
= \left(\frac{r_1}{\sigma_1},\dots,\frac{r_d}{\sigma_d}\right). \]
Write $\bs = \xx + i\yy$ and $\lambda_j=a_j + ic_j$ for 
$\xx,\yy\in\R^d$ and each $a_j,c_j\in\R$. Taking the imaginary 
parts of both sides of the above equation gives
\[ c_1 \bb^{(1)} + \cdots + c_s \bb^{(s)}
= -\left(\frac{r_1y_1}{x_1^2+y_1^2},\dots,\frac{r_dy_d}{x_d^2+y_d^2}\right), \]
and taking the dot product of both sides of this equation with $\yy$ implies
\[ \sum_{j=1}^s c_j \left(\bb^{(j)} \cdot \yy\right) 
= - \sum_{k=1}^d\frac{r_ky_k^2}{x_k^2+y_k^2}. \]
Since $1-(\xx+i\yy)\cdot \bb^{(j)} = 0$ for all $1 \leq j \leq s$, we see that
$\yy\cdot\bb^{(j)}=0$ for all $j$, and thus
\[ 0 = \sum_{k=1}^d\frac{r_ky_k^2}{x_k^2+y_k^2}. \]
Each term on the right-hand side of this equation is non-negative, 
and zero if and only if $y_k$ vanishes, so we must have 
$y_k = 0$ for all $1 \leq k \leq d$, meaning $\bs=\xx\in\R^d$ is real. 
\end{proof}

The function $h$ is continuous and strictly convex on each
orthant of $\R^d$. If $\sing_{S,\R} = \sing_S \cap \R^d$ is the real 
part of a flat and the intersection of 
$\sing_{S,\R}$ with an open orthant $O \subset \R^d$
is bounded, then $\sing_{S,\R} \cap O$
contains a unique critical point of $h$, 
which is necessarily a local minimum by convexity of $h$. 
If $\sing_{S,\R} \cap O$ is unbounded,
which happens in every orthant when $\sing_{S,\R}$ is parallel to
at least one coordinate plane, then the infimum
of $h$ on $\sing_{S,\R} \cap O$ is $-\infty$ and $h$ has
no critical point on $\sing_{S,\R} \cap O$.
Thus, for each index set $S \in \A$ the critical points of the 
height function $h = h_{\rhat}$ on the stratum $\mS_S$ are
real, have non-zero coordinates ($h$ is not defined
when one of its coordinates is zero), and there is at most 
one in each quadrant of $\R^d$. 

We call a critical point of $h_{\br}$ restricted to the flat $\mV_S$ 
a critical point of the stratum $\mS_S$, although it may not lie in 
$\mS_S$ (it could be in a proper subflat).  Each critical point $\bs$
lies in a unique stratum of lowest dimension, which we denote 
$\mS (\bs)$.

\begin{defn}[critical sets and generic directions]
The \Em{critical set} $\Omega$ of $F$ (in the direction $\rhat$) 
is the union of all critical points of $h_{\rhat}$ on the strata $\mS_S$
as $S$ ranges over all index sets $S \in \A$. The direction $\rhat$ 
is said to be \Em{generic} if no critical point of $h_{\rhat}$ is 
a critical point of two distinct flats. 
\end{defn}

\begin{defn}[normal and lognormal cones] 
For any flat $\sing_S$, the \Em{(positive) normal cone} $N(S)$ is the cone
$$N(S) = \left\{ \sum_{j \in S} a_j \bb^{(j)} : a_j > 0 \right\} \subset \R^d \,$$
and for any $\bs\in S \cap \R^d$ 
the \Em{(positive) lognormal cone} $\tilde{N}_\bs(S)$ is the cone
$$\tilde{N}_\bs(S) = \left\{ \sum_{j \in S} a_j \tilde{\bb}_\bs^{(j)} : a_j > 0 \right\} \subset \R^d \,$$
spanned by the \Em{lognormal vectors} $\btil_\bs^{(j)} := 
\left(b^{(j)}_1 \sigma_1 , \ldots , b^{(j)}_d \sigma_d\right)$.
For convenience, we write 
$$N(\bs) := N(S(\bs)) \text{ and } 
\tilde{N}(\bs) := \tilde{N}_\bs(S(\bs)).$$ 
As we will see, the ACSV analysis
depends on identifying singularities $\bs$ with $-(\grad h_{\rhat})(\bs)\in N(\bs)$
or, equivalently, with $\rhat \in \tilde{N}(\bs)$.
\end{defn}

\begin{example} \label{ex:main}
We introduce a running example that clarifies some of the definitions, using $x$
and $y$ in place of $z_1$ and $z_2$ for clarity.  
Let $F(x,y) := 1/(\ell_1(x,y) \cdot\ell_2(x,y))$ where
$$\ell_1 = 1 - \frac{2x+y}{3} \qquad 
   \text{and} \qquad\ell_2 = 1 - \frac{x+2y}{3}. $$
The real part of this hyperplane arrangement is illustrated
in Figure~\ref{fig:hyper}.  The flats are the two lines $\sing_1$ and $\sing_2$, 
and their intersection point $\sing_{1,2}=\{(1,1)\}$.  

\begin{figure}
\centering
\includegraphics[width=0.35\linewidth]{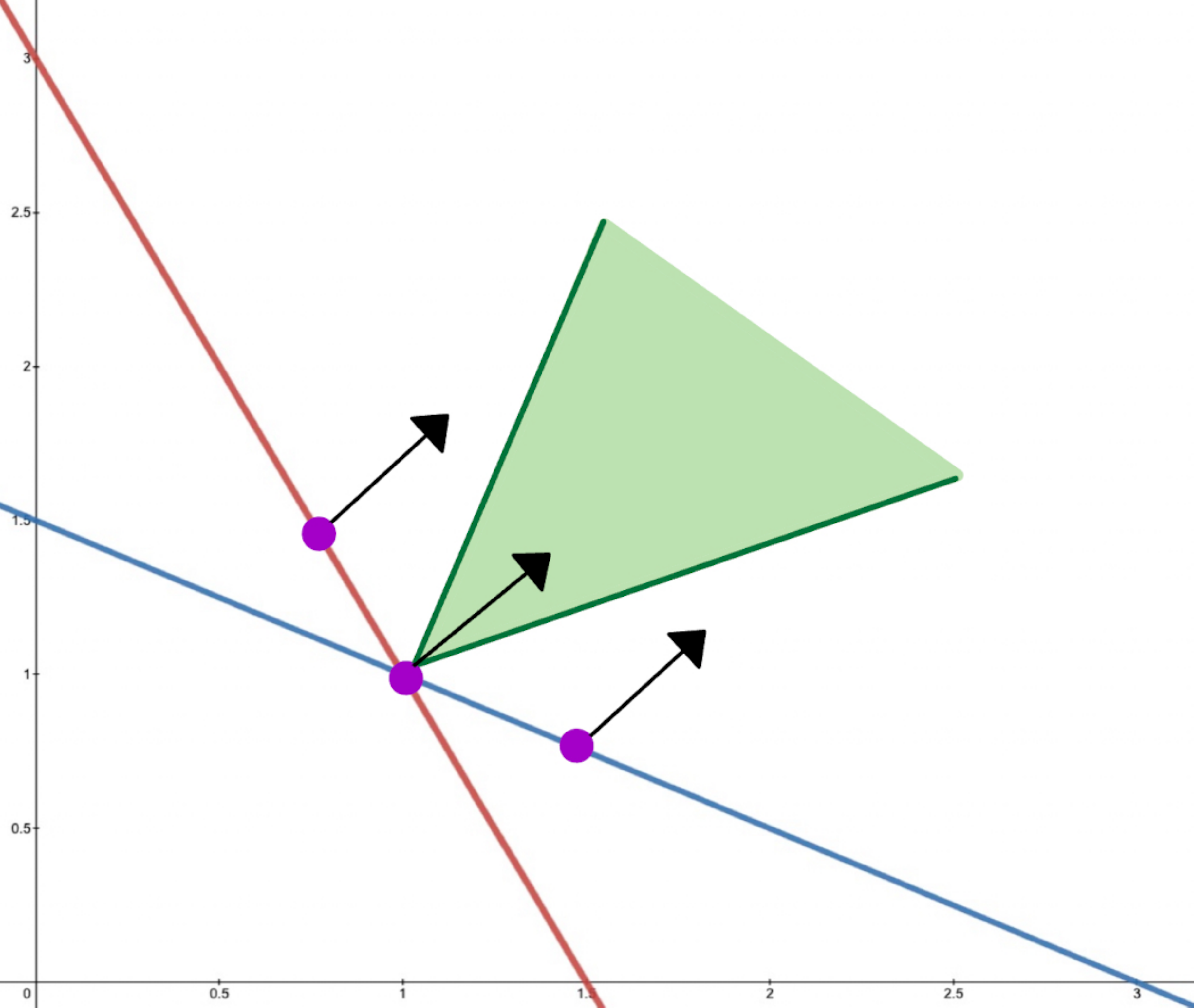} \hspace{0.5in}
\includegraphics[width=0.35\linewidth]{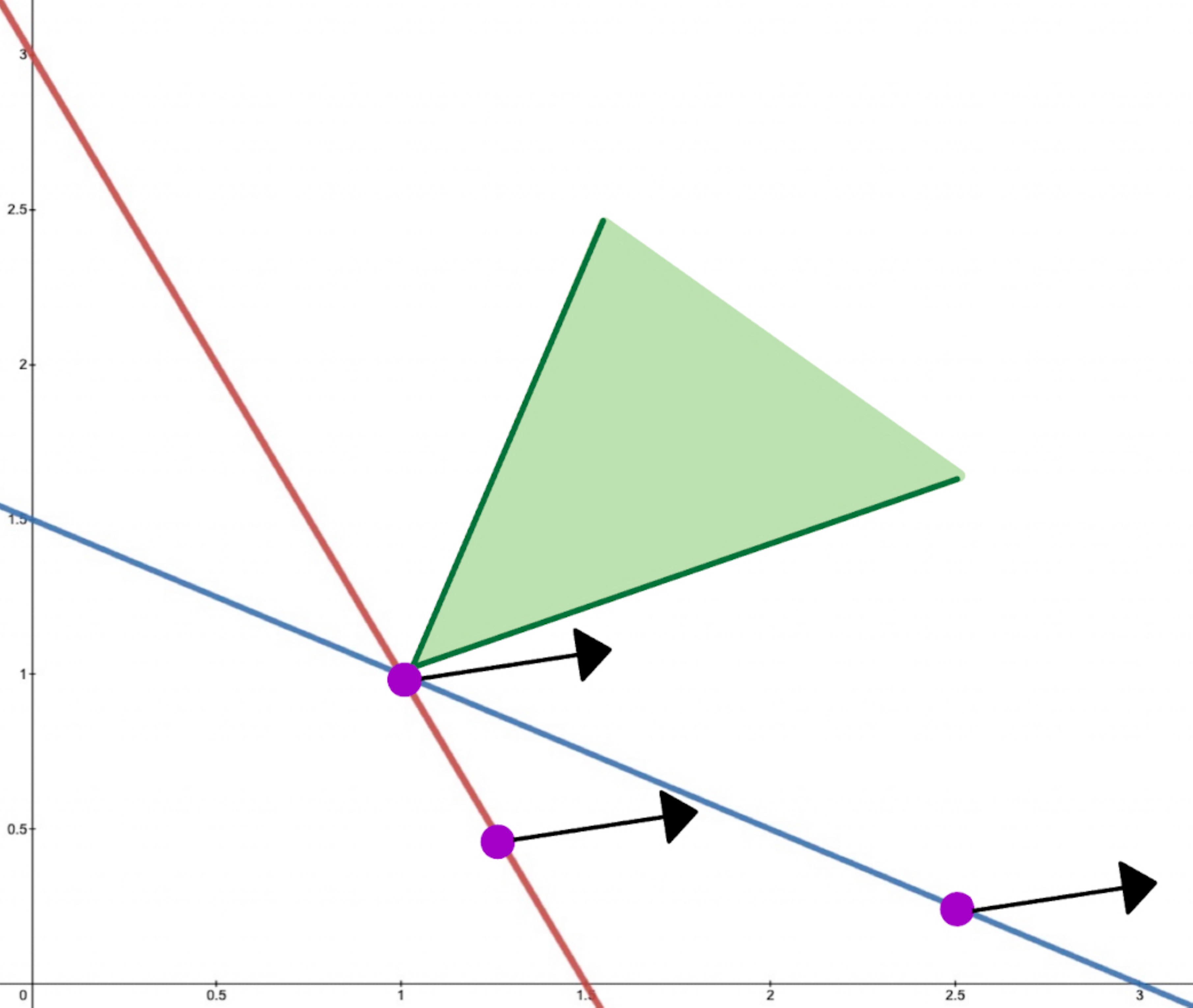}
\caption{The real part of the hyperplane arrangement
${\color{red}\ell_1(x,y)}{\color{blue}\ell_2(x,y)}=0$ in
Example~\protect{\ref{ex:main}}, together with critical points, the
lognormal cone $\tilde{N}(1,1)$ and $\rhat$, when $\rr=(1,1)$ (left) and
$\rhat=(5,1)$ (right).}
\label{fig:hyper}
\end{figure}

For the flat $\sing_1$ the critical point equations in the direction $\rr$ are 
$\ell_1 = \det M_1 = 0$, with the latter expressing that 
the normal $\bb^{(1)}$ of $\ell_1$ is parallel to 
the gradient $(\nabla h_\rr)(x,y) = (r_1 / x,r_2 / y)$. 
This gives the system 
\begin{eqnarray*}
2x + y & = & 3 \\
\frac{y}{r_2} & = & \frac{2x}{r_1} \,,
\end{eqnarray*}
with solution $\bs_1 = \left(\frac{3\hat{r}_1}{2} , 3\hat{r}_2\right)$.  
A similar computation on the flat $\sing_2$ gives the critical point
$\bs_2 = \left(3\hat{r}_1 , \frac{3\hat{r}_2}{2}\right)$, while
the single point $\bs_{1,2}=(1,1)$ in $\sing_{1,2}$ is trivially critical.
We note that $\bs_1$ lies in the stratum $\mS_1$ 
unless it coincides with $\bs_{1,2}$, 
which occurs when $\hat{r}_1=2/3$. Similarly, $\bs_2$ lies in the stratum $\mS_2$
unless $\hat{r}_1=1/3$. Thus, $\rr$ is generic unless it
is a multiple of $(1,2)$ or $(2,1)$.

Figure~\ref{fig:hyper} shows the lognormal cone $\tilde{N}(1,1)$ and the
three critical points when $\rr=(1,1)$ and $\rr=(5,1)$.  
Note that $\hat{\rr}$ lies inside $\tilde{N}(1,1)$ in the left
sub-figure, but does not in the right sub-figure. The non-generic directions
are precisely those where $\rr$ lies on the boundary of $\tilde{N}(1,1)$.
\end{example}

\subsubsection*{Topological definitions}

We now fix $\rhat$, write $h(\zz) := h_{\rhat}(\zz)$, let $\M_{\leq a}$ 
denote the sublevel set $\{ \zz \in \M : h(\zz) \leq a \}$, and set $\low$
to the least critical height, 
$$\low := \min \{ h(\bs) : \bs \in \Omega \} \, .$$
If $a < \low$ and $C$ is a cycle in $\M_{\leq a}$, then 
$C$ can be continuously deformed in $\M$ so that $\max \{ h(\zz) :
\zz \in C \}$ is as small as desired.  This follows from a 
non-proper extension of the fundamental Morse lemma described in~\cite{BMP-morse}
(because there are no so-called \emph{critical
points at infinity} in our case), or more directly from a concrete
modification of the downward gradient flow on $\M$, using a partition
of unity to keep the flow away from all strata of $\sing$.  It follows
that the pairs $H_d (\M , \M_{\leq a})$ with $a < \low$
are naturally isomorphic.  We denote any of these naturally
equivalent pairs by $(\M , -\infty)$ and use the symbol $\doteq$ to
denote equality of classes in $H_d (\M , -\infty)$.  

For those not topologically inclined, a class $[\mT]$ in 
$H_d (\M , -\infty)$ can simply be viewed as a domain of 
integration $\mT$ for the Cauchy integral~\eqref{eq:intoCIF}
which is defined up to points in $\mM$ of arbitrarily small 
height. The idea is that points at small height are asymptotically
negligible, so only the class of $\mT$ matters in our asymptotic 
calculations.

To that end, suppose that $C \doteq 0$, meaning $C$ is a null-homologous chain 
in $H_d (\M , -\infty)$.  For each $a < \low$, let $C_a$ be
a chain in $\M_{\leq a}$ such that $C \doteq C_a$, and let
$|C_a| =(2 \pi)^{-d} \mu (C_a)$ where
$\mu(C_a)$ is the area of $C_a$ defined by the $d$-dimensional Lebesgue measure.  
Because $C_a$ is supported on $\M_a$, the function $F$ is continuous on $C_a$,
hence $|F / \prod_{j=1}^d z_j|$ has some maximum value $f_a$ on $C_a$.  
It follows that for all $\rr$, 
\begin{eqnarray}
\left |(2 \pi i)^{-d} \int_C F(\zz) \zz^{-\rr} \frac{d\zz}{\zz} \right |
   & = & (2 \pi)^{-d} \left | \int_C F(\zz) \zz^{-\rr} \frac{d\zz}{\zz}
   \right | \nonumber \\
& \leq & f_a |\zz|^{-\rr} |C_a| \nonumber \\
& \leq & f_a |C_a| e^{|\rr| a} \label{eq:low}
\end{eqnarray}
because $|\zz|^{-\rr} \leq e^{|\rr| a}$ on $\M_{\leq a}$.
This yields the following result.

\begin{pr} \label{pr:low}
Suppose $C \doteq 0$.  Then the Cauchy integral integrated over $C$
is smaller than any exponential.  Specifically, for any $a < 0$ there
exists $K>0$ such that
$$\left |\int_C F(\zz) \zz^{-\rr} \frac{d\zz}{\zz} \right |
   \leq K e^{a |\rr|} \, .$$
Furthermore, if $F$ grows at most polynomially then
$$\int_C F(\zz) \zz^{-n\brhat} \frac{d\zz}{\zz} = 0 \, $$
for all $n$ larger than some fixed natural number.
\end{pr}

\begin{proof}
The first part is an immediate consequence 
of~\eqref{eq:low}.  The second follows by fixing $\rr=n\brhat$ and 
letting $a \to -\infty$.  Both $f_a$ and $C_a$ grow at most 
exponentially in $|a|$ (because $f_a$ grows at most polynomially
in $\zz$), hence for $|\rr|$ greater than the sum
of these, the right-hand side of~\eqref{eq:low} goes to zero.
As these are all bounds for the same fixed integral, its
value must be zero.
\end{proof}

We close this section with two more topological definitions that
will appear in the statements of results.  These are illustrated
in Figure~\ref{fig:fibers}.

\begin{defn}[imaginary fibers] 
Let $\xx\in\R_*^d$ be any point not in the real hyperplane arrangement $\sing_\R$.
The \Em{imaginary fiber} $\CC_\xx$ is the chain 
$$\CC_\xx = \xx + i \R^d = \{\xx+i\yy:\yy\in\R^d\},$$
oriented via the standard orientation on $\R^d$.
Because the linear functions $\ell_j$ defining $\A$ have real coefficients,
$\ell_j(\xx + i \yy) = 0$ only if $\ell_j(\xx) = 0 = \ell_j (\yy)$, so
the entire fiber $\CC_\xx$ is disjoint from $\sing$.
The point $\xx$ is called the \Em{basepoint} of the fiber $\CC_\xx$.
If $B$ is a connected component of the \Em{real domain of holomorphicity} 
$\mMR := \R_*^d \setminus \sing_\R$ then the homotopy
$$t \xx + (1-t) \xx' + i \R^d \text{ for } t \in[0,1]$$ 
defines a homotopy between $\CC_\xx$ and $\CC_{\xx'}$ 
for any $\xx, \xx' \in B$. Thus, the homology class of $\CC_\xx$ 
in $H_d (\M , -\infty)$ depends only on the component $B$
of $\mMR$ containing $\xx$, and we denote this homology
class by $\CC_B$.
\end{defn}

\begin{rem}
\label{rem:unbounded}
The height function $h_{\rhat}$ is unbounded from below on any unbounded
component of $\mMR$. Thus, $\CC_\xx \doteq 0$ whenever $\xx$ lies in an
unbounded component of $\mMR$ since we can move $\xx$ within $B$ 
until its height is less than $\low$, at which point the whole 
fiber $\CC_\xx$ lies in $\M_{ < \low}$.  The fact that imaginary 
fibers of components on which $h_{\rhat}$ is bounded from below
form a basis of $H_d (\M , -\infty)$ was proved in~\cite[page~268]{VaGe1987}.
\end{rem}

\begin{figure}
\centering
\includegraphics[width=0.3\linewidth]{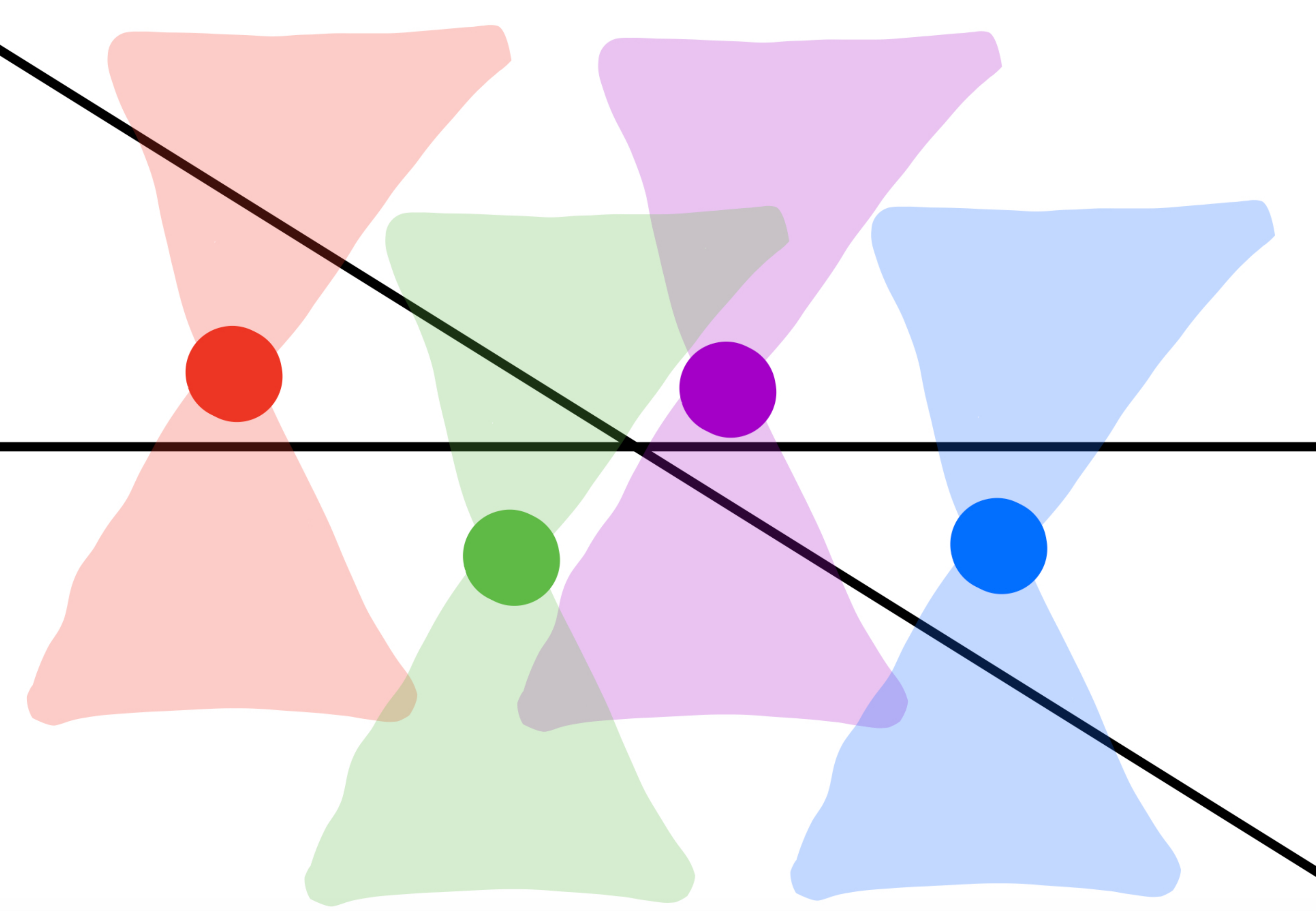}
\caption{An oblique rendering of 3-space is used to depict $\C^2$,
where the $z$-axis denotes imaginary directions in both the $x$ and $y$ 
coordinates.  We take $d=2$ and consider an arrangement of two intersecting
hyperplanes.  Each colored set represents an imaginary fiber over a 
real basepoint (large dot).  The alternating sum of these is the 
linking torus for the point $\bs$ at the intersection of the two 
hyperplanes whose real parts are the lines in the figure.}
\label{fig:fibers}
\end{figure}

\begin{defn}[signs and linking tori]
Let $\A$ be a simple arrangement, $\rhat$ be fixed, and $\bs \in \Omega$.
Suppose $\bs \in \partial B$ for some component $B$ of the real complement 
$\mMR = \R_*^d \setminus \sing_\R$ of the arrangement.  
Let $s$ be the codimension of the stratum containing $\bs$ and
let $k_1, \ldots , k_s$ be the indices such that the stratum $\mS_{\bs}$
is defined by $\ell_{k_1} = \cdots = \ell_{k_s} = 0$.  For each of
the $2^k$ components $B$ of $\mMR$ with $\bs \in \overline{B}$, 
define the \Em{sign} of $B$ with respect to $\bs$ by
$$\sgn_{\bs} (B) := \sgn \left ( \prod_{i=1}^s \ell_{k_i} (\xx) \right ),$$
where $\xx$ is any point of $B$.  Thus, the sign is positive on the 
component not separated from the origin by any of the $s$ hyperplanes, 
and alternates across each hyperplane.  We also define the \Em{absolute 
sign} of $B$ by $\sgn (B) = \sgn (\prod_{j=1}^d x_j)$ for any $\xx \in B$.
In other words, the sign is $+1$ if $B$ is in an even orthant and $-1$ 
otherwise.  The \Em{linking torus} of $\bs$ is
$$\torus_{\bs} := \sum_B \sgn (B) \sgn_{\bs} (B) \CC_B ,$$
where the sum is over all $B$ with $\bs \in \overline{B}$.
\end{defn}

\begin{example}[linking tori for Example~\ref{ex:main}]
Continuing Example~\ref{ex:main}, recall that the stratum $\mS_{1,2}$
is the singleton point $\bs_{1,2} := (1,1)$.  This is a stratum of
codimension two, and the linking torus of $\bs_{1,2}$ is
$$\tau_{12} := \tau_{\bs_{1,2}} = \CC_\ba - \CC_\bb + \CC_\cc - \CC_\dd,$$
where $\ba, \bb, \cc,$ and $\dd$ are points near $\bs$ in the four
regions carved out by removing the two real lines $\ell_1$ and
$\ell_2$ from $\R^2$, enumerated in counterclockwise order
with $\ba$ in the region whose closure contains the origin
(see Figure~\ref{fig:linking}).

\begin{figure} 
\centering
\includegraphics[width=4in]{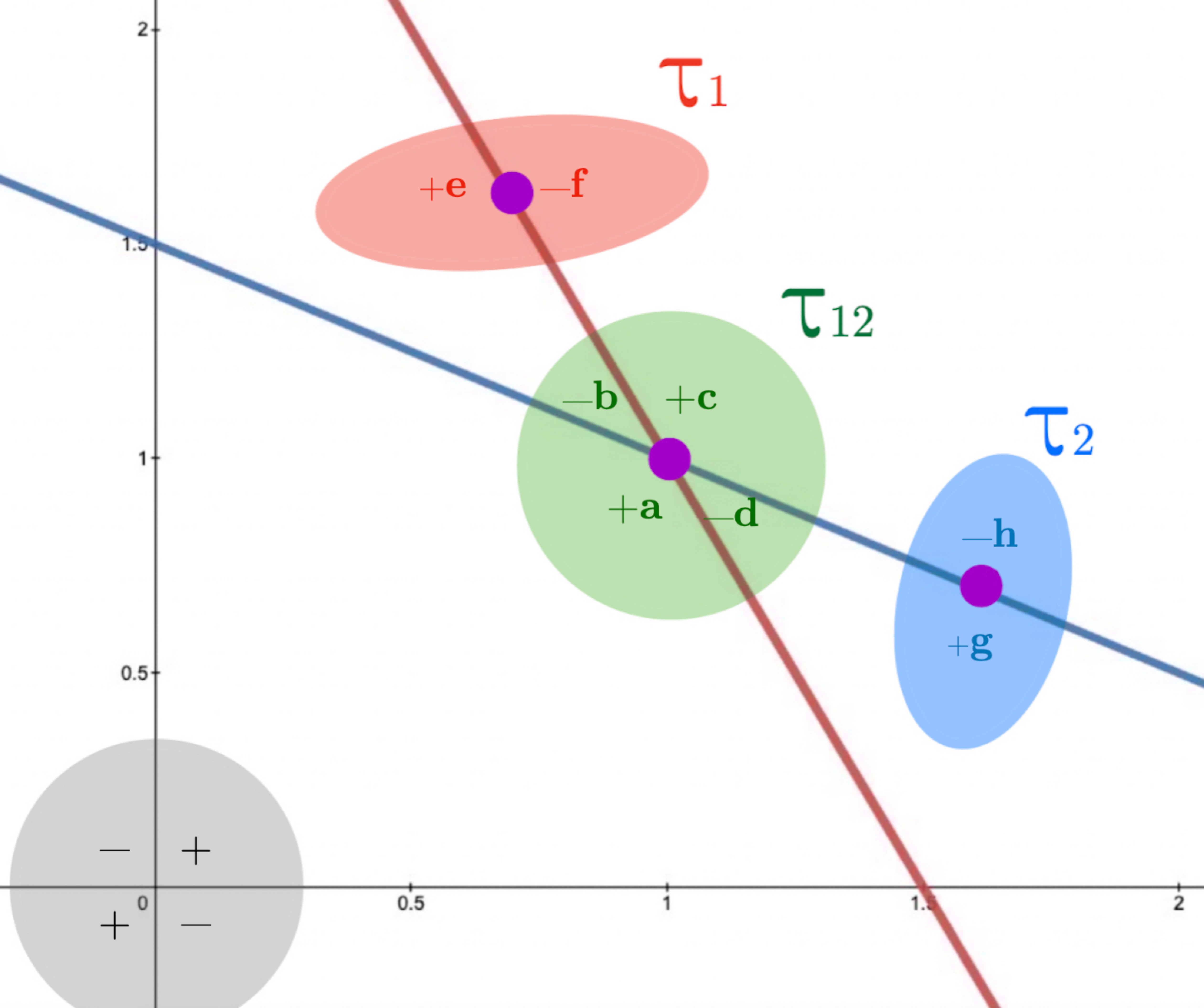}
\caption{Linking tori for Example~\protect{\ref{ex:main}}, with the
signs associated to each basepoint pictured. A `linking torus' around the
origin with the shown signs will be used to decompose the Cauchy 
integral representing generating function coefficients.}
\label{fig:linking}
\end{figure}

Suppose $\rr=(1,1)$, so we are in the situation pictured in 
Figure~\ref{fig:linking}. The critical point $\bs_1$ lies on the
codimension one stratum $\mS_1$ and has linking torus
$$\tau_{1} := \tau_{\bs_{1}} = \CC_\be - \CC_\bff,$$
where $e$ lies in the halfspace defined by $\ell_1$ whose
closure contains the origin. Similarly,
the critical point $\bs_2$ on $\mS_2$ has linking torus
$$\tau_{2} := \tau_{\bs_{2}} = \CC_\bg - \CC_\hh,$$
where $g$ lies in the halfspace defined by $\ell_2$ whose
closure contains the origin. 
Note that many of the basepoints in these linking tori lie in the
same components of $\mMR$. In particular,
$$ \CC_\bb = \CC_\be, \quad \CC_\dd = \CC_\bg, 
\;\; \text{ and } \;\; \CC_\cc = \CC_f = \CC_g, $$
so the sum of the linking tori around the critical points is
$$ \tau_1 + \tau_2 + \tau_{1,2} = \mC_\ba - \mC_\cc. $$
In fact, since $\cc$ lies in an unbounded component of $\mMR$
we have $\tau_1 + \tau_2 + \tau_{1,2} \doteq \mC_\ba$, which will 
come up again later in our analysis.
\end{example}

Linking tori appear, in a more general form, in stratified 
Morse theory.  The work of Goresky and MacPherson~\cite[Chapter~10]{GM}
shows how the topology of $\M$ is generated 
by certain \emph{attachment cycles} near $\bs$ as $\bs$ varies over the critical 
points of each stratum of $\sing$.  The cycle near a codimension
$s$ critical point $\bs$ has the structure of the product of
the {\em normal link}, a local cycle which in this case is a 
small $s$-torus, with the {\em tangential cycle},
a $(d-s)$-disk modulo its boundary (specifically the 
$(d-s)$-dimensional downward gradient flow subspace 
of $h$ on $\sing$ near $\bs$).
This toral tube intersects $\R^d$ in $2^s$ points, and is in
fact equal to $\torus_{\bs}$ in $H_d (\M , -\infty)$.

\subsection{Results}

We split our results into increasing complicated cases.

\subsubsection*{Simple arrangements in generic directions}

All of our topological arguments can already be seen
in the case of simple arrangements and generic directions.  
First, we identify which collection of critical points 
will be important in the analysis,
which is, in the general (non-linear) case, a very hard and 
perhaps undecidable step.

\begin{defn}[contributing points]
Assume $\A$ is simple and fix $\rhat \in \R_{>0}^d$.  
We call $\bs$ a \Em{contributing point} if $- \grad h_{\rhat} (\bs) \in N(\bs)$
and let $\contrib \subseteq \Omega$ denote the set of 
contributing points.
\end{defn}
\begin{unremarks}
$(i)$ The quantity $-\grad h_{\rhat} (\bs)$ is given by 
$(\hat{r}_1 / \sigma_1 , \ldots , \hat{r}_d / \sigma_d)$.
If this is in the positive cone $N(\bs)$ then it is in
linear space spanned by normals to hyperplanes containing $\bs$, so
every contributing point is a critical point.
$(ii)$ Suppose $- \grad h_{\rhat} (\bs)$ is on the relative
boundary of $N(\bs)$ (e.g., is in the span of 
some proper subset of the vectors $\{ \bb^{(k)} : k \in S(\bs) \}$).
Then $\bs$ is a critical point for the larger flat defined by
this proper subset of hyperplanes and $\rhat$ is not generic.
In fact, non-genericity is equivalent to the existence of a critical
point $\bs \in \Omega$ with $- \grad h_{\rhat} (\bs) \in \partial N(\bs)$.
\end{unremarks}
The heart of this paper, proven in Section~\ref{sec:simplegeneric}, 
is the following result.
\begin{theorem} \label{thm:topol} 
If $\A$ is simple and $\rhat$ is generic then the domain of integration
$\mT$ in the Cauchy integral representation~\eqref{eq:intoCIF} satisfies
$$\mT \doteq \sum_{\bs \in \contrib} \torus_{\bs} \, . $$
\end{theorem}

Theorem~\ref{thm:topol} gives a determination of the coefficients
$\{ k_{\bs} \}$ in Equation~\eqref{eq:homological}, namely $k_{\bs} = 1$
when $\bs \in \contrib$ and zero otherwise.  Is this 
expected or surprising?  If the torus $\mT$ may be expanded
up to the point $\bs$ while remaining in $\mT$, then the
condition $\bs \notin \contrib$ is precisely the condition
needed to extend the deformation to bypass $\bs$ and show
that $\mT \doteq C$ for some $C$ in $\M_{< h(\bs)}$.  
Theorem~\ref{thm:topol} may be understood to say 
that the converse holds and, moreover, it holds even 
if one has to pass through $\sing$ at various points 
en route to $\bs$.

One might conjecture that Theorem~\ref{thm:topol} should hold 
for general rational functions, but this is false.  A counterexample is given
in~\cite{BMP-lacuna}, where for certain critical points $\bs$ 
the coefficient $k_{\bs}$ may be seen not only 
to be nonzero but to take on the surprising value of~3.
In other words, when $\mT$ passes through $\sing$ en route to $\bs$,
it twists so as to hit $\sing$ near $\bs$ in an unpredictable manner.  
Hyperplane arrangements are special in that they do not 
allow these kinds of shenanigans.  We will give more intuition 
as to why when we discuss the relation between real hyperplane
arrangements and their complexifications in 
Section~\ref{ss:ProofofContribSum}.

To complete the analysis for simple arrangements in generic directions,
saddle point integration results are quoted from~\cite{PW-book},
yielding in Theorem~\ref{thm:mainASM} from Section~\ref{ss:integrals}
an asymptotic expansion of $[\bz^{n\br}]F(\bz)$ 
in terms of effective constants which 
can be determined an automatic manner.  In many cases the leading 
asymptotic term is given explicitly.
Algorithm~\ref{alg:1} summarizes the steps 
needed to automatically produce these formulae.  Details on our Maple 
implementation of Algorithm~\ref{alg:1} and additional examples
are given in Section~\ref{sec:implement}.  The main wrinkle
in determining dominant asymptotics from the explicit formulae is the
possibility that leading constants in Theorem~\ref{thm:mainASM} may vanish
due to cancellation, which is discussed in Section~\ref{ss:vanish}.

\begin{algorithm}
\DontPrintSemicolon
\KwIn{Rational function $F(\bz) = G(\bz)/\prod_{j=1}^m\ell_j(\bz)^{p_j}$ 
with $\ell_j(\bz) = 1 - \left(\bb^{(j)},\bz\right)$ linearly independent \\ 
\hspace{0.45in} Generic direction $\br = (r_1,\dots,r_d)$}
\KwOut{Asymptotic expansion of $[z^{n\br}]F(\bz)$ as 
$n \rightarrow \infty$, uniform in compact set around $\br$}
\;
\If{matrix with rows $\bb^{(j)}$ not full rank 
for a collection with common solution}{
	return \textsc{fail}, ``not a simple arrangement''
}
get the flats defined by the $\ell_j(\bz)$\\
find the (unique) critical point in each orthant for each flat by solving 
Equation~\eqref{eq:critpoints} \\
sort the saddles by height function $h_{\br}$ and set $\bs$ equal to the 
highest critical point \\
\While{no contributing points have been found}{
	compute the cone $N(\bs)$ \\
	\If{$(- \nabla h)(\bs)$ on boundary of $N(\bs)$}{
		return \textsc{fail}, ``$\br$ not a generic direction''
	}
	\If{$(- \nabla h)(\bs) \notin N(\bs)$}{
		set $\bs$ to next lowest height critical point and repeat 
loop
	}
	Compute the ideal $\mJ(\bs)$ in Equation~\eqref{eq:Jideal} \\
	\uIf{ $G \in \mJ(\bs)$}{
		set $\bs$ to next lowest height critical point and repeat 
loop} 
	\Else{
		$\bs$ is a contributing singularity \\
		repeat above steps to identify contributing singularities of 
the same height as $\bs$}
}
calculate asymptotic $\Phi_{\bs}$ at each contributing singularity 
$\bs$ of highest height using Theorem~\ref{thm:mainASM} \\
\textbf{return} sum of the $\Phi_{\bs}$ from the last step
\caption{{\sf SimpleHyperplaneAsymptotics}}
\label{alg:1}
\end{algorithm}

\subsubsection*{Non-simple arrangements}

Section~\ref{sec:nonsimple-generic} deals with non-simple $F(\bz)$ in 
generic directions $\br$ using a partial fraction
decomposition.  When $F(\zz) = G(\zz) / \prod_{j=1}^m \ell_j (\zz)^{p_j}$,
with linear dependencies among the normals $\bb^{(j)}$ to $\ell_j$, 
it is possible to rewrite $F$ as $\sum_{i=1}^s F_i$ where each $F_i$
is of this same form except that the vectors $\bb^{(j)}$ are linearly 
independent.  This is a known fact from the theory of matroids:
the {\bf broken circuit} complex defined in~\cite{brylawski-NBC} leads 
to the canonical {\bf no broken circuit} basis~\cite{OrlikTerao1992} 
for the matroid $\{ \bb^{(j)} : 1 \leq j \leq m \}$,
which may be harnessed for easy computation of the partial fraction
decomposition $F = \sum_{i=1}^s F_i$.  The main result of
Section~\ref{sec:nonsimple-generic} is the following
algorithm.

\begin{algorithm}[!ht]
\DontPrintSemicolon
\KwIn{Rational function $F(\bz) = G(\bz)/\prod_{j=1}^m\ell_j(\bz)^{p_j}$}
\KwOut{Collection $F_1(\bz),\dots,F_r(\bz)$ of rational functions with 
$\sup(F_j) \subset \sup(F)$ such that $F = F_1 + \cdots + F_r$ and each $F_j$ 
has independent support}
\;
$S \gets F(\bz)$ \\
\While{there exists a summand $\tilde{F}$ of $S$ with broken circuit 
$\{i_1,\dots,i_s\}$}{
	apply Equation~\eqref{eq:linDependence} with $i_j>i_s$ such that 
$\{i_1,\dots,i_s,i_j\}$ is dependent \\
	add the result to $S - \tilde{F}$
}
collect and simplify summands in $S$ with the same denominator \\
\While{there exists summand $\tilde{F}$ of $S$ with numerator in the 
ideal generated by $\sup(\tilde{F})$}{
		pick among all such $\tilde{F}$ one with maximal total degree 
in its denominator \\
		write the numerator as a polynomial combination of the 
elements in its denominator \\
		expand this new fraction and simplify, removing terms in the 
denominator \\
		collect and simplify summands with the same denominator
}
\caption{{\sf SimpleDecomp}}
\label{alg:SimpleDecomp}
\end{algorithm}

\subsubsection*{Non-generic directions}

Section~\ref{sec:simple-nongeneric} deals with non-generic directions 
$\br$ which, beyond some two-dimensional work of Lladser~\cite{lladser2006}, 
has not received much previous attention in the literature.  Due to the results 
of Section~\ref{sec:nonsimple-generic}, when addressing non-generic
directions one may assume without loss of generality that $F$ is 
simple.  

Nongeneric directions arise when the domains of dominant singularities
from different strata collide.  If $\mS$ is a stratum of codimension $k$,
then each $\bs \in \mS$ lies in the set $\crit (\rr)$ for a set of vectors
$\rr$ of dimension $k$ (projective dimension $k-1$).  As $\bs$ varies
over $\mS$, these sets foliate an open cone $K$ in $\R^d$. Typically, 
as $\rr$ passes through $\partial K$ there will
be a point $\bs (\rr) \in \crit (\rr)$ such that $\bs(\rr) \in \contrib$
on one side of $\partial K$ while $\bs(\rr) \notin \contrib$ 
on the other side.  Precisely when $\rr \in \partial K$, 
it is not defined whether the critical point $\bs (\rr)$ is in
$\contrib$ because this is not defined for non-generic directions
(as mentioned above, $- \grad h_{\rhat} (\bs)$ lies on the boundary of the 
positive cone $N(\bs)$).  

The word \emph{typically} above refers to the fact that these cones
themselves are curved versions of central hyperplane arrangements.
Their boundaries are mostly smooth except when two or more bounding 
surfaces intersect, corresponding to the stratum of $\bs$ dropping
in dimension.  Our most explicit results are limited to
the most common case, when there is a drop of~1 in dimension.  In this sense
there is still more work to be done to obtain a complete understanding of
non-generic directions.

We give two types of result. The first is the behavior precisely 
along a non-generic direction when the critical points of two 
strata, one of which is precisely one dimension less than the other, 
overlap. In this setting, we consider
$a_{n \rhat}$ as $n \to \infty$ when $\rhat$ is a fixed non-generic
rational direction; our results may be extended to $\rr = n \rhat + O(1)$,
though we do not pursue that here.  
Proposition~\ref{prop:nonGenErf} in Section~\ref{sec:simple-nongeneric}
gives the leading asymptotic behavior in terms of certain negative 
moments of Gaussian distributions, computed in Proposition~\ref{prop:1DNegGauss}. 
The second type of result concerns the behavior of $a_\rr$ as
$\rr$ crosses $\partial K$.  Patching from one formula to the
other occurs in a scaling window of width $\sqrt{|\rr|}$, with
our most explicit result occurs when all divisors have power~1.

\subsection{Applications}

Before going into details we list a few applications of our results. 

\begin{example}[\protect{Example~\ref{eg:queuing} continued}]
\label{eg:queuing2}
The simplest case of~\eqref{eq:queuing} is the two-server case, 
where~\eqref{eq:queuing} becomes 
\begin{equation} \label{eq:two servers}
F(x,y) = \frac{e^{x+y}}{(1-\rho_{11}x-\rho_{12}y)(1-\rho_{21}x-\rho_{22}y)} 
\end{equation}
By scaling the $x$ and $y$ variables, asymptotics in all cases
of interest follow from a study of
$$H(x,y) = \frac{e^{x/\rho_{21}+y/\rho_{22}}}{(1-ax-by)(1-x-y)}$$
where $a=\rho_{11}/\rho_{21}>0$ and $b=\rho_{12}/\rho_{22}>0$.
By symmetry we may assume $a\geq b$; if $a=b$ then $H$ is 
not simple, so we postpone this case until the end of the example.
The asymptotic expansion in Theorem~\ref{thm:mainASM} gives 
coefficient asymptotics for $H$, as follows.

Parametrize directions by taking $\rr = (r,1)$ with $r>0$. 
The flat defined by $1-ax-by=0$ has the critical point
\[ c_{1} = \left(\frac{r}{(r+1)a},\frac{1}{(r+1)b}\right), \]
which is always contributing. Similarly, the flat defined by 
$1-x-y=0$ has the critical point
\[ c_{2} = \left(\frac{r}{r+1},\frac{1}{r+1}\right), \]
which is always contributing. The flat defined by 
$1-ax-by=1-x-y=0$ contains a single point
\[ c_{12} = \left( \frac{1-b}{a-b},\frac{a-1}{a-b} \right), \]
which is always a critical point. The point $c_{12}$ is contributing, and
$(r,1)$ is a generic direction, if and only if there exist $A_1,A_2>0$
such that
\[ (-\nabla h)(c_{12}) =  \left( \frac{r(a-b)}{1-b},\frac{a-b}{a-1} \right) = A_1(a,b) + A_2(1,1); \]
since we assume $a \geq b$, this occurs precisely when $0<b<1<a$ and
\[ \frac{1-b}{a-1} < r < \frac{(b-a)a}{(a-1)b}. \]
The points $c_1,c_2,$ and $c_{12}$ have
corresponding heights
\begin{align*} 
h_1 &= -r\log\left(\frac{r}{r+1}\right) - \log\left(\frac{1}{r+1}\right)+\log(a^rb) \\
h_2 &= -r\log\left(\frac{r}{r+1}\right) - \log\left(\frac{1}{r+1}\right) \\
h_{12} &= -r\log\left|\frac{1-b}{a-b}\right| - \log\left|\frac{a-1}{a-b}\right|.
\end{align*}
The theory developed below shows that $h_{12}$ is the highest height contributing
point whenever $c_{12}$ is contributing. Otherwise, $h_1$ has the highest height 
among the contributing points if and only if $a^rb>1$. Theorem~\ref{thm:mainASM}
determines asymptotics for each of the
three different asymptotic regimes:  if $0<b<1<a$ and 
$\frac{1-b}{a-1} < r < \frac{(b-a)a}{(a-1)b}$ then
\[ [x^{rn}y^n]H(x,y) \sim  \left(\frac{a-b}{1-b}\right)^{rn}\left(\frac{a-b}{a-1}\right)^n 
\frac{\exp\left[\frac{(1-b)\rho_{22}+(a-1)\rho_{21}}{(a-b)\rho_{21}\rho_{22}}\right]
(1-b)(a-1)}{a-b}.\]
If those conditions do not hold and $a^rb>1$ then 
\[ [x^{rn}y^n]H(x,y) \sim  \left(a+\frac{a}{r}\right)^{rn}\left(b+br\right)^n 
\frac{\exp\left[\frac{a\rho_{21}+rb\rho_{22}}{ab(r+1)\rho_{21}\rho_{22}}\right]
ab(r+1)^{3/2}}{((a-1)rb+a(b-1))\sqrt{2n\pi r}} ,\]
while if those conditions do not hold and $a^rb<1$ then
\[ [x^{rn}y^n]H(x,y) \sim  \left(1+\frac{1}{r}\right)^{rn}\left(1+r\right)^n 
\frac{\exp\left[\frac{\rho_{21}+r\rho_{22}}{(r+1)\rho_{21}\rho_{22}}\right]
(r+1)^{3/2}}{(1-b+r-ar)\sqrt{2n\pi r}} .\]
If $a=b\neq1$ then we are in the non-simple case, and 
Algorithm~\ref{alg:SimpleDecomp} provides the decomposition
\[ H(x,y) = \frac{e^{x/\rho_{21}+y/\rho_{22}}}{(1-ax-ay)(1-x-y)} 
= \frac{a}{a-1}\cdot\frac{e^{x/\rho_{21}+y/\rho_{22}}}{1-ax-ay}-
\frac{1}{a-1}\cdot\frac{e^{x/\rho_{21}+y/\rho_{22}}}{1-x-y}. \]
Thus, as coefficient extraction distributes over a sum,
{\small
\[ [x^{rn}y^n]H(x,y) \sim 
\left(a+\frac{a}{r}\right)^{rn}\left(a+ar\right)^n 
\frac{\exp\left[\frac{\rho_{21}+r\rho_{22}}{a(r+1)\rho_{21}\rho_{22}}\right]
a\sqrt{r+1}}{(a-1)\sqrt{2n\pi r}}
+
\left(1+\frac{1}{r}\right)^{rn}\left(1+r\right)^n 
\frac{\exp\left[\frac{\rho_{21}+r\rho_{22}}{(r+1)\rho_{21}\rho_{22}}\right]
\sqrt{r+1}}{(a-1)\sqrt{2n\pi r}}.
\]
}

\noindent
Finally, if $a=b=1$ then $H(x,y)=e^{x/\rho_{21}+y/\rho_{22}}/(1-x-y)^2$
is simple but has a second order pole. Theorem~\ref{thm:mainASM} then implies
\[ 
[x^{rn}y^n]H(x,y) \sim \left(1+\frac{1}{r}\right)^{rn}\left(1+r\right)^n
\sqrt{n} \; \frac{\exp\left[\frac{\rho_{21}+r\rho_{22}}
{(r+1)\rho_{21}\rho_{22}}\right](r+1)^{3/2}}{\sqrt{2\pi r}}.
\]
\end{example}

\begin{example}[{Pemantle and Wilson~\cite[Ex.~12.1.3]{PW-book}}]
\label{eg:probs}
Fix $r,s \in \N$. A sequence of $r+s$ independent biased coin 
flips occurs, with the first $n$ flips coming up heads with probability 
$2/3$ and the remaining $r+s-n$ flips coming up heads with probability 
$1/3$. A player wants to see $r$ heads and $s$ tails, and may pick $n$. 
On average, how many choices of $n$ will be winning choices?

By independence of the coin flips, the probability that $n$ is a winning 
choice equals
\[ p_{r,s}(n) = \sum_{0 \leq a \leq n} \binom{n}{a}(2/3)^a(1/3)^{n-
a}\binom{r+s-n}{r-a}(1/3)^{r-a}(2/3)^{s-(n-a)}, \]
so $f_{r,s}=\sum_{n \geq 0}p_{r,s}(n)$ counts the expected number of 
winning values of $n$. It follows that the bivariate generating function 
$F(x,y)=\sum_{r,s \geq 0}f_{r,s}x^ry^s$ is a product of rational 
functions with linear denominators,
\[ F(x,y) = \frac{1}{(1-2x/3-y/3)(1-x/3-2y/3)}. \]
Theorem~\ref{thm:mainASM} then immediately implies that,
for any $\epsilon>0$,
$$ f_{r,s} \sim 
\begin{cases}
   \left(\frac{2(r+s)}{3s}\right)^{sn}\left(\frac{r+s}{3r}\right)^{rn} \, 
   n^{-1/2} \, \frac{(r+s)^{3/2}\sqrt{2}}{\sqrt{rs\pi}(s-2r)} &: r<s/2-
   \epsilon \\[+3mm]
   \hfill 3 &: s/2+\epsilon<r<2s-\epsilon \\[+3mm]
   \left(\frac{2(r+s)}{3r}\right)^{rn}\left(\frac{r+s}{3s}\right)^{sn} \, 
   n^{-1/2} \, \frac{(r+s)^{3/2}\sqrt{2}}{\sqrt{rs\pi}(r-2s)} &: 
   2s+\epsilon < r
\end{cases} $$
as $(r,s)\rightarrow\infty$.
In particular, the expected number of winning values of $n$ exponentially 
decreases in the first and last cases, but approaches the constant 3 when 
$s/2+\epsilon<r<2s-\epsilon$.  Moreover, in this case the leading term 
is not an asymptotic series but the single term~3, as discussed in
Remark~\ref{rem:highpoles} (see also~\cite{Pemantle2000}).  It follows 
that when $(r,s) \to \infty$ with $s/2 + \epsilon < r <2s - \epsilon$, 
the error term $|3 - f_{r,s}|$ decreases exponentially in $s$.

This work allows, for the first time, detailed information at 
slopes $1/2$ and~2 and how behaviour transitions around such directions. 
Proposition~\ref{prop:nonGenDir} implies that both 
$f_{2s,s}$ and $f_{s,2s}$ approach $3/2$ as $s \to \infty$.
More precisely, Proposition~\ref{prop:nonGenErf} implies that
the transition, in which dominant asymptotics go from the constant
3 through the constant $3/2$ to end up exponentially small,
occurs in a scaling window of width $s^{1/2}$:
$$f_{2s + t \sqrt{s} , s} \sim \frac{3}{2} \left(\erf(t)+1\right),$$
where $\erf (z) = \frac{2}{\pi} \int_{0}^z e^{-u^2/2} \, du$
is the Gaussian error function. 
\end{example}

\section{Simple Arrangements in Generic Directions}
\label{sec:simplegeneric}

Unless otherwise stated, in this section we assume that
$$ F(\bz) = \frac{G(\bz)}{\prod_{j=1}^m \ell_j^{p_j}(\bz)} 
\mbox{ is simple and } \rhat \mbox{ is generic.} \eqno{(H1)}$$
We begin by describing an iterative \emph{slide and replace} approach 
to finding critical points which affect dominant asymptotics. 
This iterative algorithm, while correctly determining asymptotics, 
becomes unnecessary once we prove Theorem~\ref{thm:topol}. 
Our topological results allow one to shortcut the slide and
replace algorithm, since we know a priori that the 
easily determined contributing points will be those critical points
selected by the algorithm.

\subsection{Slide and replace}

Let $\mB$ denote the collection of \emph{bounded} components of
$\mMR$.  For $\bs \in \crit$ denote by 
$\adj (\bs) \subseteq \mB$ the collection $\{ B : \bs \in 
\overline{B} \}$ of bounded components of $\mMR$ adjacent 
to $\bs$.  For $B \in \mB$, we recall that the 
imaginary fiber $\CC_B$ is well defined in $H_d (\M , -\infty)$ 
and independent of the particular basepoint $\xx \in B$.
Let $\bU$ denote the set unbounded components of $\mMR$, so 
$\mC_B\doteq0$ for any $B \in \bU$ by Remark~\ref{rem:unbounded}.
In any orthant the height function $h_{\rhat}$ is strictly convex,
so its minimizer on any bounded closed convex set is unique and occurs
on the boundary of the set.

\begin{proposition} \label{pr:critComp}
For $B \in \mB$, let $\bx_B$ denote the unique $\bx \in \overline{B}$
on which $h$ is minimized.  Then $B \mapsto \bx_B$ defines a one-to-one
correspondence between $\mB$ and $\crit$.
\end{proposition}

\begin{proof}
The minimizer $\bx_B$ lies in $\partial \overline{B}$ and thus
on some stratum $\mS = \mS_L$.
The intersection $\mS \cap \overline{B}$ is a neighbourhood of $\bx$
in $\mS$, therefore $\bx$ is a local minimum for $h$ on $\mS$ and thus
$\bx$ is the critical point for the stratum $\mS$.  
We have shown that the range of the map $B \mapsto \bx_B$ 
is contained in $\crit$, and it remains to show that
the map is a bijection.

Pick any $\bs \in \crit$ and let $\mS = \mS_L$ be the stratum
containing $\bs$.  Writing the flat $\sing_L$ as the common zero set
of $\ell_1, \ldots  \ell_k$, the $k$ hyperplanes 
defined by the $\ell_j$ divide $\R^d$
into $2^k$ regions each containing a unique (possibly unbounded)
component $B$ of $\mMR$ whose closure contains $\bs$.  If $k=1$
then $\bb^{(1)}$ is a non-zero multiple of $\nabla h (\bs)$, and
$\bs$ is a local minimum for $h$ precisely on the closure of the
component $B \in \adj(\bs)$ closer to the origin. When $k>1$ none of
these $k$ hyperplanes are orthogonal to $\nabla h (\bs)$, by our assumption
of genericity of $\rr$. Therefore
there is a unique component $B(\bs)$ of $\mMR$ such that,
in a neighbourhood of $\bs$, if $\bx$ is on the boundary of
$B(\bs)$ and $\bv$ is a vector from $\bx$ pointing into $B(\bs)$
then $h$ increases in the direction $\bv$.  Because the point $\bs$
is a minimum for $h$ on $\mS$, it follows that $\bs$ is a local
minimum for $h$ on $B(\bs)$, hence by convexity the global minimum
for $h$ on $B(\bs)$.  Thus, we have $\bx_{B(\bs)} = \bs$.

On the other hand, for any other component $B' \neq B$ adjacent to $\bs$
there is a point $\bx$ on the boundary of $B'$, which can be taken arbitrarily
close to $\bs$, and a vector $\bv$ pointing into $B'$
such that $h$ decreases at a linear rate along $\bv$.  Because
$\bs$ is a critical point $h(\bx) = h(\bs) + O(|\bx-\bs|^2)$ on $\sing_L$, hence
$h$ cannot be minimized on $\overline{B'}$ at $\bs$.  It follows
that the choice $B(\bs)$ is unique, finishing the proof.
\end{proof}

Given $\xx \in B$ for a bounded component $B\in \mB$, 
the \emph{slide and replace} operation 
represents the fiber $\CC_\xx$ in terms of linking 
tori and lower fibers.  The slide consists of moving $\xx$ near to 
the critical point $\bs := \xx_B$ on the boundary of $B$ at which 
$h$ is minimized.  The replacement consists of replacing $\CC_\xx$ 
by $\torus_{\bs}$ plus a sum of fibers $\CC_\yy$ over points $\yy$
in components $B'$ with $h(\xx_{B'}) < h(\xx_B)$.  We illustrate this
in an example.  

\begin{example}[slide and replace] \label{ex:slide}
The left side of Figure~\ref{fig:slide}
illustrates the slide and replace operation for 
our running Example~\ref{ex:main} when 
$\hat{r}_1 > 2/3$. Starting with a basepoint $\xx$ near 
the origin in the first quadrant, we slide
$\xx$ near the critical point $\bs_1$ in the flat $\sing_1$.
Ideally we would like to replace
$\CC_\xx$ by $\torus_{\bs_1}$ but these chains are not equal. Instead,
the linking torus equals $\CC_\xx - \CC_\yy$, where $\yy$ is in the component 
$B'$ lying just to the right of $\bs_1$, and  
we replace $\CC_\xx$ by $\torus_{\bs} + \CC_\yy$.  
Repeating the slide and replace operation on $\CC_\yy$
yields $\CC_\yy = \torus_{\bs_2} + \CC_\zz$, where $\bs_2$ is
the critical point on the next line segment to the right and $\zz$ is 
in the unbounded component above $\bs_2$.  Because $\CC_\zz \doteq 0$
all fibers have been replaced by tori, so the algorithm is complete and
returns the identity
$$\CC_\xx \doteq \torus_{\bs_1} + \torus_{\bs_2} \, .$$

\begin{figure}
\centering
\includegraphics[width=2.5in]{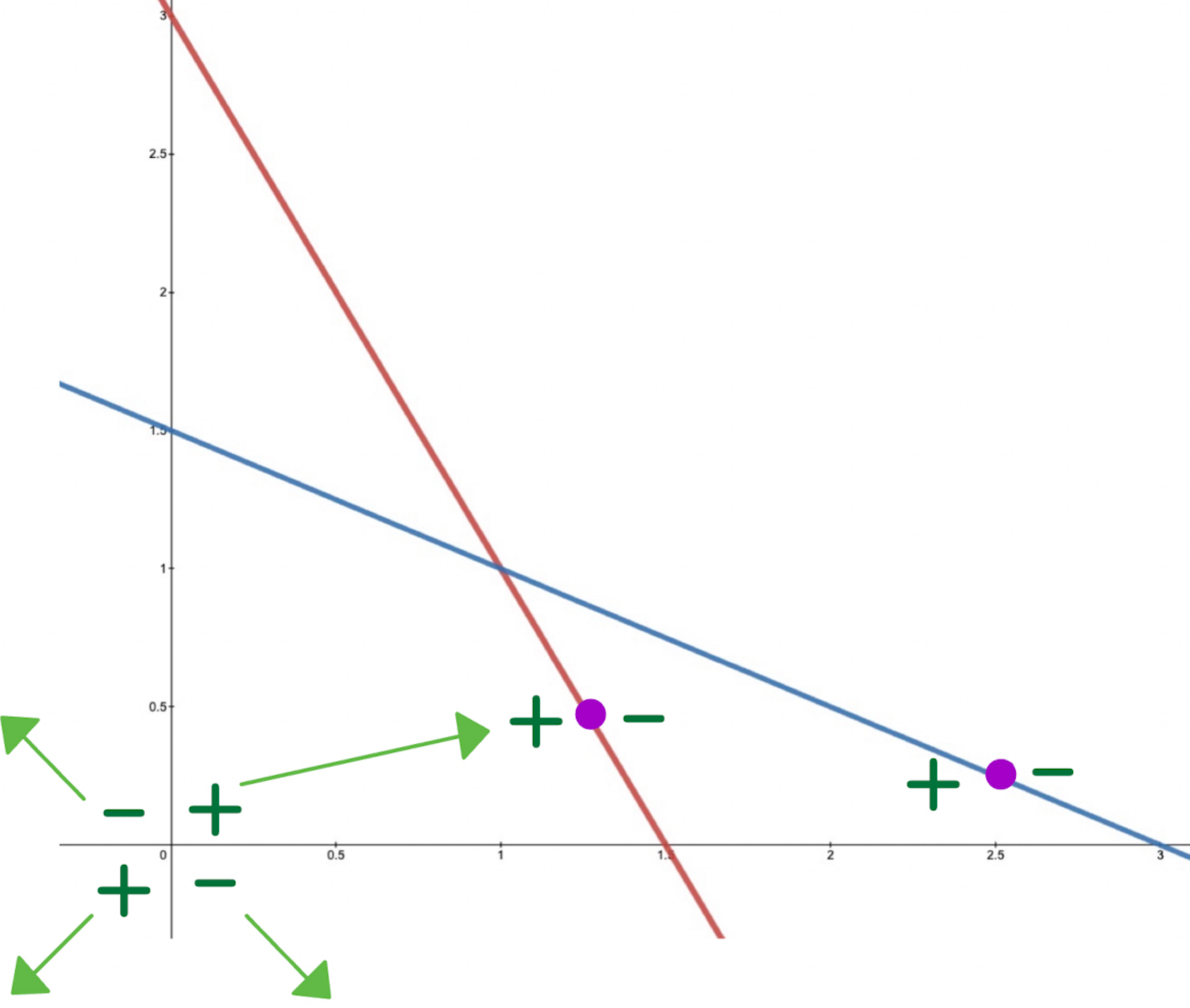}
\hspace{0.3in}
\includegraphics[width=2.5in]{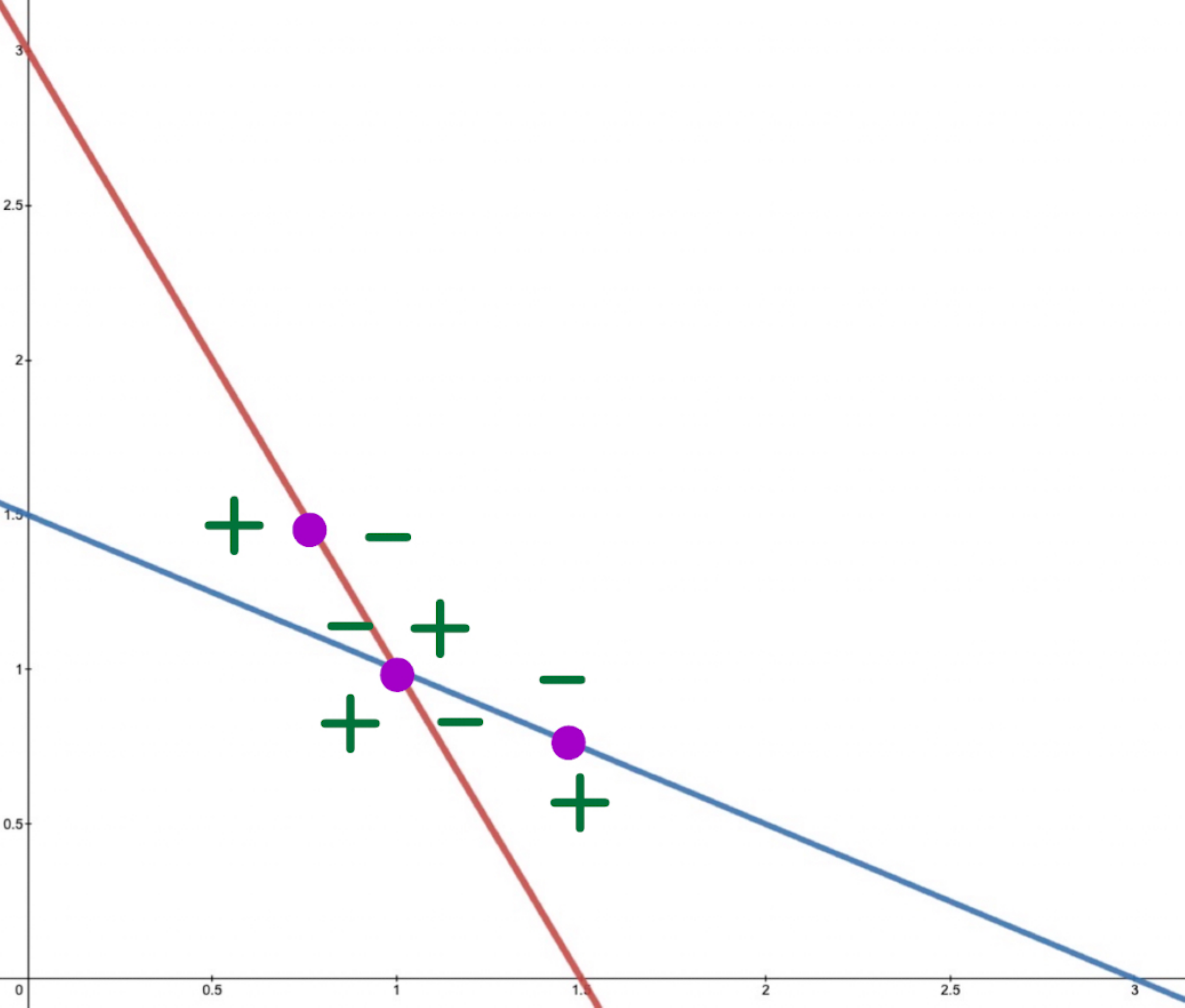} 
\caption{\emph{Left:} When $\rr=(5,1)$ we start with 
the point $\xx$ near the origin,
bring $\xx$ to the critical point $\bs_1$, then add and subtract
imaginary fibers to ultimately obtain $\CC_\xx$ as a sum of linking tori.
Note that $\CC_{\xx'}\doteq0$ for $\xx'$ in any component of $\M_\R$
whose closure contains the origin but not $\xx$. \emph{Right:} When 
$\rr=(1,1)$ we apply the slide and replace algorithm to represent $\CC_\xx$
as a sum of the linking tori seen in Figure~\ref{fig:linking}. Note that
the signs in all bounded components not containing $\CC_\xx$ cancel in
both cases.}
\label{fig:slide}
\end{figure}

The right side of Figure~\ref{fig:slide} depicts the result of running the
same algorithm
$\hat{r}_1 \in (1/3 , 2/3)$.  This time $\xx$ slides near
$\bs_{1,2} = (1,1)$ and replacing $\CC_\xx$ by
$\torus_{\bs_{1,2}}$ generates three additional imaginary fibers,
$$\CC_\xx = \torus_{\bs} + \CC_{\bb} - \CC_{\cc} + \CC_{\dd} $$
with $\bb,\cc,$ and $\dd$ as in Figure~\ref{fig:linking} above.
The points $\bb$ and $\dd$ slide to the respective critical points
$\bs_1$ and $\bs_2$, while the point $\cc$ slides to infinity.
Ultimately, we obtain the same relation
$$\CC_\xx \doteq \torus_{\bs_1} + \torus_{\bs_2} + \torus_{\bs_{1,2}}$$
observed above.
\end{example}

Formalizing the slide and replace algorithm leads to the following
two results.  By the previous proposition we may index the cycles
$\{ \CC_B : B \in \mB \}$ by $\crit$, denoting
$$\CC_{\bs} := \CC_{B(\bs)} \, .$$

\begin{proposition} \label{pr:Cspans}
The spans of $\{ \mC_{\bs} : \bs \in \crit \}$ and
$\{ \torus_{\bs} : \bs \in \crit \}$ as $\Z$-modules in
$H_d (\M , -\infty)$ are equal.
\end{proposition}

\begin{proof}
Given $\bs \in \crit$, we use the definition of $\torus_{\bs}$ 
to write
\begin{equation} \label{eq:torus2}
\mC_{\bs} = \torus_{\bs}
   - \sgn(\bs) \sum_{B' \in \adj (\bs) \setminus \{ B(\bs) \} }
   \sgn_{\bs} (B') \mC_{B'} \, .
\end{equation}
For each $B'$ in the sum, the infimum value of $h$ on $B'$ is achieved
somewhere other than $\bs$, hence $B' = B(\bs')$ for some $\bs'$ with
$h(\bs') < h(\bs)$.  Therefore,~\eqref{eq:torus2}
writes $\mC_{\bs}$ as $\torus_{\bs}$ plus a linear combination of
classes $\mC_{\bs'}$ with $h(\bs') < h(\bs)$.  Iterating, we express
$\mC_\bs$ as an integer sum of classes $\torus_{\bs}$ (because the
iteration ends at null classes $\mC_{B'}$ where the component $B'$
is unbounded).  The converse holds since $\torus_{\bs}$ is by definition 
an integer combination of classes $\mC_B$ for $B \in \adj (\bs)$.
\end{proof}

\begin{proposition} \label{pr:torusBase}
The collection $\{ \torus_{\bs} : \bs \in \crit \}$ forms a 
module basis of $H_d (\M , -\infty)$ with coefficients in~$\Z$.  
Furthermore, the change of basis matrix from
$\{\mC_{\bs} : \bs \in \crit\}$ is upper-triangular, for any total 
order~$\prec$ that extends the partial order $h(\bs) < h(\bs')$. This means
$$\mC_{\bs} = \sum_{\bs' \in \crit} \nu_{\bs , \bs'} \, \torus_{\bs'}$$
for \Em{linking constants} $\nu_{\bs , \bs'}$, where $\nu_{\bs , \bs'}=0$
unless $\bs' \preceq \bs$.
\end{proposition}

\begin{proof}
This relies on the fact that $\{ \CC_B : B \in \mB \}$ form
a basis for $H_d (\M , -\infty)$ with coefficients in $\Z$.
This is essentially Lemma~5.1 and Theorem~5.2 of~\cite{OrlikTerao1994},
most of it having been proved in~\cite{orlik-solomon1980} and
in~\cite[Chapters~3~and~5]{OrlikTerao1992}.  The result then
follows from Proposition~\ref{pr:Cspans}, with upper triangularity
following from the fact that $\bs' \preceq \bs$ whenever 
$B(\bs') \in \adj (\bs)$.  
\end{proof}

\subsection{Proof of Theorem~\ref{thm:topol}}
\label{ss:ProofofContribSum}

Proposition~\ref{pr:torusBase} implies that the (class of the) domain of
integration $\mT$ in the Cauchy integral~\eqref{eq:intoCIF}
can be written as a finite integer linear combination 
of linking tori in $H_d (\M , -\infty)$, while
Theorem~\ref{thm:topol} states that this linear combination is simply
a sum of the linking tori corresponding to contributing points. 
We begin our proof of Theorem~\ref{thm:topol} by characterizing 
contributing points via a minimizing property of the height function.  Informally, 
$\bs \in \contrib$ if and only if $\bs = \xx_B$ when all hyperplanes
other than those containing $\bs$ are dropped from $\A$ and $B$ is the
component whose closure contains both the origin and $\bs$.

\begin{proposition}
\label{prop:contribsing}
The point $\bs$ in the flat $\mV_{k_1,\dots,k_s}$ is a contributing
singularity if and only if $\bs$ is the unique minimizer of $h_{\rhat}$ on
the convex connected set
$$ \Gamma = \left\{\bz :\ell_{k_j}(\bz) \geq 0 
   \text{ for } j=1,\dots,s \right\}
   \cap \left\{\bz : \sigma_j z_j > 0 \text{ for } j=1,\dots,d \right\} $$
consisting of the polyhedron in the orthant containing $\bs$ defined by
the $\ell_j(\bz)$ and the coordinate axes.
\end{proposition}

\begin{proof}
In any fixed orthant, the strictly convex function $h(\bz)$ has a unique
minimizer on the convex set $\Gamma$.  Because $\bs$ is critical, and we 
are in a generic direction, we can write
$$ (-\nabla h)(\bs) =  \sum_{j=1}^s a_j \bb^{(k_j)} $$
for some $a_j \in \R_*$. If $a_j>0$ for all $j$ then
\[ (\nabla h)(\bs) \cdot (\ww - \bs) = \sum_{j=1}^s a_j\left(1-\ww \cdot
   \bb^{(k_j)}\right) \geq 0 \]
for all $\ww \in \Gamma$, so $h(\zz)$ decreases as $\zz$ moves
from $\ww$ towards $\bs$, and $\bs$ minimizes $h(\bz)$.
Conversely, if, say, $a_1<0$ and we pick $\ww$ such that
$\ell_{k_j}(\ww)=0$ for $j=2,\dots,s$ and $\ell_{k_1}(\ww)>0$, then
$(\nabla h)(\bs) \cdot (\ww - \bs) = a_1(1-\bb^{(k_1)} \cdot \ww) < 0$
and for $\epsilon>0$ sufficiently small we have $\bv = \bs +
\epsilon(\ww - \bs) \in \Gamma$ with $h(\bv) < h(\bs)$. Thus, $\bs$ is
contributing if and only if it is the unique minimizer of $h_{\rhat}$ on $\Gamma$.
\end{proof}

\begin{example}
\label{ex:main2}
Continuing Example~\ref{ex:main}, let $B$ be the bounded component of $\mMR$
whose closure contains the origin.  
The cones $N(\bs_1)$ and $N(\bs_2)$ are rays which contain 
$-(\nabla h)(\bs_1)$, meaning they are always
contributing points.  The critical point $\bs_{1,2} = (1,1)$ is more
interesting.  By definition $\bs \in \contrib$ if and only if
$\hat{r}_1 \in [1/3 , 2/3]$ and $\rhat$ is generic 
(which removes the endpoints).  For precisely
these values of $\hat{r}_1$, the point $\bs_{1,2}$ is the minimizer of
$h_{\rhat}$ on $\overline{B}$.  When $\hat{r}_1 < 1/3$ the minimizer 
of $h_{\rhat}$ on $\overline{B}$ is $\bs_2$ on the upper edge of
$\overline{B}$, while for $\hat{r} > 2/3$ the minimizer is $\bs_1$ 
on the right edge of $\overline{B}$.
\end{example}

Because we want to compute linking constants around critical points far from
the origin, we next examine how they can change when crossing over $\sing$. 
If $B$ is an unbounded component of $\mMR$ then we write
$h(\sigma (B)) := - \infty$ and $\nu_{\sigma(B),\bs_*} := 0$
for any $\bs_*\in\crit$.

\begin{lemma} \label{lem:cross}
Fix $\bs_* \in \crit$ and let $B , B'$ be two components of $\mMR$
in the same orthant as $\bs_*$.  Suppose there $B$ and $B'$ are separated
by a unique hyperplane $\sing_j$ not containing $\bs_*$ (so that
there is a unique index $j$
such that $\ell_j$ is positive on $B$ and negative on $B'$, or vice versa,
and that $\ell_j (\bs_*) \neq 0$).  Then $\nu_{\sigma (B) , \bs_*}
= \nu_{\sigma (B') , \bs_*}$.
\end{lemma}

\begin{proof}
Assume, without loss of generality, that $B$ and $B'$ lie in the positive
orthant. If both components are unbounded then $\nu_{\sigma (B) , \bs_*}
= \nu_{\sigma (B') , \bs_*} = 0$, and we are done. Otherwise, one of 
$\sigma(B)$ and $\sigma(B')$ is finite and must lie on $\sing_j$:
if not then by strict convexity of $h$ the line segment between 
them intersects $\sing_j$ at a point of lower height than 
$\sigma(B)$ and $\sigma(B')$, contradicting Proposition~\ref{pr:critComp}.

We may therefore assume that $\ell_j (\sigma (B)) = 0$, 
and write
\begin{equation} \label{eq:torusBase}
\mC_B - \mC_{B'} = \torus_{\sigma (B)}
   + \sum_{B'' \in \adj (\sigma(B)) \setminus \{ B , B' \}}
   \sgn_\sigma (B'') \mC_{B''} \, .
\end{equation}
Grouping the elements of $\adj (\sigma(B))$ into pairs of components
$(\beta,\beta')$ separated only by $\sing_j$, we note that $\mC_{\beta}$
and $\mC_{\beta'}$ appear with opposite signs on the right-hand side 
of~\eqref{eq:torusBase}, and that both $\sigma(\beta)$ and $\sigma(\beta')$
have height less than $\sigma(B)$. The result then follows by induction
on the maximum heights of $\sigma(B)$ and $\sigma(B')$:
if $\sigma(B)$ and $\sigma(B')$ have height $-\infty$ then
$\mC_{B} \doteq \mC_{B} \doteq 0$, otherwise we have decomposed them as
a sum of pairs $(\beta,\beta')$ with smaller maximum height whose contributions
to~\eqref{eq:torusBase} cancel.
\end{proof}

The last preliminary step is to we express
the torus $\mT$ from the original Cauchy integral in the fiber basis.
Let $\adj(\bzer)$ denote the components of $\mMR$ adjacent to the
origin and define
$$\torus_{\bzer} := \sum_{B \in \adj(\bzer)} \sgn (\mO)\mC_{B} \, ,$$
where $\sgn (\mO)$ is~1
for the positive orthant and alternates across coordinate planes.
This was depicted by the gray fibers in Figure~\ref{fig:linking}.

\begin{pr} \label{pr:mT}
If $\mT$ is the chain of integration in~\eqref{eq:intoCIF} then
$\mT \doteq \torus_{\bzer} .$
\end{pr}

\begin{proof}
The chain $\mT$ is parametrized by a map $\eta$ from the standard 
torus $(\R / (2 \pi \Z))^d$ defined by 
$$\eta (\bt) := (\ee \cos \theta_1 + i \delta \sin \theta_1 , \ldots , 
   \ee \cos \theta_d + i \delta \sin \theta_d ) \, $$
with $\delta = \ee$ arbitrarily small.  Increasing $\delta$
to infinity, we note that the chain remains in $\M$ and that the 
intersection with $\M_{\geq \low}$ converges to the alternating
sum of imaginary fibers $\torus_{\bzer}$.  The one-dimensional story
is illustrated in Figure~\ref{fig:Ex1ImFiber}.
\end{proof}

\begin{figure}
\centering
\includegraphics[width=0.3\linewidth]{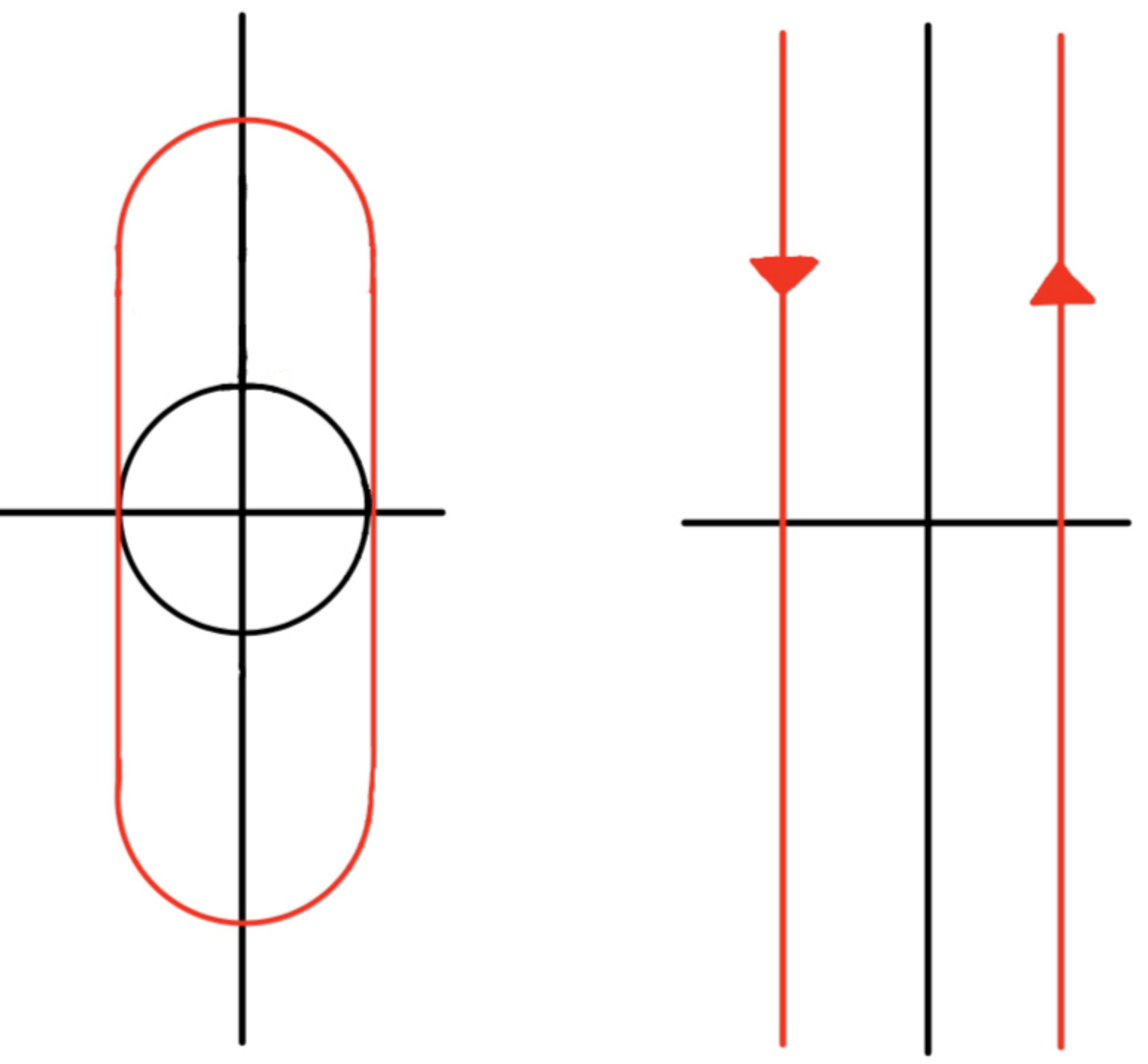} 
\caption{Deforming the domain of integration from circle to 
the difference of two imaginary fibers: stretching the grey circle
in the imaginary direction, on any bounded region the chain converges
locally to the red difference of fibers.}
\label{fig:Ex1ImFiber}
\end{figure}

The following result now immediately implies Theorem~\ref{thm:topol}.

\begin{theorem} \label{th:sumTorus}
$$\torus_{\bzer} = \sum_{\bs \in \contrib} \, \torus_{\bs} .$$
\end{theorem}

\begin{proof}
Let $\bs^*$ be a critical point in the positive orthant $\mO$ and $B$
be the unique component of $\mMR$ in $\mO$ whose closure contains the
origin. The hyperplanes through $\bs_*$ divide $\mO$ into regions on
which $\nu_{\cdot , \bs_*}$ is constant by Lemma~\ref{lem:cross},
taking the common value 1 on the region containing $B(\bs_*)$
and vanishing on the other regions (as their critical points 
minimizing $h$ have smaller
height than $\bs_*$). Proposition~\ref{prop:contribsing} implies 
that the origin is in the closure of the region
containing $B(\bs^*)$ if and only if $\bs_*$ is contributing, 
so $\nu_{\sigma(B) , \bs_*} = 1$ if $\bs_* \in \contrib$
and zero otherwise. A similar argument proves the result for the other
orthants of $\R^d$.
\end{proof}

\subsection{Integrals over linking tori} \label{ss:integrals}

Throughout this section, the assumption that $F(\zz) = G(\zz) / \prod \ell_j(\zz)^{p_j}$
is simple and $\rr$ is generic remain in force.  If $G$ is a polynomial,
Proposition~\ref{pr:low} and Theorem~\ref{thm:topol} imply that
when $F$ is a rational function, 
\begin{equation} \label{eq:decomp}
[\bz^{n\br}] F(\bz) = \frac{1}{(2\pi i)^d}
   \sum_{\bs \in \contribS} 
   \int_{\torus_{\bs}} 
   \frac{F(\zz)}{\zz^{\rr + \one}} \, d\zz \, .
\end{equation}
More generally, when $G$ is analytic, there is an extra term on the right
decaying faster than any exponential $e^{-c |\rr|}$.  In any case,
we see that the problem is reduced to evaluating Cauchy 
integrals over linking tori, which has been solved, for instance, 
in~\cite{PW2} and~\cite[Section~10.3]{PW-book}.  

In order to provide some intuition, before quoting the general 
result we look at a special case.  Suppose that $\bs$ forms a 
zero-dimensional stratum, that is, a single point where $d$
of the $m$ hyperplanes intersect.  The linking torus is then
homeomorphic to an actual torus (cf.\ the end of Section~\ref{ss:definitions},
there is no tube).  In coordinates given by the linear functions defining
$\bs$, the
linking torus is the product of circles winding once counterclockwise
about the origin in each coordinate.  Taking $d$ iterated residues
(or, equivalently, a $d$-fold residue) one arrives at
Proposition~\ref{pr:complete} below.

\begin{defn}[logarithmic gradients] \label{def:log}
The \Em{logarithmic gradient} of any differentiable function $f$ is
\begin{equation} \label{eq:lgrad}
\lgrad f(\zz) := \left ( z_1 \frac{\partial f}{\partial z_1} , \cdots ,
   z_d \frac{\partial f}{\partial z_d} \right ) \, .
\end{equation}
If $\bs$ is a critical point on the stratum defined by 
the intersection of $s \leq d$ hyperplanes $\sing_{k_1}, \ldots , 
\ell_{k_s}$ then the \Em{log-gradient matrix} 
$\Gamma = \Gamma (\bs)$ is the $d \times d$ matrix whose first $k$ rows
are $\{ \lgrad \ell_{k_j} (\bs) : 1 \leq j \leq s \}$ and whose last $d-k$
rows are any $d-k$ standard basis vectors that complete these
logarithmic gradients to a positively oriented basis for $\R^d$. 
\end{defn}

\begin{pr}[\protect{\cite[Theorem~10.3.3]{PW-book}}] \label{pr:complete}
Suppose that $\bs$ has codimension $d$ and let $\Gamma$ be the 
log-gradient matrix as in Definition~\ref{def:log}.  If $G$ is a polynomial then
$$\frac{1}{(2 \pi i)^d} \int_{\torus_{\bs}} \zz^{-\rr-\one} F(\zz) \, d\zz 
   = \frac{\bs^{-\rr}}{(\pp - 1)!} \frac{G(\bs)}{|\det \Gamma|} 
   (\rr \Gamma^{-1})^{\pp - 1} \, .$$
When $G$ is analytic, rather than a polynomial, then there is a remainder term
decaying faster than an exponential function in $|\rr|$.  When all powers are~1,
the right-hand side simplifies to $\bs^{-\rr} G(\bs) / |\det \Gamma|$.
$\Cox$
\end{pr}

\begin{example}[$d$-fold residue] \label{ex:complete}
Continuing Example~\ref{ex:main2}, suppose $\hat{r}_1 \in (1/3,2/3)$
so that $\bs_{1,2} = (1,1)\in \contrib$. Then
$\Gamma(\bs_{1,2}) = -\begin{pmatrix} 2/3 & 1/3 \\ 1/3 & 2/3 \end{pmatrix}$
and $\pp = \one$, so
$$\frac{1}{(2 \pi i)^d} \int_{\torus_{\bs}} \zz^{-\rr-\one} F(\zz) \, d\zz 
   = \frac{1^n}{|\det \Gamma|} = 3 \, .$$
\end{example}

In general, any linking torus can be expressed as the product of an $s$-torus
with a $(d-s)$-dimensional imaginary fiber. Under our current assumptions,
we can thus always compute $s$ residues and be left with a $(d-s)$-dimensional
integral in stationary phase. The following result 
shows that $h$ is well-behaved around its critical points, so that we will
be able to compute the stationary phase integral.

\begin{pr}[$h$ is Morse] \label{pr:morse}
Let $\sing$ be the complexification of any real hyperplane arrangement
and let $\bs \in \sing_*$ be a critical point of $h|_{\mS}$ for any stratum $\mS$
of any dimension $k$.  Then the Hessian matrix of $h|_{\mS}$ at $\bs$ 
in any local $k$-dimensional complex coordinates is nonsingular.
\end{pr}

\begin{proof}
Because $\bs$ is a critical point the differential $dh$ vanishes on $\mS$
at $\bs$, so $h(\bs + \zz)$ is locally a quadratic form in $\zz$. If 
$\xx$ is any real vector then every coordinate of $\xx+i\yy$ has increasing
modulus as any coordinate of the real vector $\yy$ moves away from the origin.
Thus, if $\mS_\R$ denotes the real part of the stratum $\mS$ 
then $h_\rr(\zz) = -\sum_j r_j \log |z_j|$ has a local maximum 
at $\zz=\bs$ on the space $\bs + i\mS_\R$ with real dimension $k$
and a local minimum at $\zz=\bs$ on the space $\bs + \mS_\R$ with
real dimension $k$. In other words, $h(\bs+\zz)$ is Morse 
at the origin of $\mS \oplus i \, \mS$ with middle index.
\end{proof}

Proposition~\ref{pr:morse} implies quadratic nondegeneracy of the phase 
function for the stationary phase integral, allowing it to be
evaluated as a standard saddle point integral.

\begin{pr}[\protect{\cite[Theorem~10.3.4]{PW-book}}] \label{pr:partial}
Let $S = \{k_1,\dots,k_s\}$ and suppose $\bs$ lies on the flat $\sing_S$.
Let $\Gamma$ be the log-gradient matrix in Definition~\ref{def:log} and $M$ be
the $d \times d$ matrix whose first $k$ rows
are the coefficient vectors $\{ \bb^{(k_j)} : 1 \leq j \leq k \}$ 
and whose last $d-k$ rows are any $d-k$ standard basis vectors that complete these
gradients to a positively oriented basis for $\R^d$. 
Define the $(d-k)\times(d-k)$ matrix $\mH$ to be the Hessian of
\[ \phi(\by) = \br \cdot \log\left(\bs - iM^{-1} 
\begin{psmallmatrix} \bzer \\ \by \end{psmallmatrix} \right) \] 
evaluated at $\by=(y_1,\dots,y_{d-k})=\bzer$, where the logarithm is taken
coordinate-wise. Then there is an explicitly computable asymptotic series
in $|\rr|=|r_1| + \cdots + |r_d|$ beginning
\begin{align*}
\frac{1}{(2 \pi i)^d} \int_{\torus_{\bs}} \zz^{-\rr-\one} F(\zz) \, d\zz
   = \left[ 
   \frac{ \bs^{-\rr} G(\bs)}
   {\prod_{j \notin S}\ell_j(\bs)^{p_j}\prod_{1 \leq j\leq s}(p_{k_j}-1)!
   \sqrt{\det \mH} \; |\det \Gamma| } 
   \right]&
   (2 \pi |\rr|)^{-(d-k)/2} (\rr \Gamma^{-1})^{\pp - 1} \\
   & \hspace{0.5in} \times \left(1 + O\left(\frac{1}{|\rr|}\right)\right).
\end{align*}
All asymptotic terms in this series are uniform as $\bs$ varies
over a compact subset of $S$.
$\Cox$
\end{pr}

\begin{rem}
Raichev and Wilson~\cite{RaichevWilson2011} give explicit formulas for the
coefficients in the asymptotic expansion in Proposition~\ref{pr:partial}, 
building on expansions of H{\"o}rmander~\cite[Theorem~7.7.5]{Hormander1990a};
see also Melczer~\cite[Proposition 5.3]{melczer-book}. 
\end{rem}

We summarize the results of this Section as follows.

\begin{theorem}
\label{thm:mainASM}
Suppose $F(\bz)$ is simple and $\br$ is a generic direction, and let $\contrib$ 
denote the set of contributing singularities for $F(\bz)$. Then as
$\rr\rightarrow\infty$ with $\rhat$ staying in a compact subset of
$\R_{>0}^d$ consisting of only generic directions there exist
asymptotic series $\Phi_{\bs}(\rr)$ such that 
\begin{equation} \label{eq:MainAsm}
[\zz^{\rr}] F(\bz) = \sum_{\bs \in \contribS(\rhat)} \Phi_{\bs}(\rr).
\end{equation}
If $\bs$ lies on the flat $\mV_S$ with $S = \{k_1,\dots,k_s\}$ then,
for any positive integer $K$, there exist effective constants $C_j^{\bs}$ 
such that 
$$ \Phi_{\bs}(\rr) = \bs^{-\br} \, |\rr|^{p_{k_1}+\cdots+p_{k_s} - (s+d)/2}
   \left ( \sum_{j=0}^K C_j^{\bs} |\rr|^{-j} 
   + O\left(|\rr|^{-K- 1}\right)\right ) \, . $$
If $G(\bs) \neq 0$ then the leading asymptotic term of $\Phi_{\bs}$ is
given by Proposition~\ref{pr:partial}.  The error term varies
uniformly as $\rhat$ varies without crossing nongeneric directions 
and $\bs$ varies over a compact subset of any stratum. 
\end{theorem}

\begin{rem} \label{rem:highpoles}
When $\bs$ has dimension zero and $G$ is a polynomial
then $\Phi_{\bs} = \bs^{-n \br} P_{\bs}(n)$, where $P_{\bs}(n)$ 
is a polynomial in $n$ of degree $p_{k_1}+\cdots+p_{k_d}-d$
which can be determined exactly.  The asymptotic expansion in
deceasing powers of $|\rr|$ has no further terms and the remainder
is exponentially decreasing in $|\rr|$.  This phenomenon, that 
asymptotic behaviour can be determined in this case up to an 
exponentially small remainder, was previously noted by 
Pemantle~\cite{Pemantle2000}.
\end{rem}

\subsection{Implementation and Additional Examples}
\label{sec:implement}
A Maple implementation of Theorem~\ref{thm:mainASM} via Algorithm~\ref{alg:1} 
is available at
\begin{center}
\url{https://github.com/ACSVMath/ACSVHyperplane}
\end{center} 
and allows the user to easily derive asymptotics 
for a generic direction in the non-simple case. We illustrate a few 
additional examples here, pointing out the computer algebra issues 
which arise.

\begin{figure}
\centering
\includegraphics[width=0.8\linewidth]{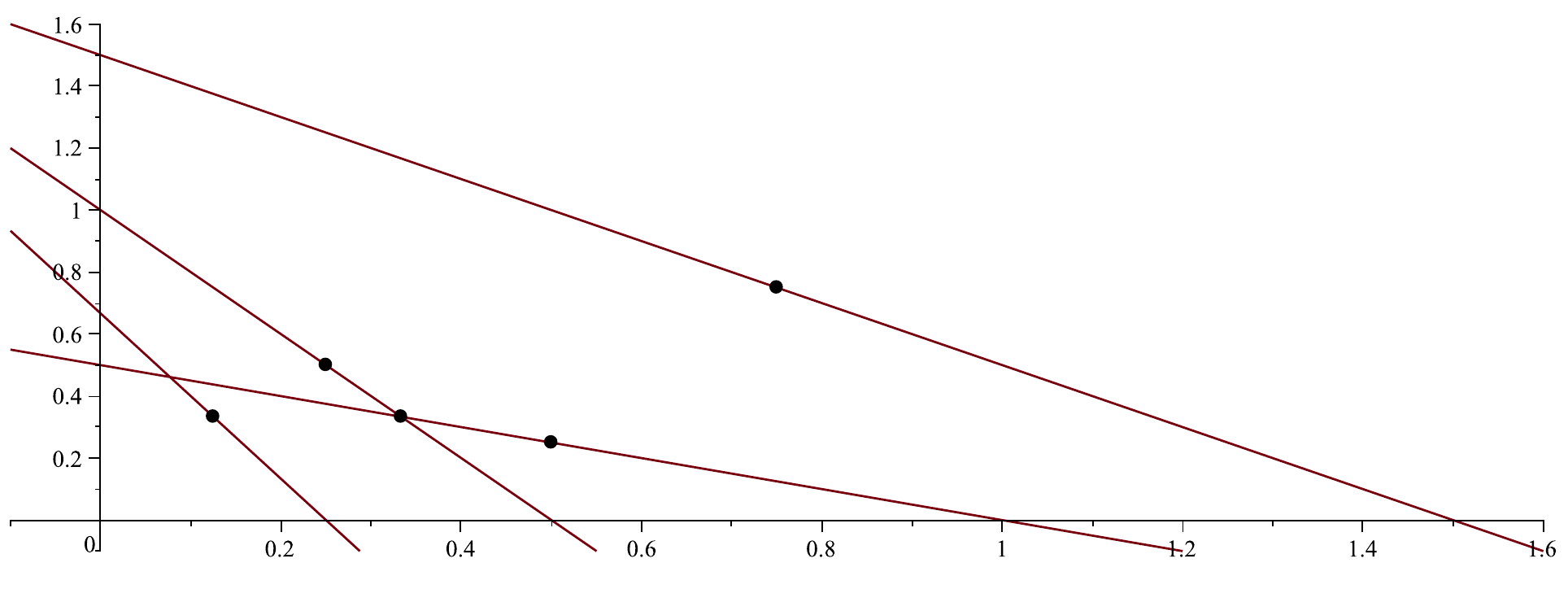}
\caption{Real singularities and contributing points of the rational 
function in Example~\ref{ex:compute1}. No singularities outside the first 
quadrant are contributing.}
\label{fig:compute}
\end{figure}

\begin{example}
\label{ex:compute1}
Consider asymptotics in the main diagonal direction 
$\rr = (n,n)$ for the rational function 
\[ F(x,y) = \frac{1}{(1-2x-y)(1-x-2y)(1-4x-3y/2)(1-2x/3-2y/3)},\]
whose real singularities and contributing points are illustrated in 
Figure~\ref{fig:compute}.  The algorithm detects 5 contributing points: 
note that one of the common intersections of two factors is contributing 
while the other intersections (one in the first quadrant and further 
intersections in the other quadrants) are not contributing. The algorithm 
then computes asymptotic contributions with leading terms
\[ \frac{64 \cdot 8^n}{11 \sqrt{n \pi}}, \quad \frac{9^n}{3}, \qquad 
\frac{32 \cdot 8^n}{3 \sqrt{n \pi}}, \qquad \frac{10368 \cdot 24^n}{625 
\sqrt{n \pi}}, \qquad \frac{-128 \cdot (16/9)^n}{625 \sqrt{n \pi}},  \]
giving dominant asymptotics
\[ [x^ny^n]F(x,y) = \frac{10368 \cdot 24^n}{625 \sqrt{n 
\pi}}\left(1+O\left(\frac{1}{n}\right)\right). \]
\end{example}

\begin{example}
Our previous examples have been restricted to dimension 2, and admitted 
contributing points in the first quadrant only, in order to guide 
intuition and allow for visualization. In order to illustrate our results 
in a more general setting we consider the rational function
\[ F(x,y,z,w) = \frac{1}{(1+2x+y+z+w)(1-x-3y-z-w)(1-x+y-4z+w)(1-x-y+z-
5w)}\]
in the direction $\br=(n,2n,n,2n)$. Here there are 20 contributing 
singularities in multiple orthants, many flats have contributing 
singularities in multiple quadrants, and several contributing points have 
irrational coordinates. We use the symbolic-numeric methods of Melczer 
and Salvy~\cite{MelczerSalvy2021} to store the coordinates of the 
contributing points: for each contributing point $\bs$ the algorithm 
outputs an algebraic number $\alpha$ defined by a square-free integer 
polynomial $P(u)$ and isolating interval, together with integer 
polynomials $Q_j(u)$ for each coordinate, such that the $j$th coordinate 
$\sigma_j = Q_j(\alpha)/P'(\alpha)$. 

For example, the flat defined by 
\[ 0 \;=\; 1+2x+y+z+w \; = \;  1-x-3y-z-w \;  = \;  1-x+y-4z+w 
\;  = \;  1-x-y+z-5w \]
contains three critical points, two of which are contributing points with 
coordinates
\[ \left\{ \left( \frac{-3w-1}{2}, \frac{3-3w}{4}, \frac{11w-3}{4} 
\right) : 99w^3-61w^2-9w+3 = 0, \, w \approx .1791\dots \text{ or } 
.6843\dots \right\} \]
containing algebraic numbers of degree three (the final algebraic 
conjugate is critical but not contributing). The algorithm encodes these 
numbers as 
\[ (x,y,z,w) = 
\left(\frac{Q_x(u)}{P'(u)},\frac{Q_y(u)}{P'(u)},\frac{Q_z(u)}{P'(u)},
\frac{Q_w(u)}{P'(u)}\right) \]
evaluated at the roots of $P(u) = 2970u^3-282927u^2+8961876u-94409378$ of 
value approximately 29.731 and 34.86, where
\begin{align*} 
Q_x(u) &= -7200u^2+459750u-7326354 &
Q_y(u) &= 5310u^2-335979u+5298699 \\
Q_z(u) &= -1650u^2+100215u-1504811 & 
Q_w(u) &= 1830u^2-117882u+1896944.
\end{align*}
The roots of $u$ are given to sufficient precision to separate them among 
those of $P$. The algorithm also computes the constants giving $-(\nabla 
h_{\br})(\bz)$ as a linear combination of the coefficient vectors 
$\bb^{(j)}$, and determines $u$ with enough accuracy to decide whether 
all these constants are positive (in which case the critical point is 
contributing) or whether any are zero (in which case $\br$ is a non-
generic direction).

There is some randomness in selecting the polynomial $P(u)$; see Melczer 
and Salvy~\cite{MelczerSalvy2021} for the advantages of this 
representation and how to compute it, and the Maple worksheet 
accompanying this paper for more details on this example. Computing the 
asymptotic contributions of each contributing point here gives dominant 
asymptotics of the form
\[ [x^ny^nz^nw^n]F(x,y) = \frac{C}{\pi}\left(\frac{-11072781\sqrt{249}-
54475983}{3125}\right)^n n^{-1} \left(1+O\left(\frac{1}{n}\right)\right), 
\]
where $C=-0.109\dots$ is an explicit degree 4 algebraic number.
\end{example}

\subsection{Vanishing of Leading Coefficients} \label{ss:vanish}

To find dominant asymptotics of $[\bz^{n\rhat}]F(\bz)$ using 
Theorem~\ref{thm:mainASM}, one starts with a contributing singularity 
$\bs$ of largest height and tries to find the first constant $C_j^{\bs}$ 
which is non-zero, repeating the process until all maximum height 
contributing singularities have been examined.  If the right-hand side
of~\eqref{eq:MainAsm} is non-zero, such a non-zero constant $C_j^{\bs}$ 
exists and will eventually be found.  When there is a unique 
contributing point $\bs$ of maximal height with
$G(\bs) \neq 0$ then this gives dominant asymptotic behaviour.  
If, however, there is more than one such contributing point
then it is possible that their sum will sometimes be of smaller
order: for instance, if $F(z)=1/(1-z^2)$ our algorithm will correctly
return $[z^n]F(z) = 1+(-1)^n$. For generic directions $\rr$ the 
leading term in our asymptotic statement will be
non-zero.

A more serious difficulty occurs when all $C_j^{\bs}$ vanish at
all contributing points $\bs$ of largest height.  
This implies an exponentially smaller
estimate $\Phi_{\brho}$ holds, where $\brho$ is a critical point of height
less than $\bs$.  We give additional details on this in 
Theorem~\ref{th:ideal} below, after some examples illustrating 
the possible outcomes.

\begin{example} \label{ex:three}
Consider the rational functions
\[ A(x,y) = \frac{1-x-y}{1-x-y}, \quad B(x,y) = \frac{1}{1-x-y}, 
\quad C(x,y) = \frac{x-2y^2}{1-x-y}, \quad 
\text{and} \quad D(x,y) = \frac{x-y}{1-x-y}. \]
These functions share the same denominator, and all appear to
admit the single contributing singularity $\bs = (1/2,1/2)$.  
Of course, $A$ only appears to admit this point because it is not
simplified in lowest terms (simplified it is the constant 1) but
$\bs$ is a contributing point of $B,C,$ and $D$. Applying 
Theorem~\ref{thm:mainASM} to the main diagonal of $B$ gives
\[ [x^ny^n]B(x,y) = 4^n\left(\frac{1}{\sqrt{\pi}n^{1/2}}+ 
O\left(n^{-3/2}\right)\right).\]
Because the numerator of $C(x,y)$ vanishes at $\bs$ we 
determine dominant asymptotic behaviour of its main diagonal
by computing higher order constants, ultimately obtaining
\[ [x^ny^n]C(x,y) = 4^n\left(\frac{1}{4\sqrt{\pi}n^{3/2}} 
+ O\left(n^{-5/2}\right)\right).\]
Finally, direct inspection shows that the main diagonal of $D$ 
is identically zero, but applying Theorem~\ref{thm:mainASM} in an automatic manner 
only allows us to show for any $K>0$ that
\[ [x^ny^n]D(x,y) = O\left(\frac{4^n}{n^K}\right). \]
Note that for any $a>0$ we have
\[ [x^{an}y^n]D(x,y) = \left(\frac{(a+1)^{a+1}}{a^a}\right)^n\,\left(\frac{a-1}{\sqrt{2\pi a(a+1)n}} + O\left(n^{-3/2}\right)\right),  \]
so that the exponential growth of $[x^{an}y^n]C(x,y)$ approaches $4^n$ as $a\rightarrow1$.
\end{example}

\begin{rem}
Such unexpected exponential drops are related to the \emph{connection problem} from
enumerative combinatorics: given a 
sequence satisfying a linear recurrence relation with polynomial 
coefficients---including any diagonal sequence $[\bz^{n\br}]F(\bz)$ with $F$ 
rational---one can compute a finite asymptotic basis of functions such 
that dominant asymptotics of the sequence under consideration is a
$\C$-linear combination of basis elements. The real coefficients in such a 
linear combination of basis elements can be determined rigorously to any 
desired numerical accuracy~\cite{Hoeven2001,Mezzarobba2016}, however it 
is still unknown whether it is decidable to determine which constants are 
exactly zero (see also Section VII.9 of Flajolet and 
Sedgewick~\cite{FlajoletSedgewick2009}).
\end{rem}

When $\bs$ forms a zero-dimensional stratum, things are better.
The degree of the polynomial produced by Proposition~\ref{pr:complete} 
is bounded above by the sum of $(p_j - 1)$ over those $j$ such that 
$\ell_j(\bs)=0$.  It follows that testing $\Phi_{\bs} = 0$
can be done rigorously
by computing a sufficient number of effective coefficients.  When the 
contribution of all maximal height contributing singularities is known 
to vanish, one can repeat this effective process on the contributing 
singularities of next highest height, and so on.  We now give an example
where the highest contributing critical point $\bs$ is of dimension zero 
and it is indeed necessary to go to the next lower critical point.

\begin{example}
Let 
$$ F(x,y) = \frac{x-y}{\left(1 - \frac{2x+y}{3}\right)
   \left(1 - \frac{x+2y}{3}\right)} $$
and consider asymptotics in the direction $(r,s)=n(\alpha,1-\alpha)$ for
$\alpha = \hat{r}_1 \in (2/3,3/4)$.
This rational function has the same denominator the one in 
Example~\ref{ex:main}, so all critical points, contributing 
singularities, etc. are the same.  In particular, the point $\bs_{1,2} = 
(1,1)$ is a contributing singularity of maximal height.  This time
the contribution at $\Phi_{\bs_{1,2}}$ is identically zero:
this is easy to discover because our degree bound implies the 
leading asymptotic term is a polynomial of degree zero, and one 
need only evaluate the first coefficient.  
Alternatively, one may verify that the numerator $x-y$ 
is in a specific ideal described below.  Dominant asymptotics of 
$a_{r,s}$ are thus determined by $\Phi_{\bs_1}$ and $\Phi_{\bs_2}$,
where $\bs_1$ and $\bs_2$ are the one-dimensional critical points
$$ \bs_1 = \left(\frac{3\alpha}{2},3(1-\alpha)\right) 
\qquad\text{and}\qquad 
\bs_2 = 
\left(3\alpha,\frac{3(1-\alpha)}{2}\right). $$
Exponentiating the heights of $\bs_1$ and $\bs_2$ gives the values
$$ h_1 = \frac{1}{3(1-\alpha)}\left(\frac{2(1-\alpha)}{\alpha}\right)^\alpha 
   \qquad\qquad h_2 = 
   \frac{2}{3(1-\alpha)}\left(\frac{1- \alpha}{2\alpha}\right)^\alpha \ , $$
and basic calculus shows 
$$
\begin{cases}
h_1 < h_2 &: \alpha \in (1/3,1/2) \\
h_1 = h_2 &: \alpha = 1/2 \\
h_1 > h_2 &: \alpha \in (1/2,2/3)
\end{cases}.
$$ 
Thus, Theorem~\ref{thm:mainASM} implies 
$$ a_{r,s} = \left ( \frac{1}{3\alpha} \right )^{r}
   \left(\frac{2}{3(1-\alpha)}\right)^s (r+s)^{-1/2} 
   \left(\frac{-3\sqrt{2}}{2\sqrt{(1-\alpha)\alpha\pi}} 
   + O\left((r+s)^{-1}\right)\right) $$
for $\alpha \in (1/3,1/2)$, while 
$$ a_{r,s} = \left ( \frac{2}{3\alpha} \right )^r
   \left ( \frac{1}{3(1-\alpha)} \right )^s (r+s)^{-1/2} 
   \left(\frac{3\sqrt{2}}{2\sqrt{(1-\alpha)\alpha\pi}} 
   + O\left((r+s)^{-1}\right)\right) $$
for $\alpha \in (1/2,1/3)$. When $\alpha=1/2$ dominant asymptotic
behaviour is given by the sum of these contributions, which cancel,
reflecting the fact that the main diagonal of $F(x,y)$ is zero
by symmetry.
\end{example}

\subsubsection*{Exponential growth rates and the annihilating ideal}

As we have seen, asymptotic growth of a coefficient 
sequence can be lower than that predicted by the 
highest height contributing points for some fixed direction. 
For this reason, it is useful to introduce two different directional
exponential rates.

\begin{defn} \label{def:rates}
The (limsup) \Em{exponential rate} of a sequence $a_{\bss}$
in the direction $\rr$ is the quantity 
$$\beta (\rr) := \limsup_{n \to \infty} \frac{1}{n} \log |a_{n \rhat}|,$$
while the (limsup) \Em{neighbourhood exponential rate} is
$$\obeta (\rr) := \limsup_{\substack{\bss \to \infty \\ \bss / |\bss| \to \rhat}} 
   \frac{1}{|\bss|} \log |a_{\bss}|.$$
The (neighbourhood) exponential growth rate of a generating function is
the (neighbourhood) exponential growth rate of its coefficient sequence.
\end{defn}

Using the limsup helps smooth behaviour because many combinatorial
sequences have periodicity, either exact with period $k$
(e.g., terms vanish for certain values of $n$ mod $k$) or via 
a phase term such as a factor of $\cos (n \theta \pi)$ for some,
possibly irrational, $\theta$.  
The reason to look at neighbourhood growth rates is explained by 
Example~\ref{ex:three}: for all the functions but the trivial case of $A(x,y)=1$
we saw that
$$\obeta (1,1) = h_{1/2,1/2} (\bs) = \log 4,$$
while the behavior precisely on the diagonal has
a polynomial drop (for $C$) or is identically
zero (for $D$).  

\begin{lemma}\label{lem:convo}
The neighbourhood exponential rate does not increase when multiplying $F(\zz)$ by
any Laurent polynomial $\kappa(\zz)$.
\end{lemma}

\begin{proof}
Multiplication by $\kappa$ translates into the convolution of the series $F$ with 
the Laurent polynomial $\kappa$. The claim follows as $\kappa$ has finite
support.
\end{proof}

\begin{rem}
The (non-neighbourhood) exponential rate of a sequence can change both increase
and decrease when multiplying its generating function by a polynomial. For instance, 
\begin{align*}
F(x,y) = \frac{1}{1-x-y} &\implies \beta(1,1) = \log 2 \\[+2mm]
F(x,y) = \frac{x-y}{1-x-y} &\implies \beta(1,1) = 0 \\[+2mm]
F(x,y) = \frac{x(x-y)}{1-x-y} &\implies \beta(1,1) = \log 2.
\end{align*}
\end{rem}

On a heuristic level, what is going on is that algebraic and analytic 
techniques can detect the neighbourhood rate better than they can detect 
delicate behaviors such as cancellation among sums of powers of algebraic
numbers. Algebraic cancellation\footnote{Topological cancellation,
which can also cause a drop in the limsup neighbourhood exponential rate,
is harder to study. See~\cite{BMP-lacuna} for one example of this.} 
can be detected using the following definition.

\begin{defn}[annihilating ideal] 
The \Em{annihilating ideal} of a flat $\sing_{k_1 , \dots , k_s}$ is
the polynomial ideal 
\begin{equation} \label{eq:Jideal} 
\mJ (L) := \langle \ell_{k_j}(\bz)^{p_{k_j}} : 
   j=1,\dots,s \rangle
\end{equation}
and the \Em{annihilating ideal} of a point $\bs$ on the stratum 
$\mS_{k_1,\dots,k_s}$ is $\mJ(\bs) := \mJ(\sing_{k_1,\dots,k_s})$.
\end{defn}

The annihilating ideal contains precisely the functions $G$ such that
the singular set of 
$\frac{G(\bz)}{\prod_{j=1}^m \ell_j(\bz)^{p_j}}$ does not contain 
$\mV_{k_1,\dots,k_s}$.  Indeed, if $G(\bz) \in \mJ(\bs)$ then 
$G(\bz)$ can be rewritten as a polynomial linear combination of the 
$\ell_{k_j}(\bz)^{p_{k_j}}$, meaning that
$F$ can be decomposed as a sum of meromorphic functions whose 
poles do not contain $L$.  For polynomials, the condition 
$G \in \mJ(\bs)$ can be verified automatically by, for instance, a 
Gr{\"o}bner basis computation. 

\begin{theorem} \label{th:ideal}
Consider a simple arrangement in a generic direction $\rr$,
let $\bs$ be a point in a stratum $S$ of codimension $s$ defined
by the vanishing of the linear functions $\ell_{k_1},\dots,\ell_{k_s}$,
and let $\mJ(\bs)$ be the annihilating ideal.  
\begin{enumerate}[(i)]
\item 
If $G \in \mJ(\bs)$ then the asymptotic series $\Phi_\bs (\rr)$ 
defined in Theorem \ref{thm:mainASM}
has exponential growth strictly less than $h_{\rr} (\bs)$.  
Hence, if all other critical points $\bs'$ with $h_{\rhat} (\bs')
\geq h_{\rhat} (\bs)$ are ruled out by Algorithm~\ref{alg:1} then
$$\overline{\beta} (\rhat) < h_{\rhat} (\bs) \, .$$
\item
If $G \notin \mJ (\bs)$ then $\Phi_\bs(\rr)$ has neighbourhood exponential 
growth at least $h_{\rr}(\bs)$.
In other words, if all critical points $\bs'$ with $h_{\rhat} (\bs')
> h_{\rhat} (\bs)$ are ruled out by Algorithm~\ref{alg:1} then
$$\obeta (\rhat) = h_{\rhat} (\bs) \, .$$
\end{enumerate}
\end{theorem}

\begin{proof}
First we prove $(i)$. If $G = \sum_{j=1}^s f_j \ell_j^{p_j}$
where the $f_j$ are locally analytic functions on a neighbourhood of $\sigma$
then we can write $F$ as a sum of meromorphic functions whose pole sets
no longer contain the stratum $S$. In generic directions this implies
that $\Phi_\bs(\rr)$ can be written as a sum of integrals over imaginary
fibers with basepoints arbitrarily close to critical points of 
lower height than $h_{\rr}(\bs)$, giving the stated result.

To prove $(ii)$ we make an analytic change of coordinates $\zz=(\uu,\vv)$ 
by setting $v_j = \ell_{k_j}(\zz)$ for $j=1,\dots,s$ and letting
$\uu$ be $d-s$ remaining coordinates parametrizing the flat $L$ containing $S$
near $\bss$. We can rewrite $G$ in these coordinates as
  \begin{equation}\label{eq:G_decomp}
G=\sum_{\mm\in\N^s : m_j < p_j} c_\mm(\uu)\vv^\mm + \sum_{j=1}^s v_j^{p_j}\tilde{c}_j(\uu,\vv),
  \end{equation}
where each $c_\mm(\uu)$ is an analytic function on $L$ and each 
$\tilde{c}_j(\uu)$ is an analytic function in a vicinity of $\bs$ in~$\Comp^d$.

Reasoning as in part $(i)$, the contribution of each term in the second
sum of~\eqref{eq:G_decomp} has an exponential rate strictly
less than $h_{\rhat} (\bs)$, while the first sum of~\eqref{eq:G_decomp} identically
vanishes if and only if $G$ lies in the ideal $\mJ(\bs)$.

Assume first that $\bpt = \bone$, so $\bs$ is a simple pole and
the first sum in \eqref{eq:G_decomp} contains 
only the term $c_\bzer(\uu)$. If $c_\bzer(\bzer)\neq0$ then the leading
term of $\Phi_\bs(\rr)$ given by Proposition~\ref{pr:partial} is nonzero,
and we have found the correct exponential rate $h_{\rhat}(\bs)$. Similarly, if
$c_\bzer(\uu)\neq0$ for some point $\uu$ in any sufficiently small neighbourhood
of the origin then, because the non-degenerate critical point $\bs(\br)$ 
on the stratum $S$ varies smoothly
with $\rr$, the neighbourhood exponential rate is $h_{\rhat}(\bs)$.
The only other case occurs when $c_\bzer$ vanishes in a neighbourhood
of zero, but then $c_\bzer$ is identically zero by analyticity, meaning 
$G \in \mJ(\bs)$.

Now let $\bpt$ be general and recall
Lemma \ref{lem:convo}, which states that multiplying $F$ by any 
Laurent polynomial $\bv^\mm$ does not increase 
its neighbourhood exponential growth.
If $c_\bzer$ is not identically zero then the argument from
the simple pole case implies that the neighbourhood exponential rate is 
$h_{\rhat}(\bs)$, and we are done. If $c_\bzer$ is identically zero then
multiplying $F$ by $\vv^{\bpt-\bone-\be_k}$,
where $\be_k$ has a 1 in coordinate $k$ and 0
elsewhere, gives a sum of rational functions which, in $(\uu,\vv)$ coordinates,
is $\frac{c_{\be_k}(\uu)}{v_1\cdots v_s}$ plus a sum of terms that each miss
at least one variable $v_j$ from their denominator, and thus have exponentially
smaller asymptotic growth. Applying the argument from the simple pole case 
to each $\frac{c_{\be_k}(\uu)}{v_1\cdots v_s}$ shows that if $c_{\mm}$
is not identically zero for some vector $\mm$ with $|\mm|=1$ then
the neighbourhood exponential rate is 
$h_{\rhat}(\bs)$. If all terms $c_{\mm}$ with $|\mm|=1$ are identically zero, 
then the argument can be repeated on the terms $c_{\mm}$ with $|\mm|=2$, and
so on until the result is established.
\end{proof}

\section{Non-Simple Arrangements in Generic Directions}
\label{sec:nonsimple-generic}

Any rational function whose singular set forms a non-simple hyperplane
arrangement can be decomposed into a finite sum of rational
functions, each of which has a singular set that defines a simple 
hyperplane arrangements. To make this explicit we use a known, canonical partial 
fraction expansion in terms of the
no-broken-circuit basis for the matroid $\{ \bb^{(1)} , \ldots , \bb^{(k)} \}$
for each flat with some codimension $\ell$ and $k \geq \ell$ hyperplanes.
We begin with an example.

\begin{example}
\label{ex:nonSimp}
Consider $F(x,y) = 1/(\ell_1\ell_2\ell_3)$ as shown in 
Figure~\ref{fig:nonSimp}, where
$$\ell_1(x,y) = 1-x/3-2y/3 \qquad\ell_2(x,y) = 1-2x/3-y/3 
   \qquad\ell_3(x,y) = 1-3x/5-2y/5 \, . $$

\begin{figure}
\centering
\includegraphics[width=0.8\linewidth]{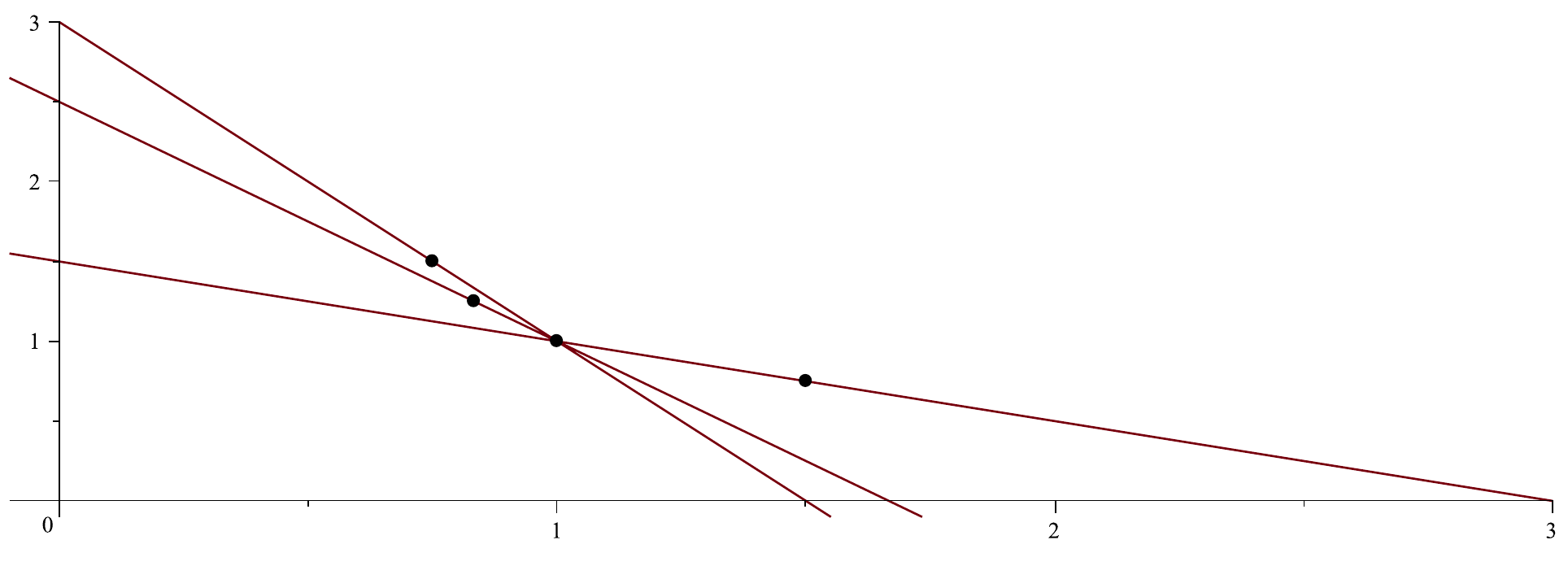}
\caption{Real singularities and contributing points of the rational 
function in Example~\ref{ex:nonSimp}. All three factors intersect at the 
point $(1,1)$, meaning $F$ is not simple.}
\label{fig:nonSimp}
\end{figure}

In the main diagonal direction $\br=(1,1)$, $F$ admits contributing points 
$(3/2,3/4),(3/4,3/2),$ and $(5/6,5/4)$ on the flats $\mV_1,\mV_2,$ and 
$\mV_3$, respectively, and contributing point $(1,1)$ which is the single 
point in all the flats $\mV_{1,2},\mV_{1,3},\mV_{2,3},$ and 
$\mV_{1,2,3}$. 
As the three lines $\ell_1,\ell_2,$ and $\ell_3$ in two dimensions have common 
intersection point $(1,1)$, they are linearly dependent. 
Basic linear algebra allows one to derive
$$ (1/5)\ell_1(x,y) + (4/5)\ell_2(x,y) =\ell_3(x,y) \, .$$
Dividing this equation by $\ell_3(x,y)$ gives
$$ (1/5)\frac{\ell_1(x,y)}{\ell_3(x,y)} 
   + (4/5)\frac{\ell_2(x,y)}{\ell_3(x,y)} = 1, $$
so that 
$$ F(x,y) = \frac{1}{\ell_1(x,y)\ell_2(x,y)\ell_3(x,y)} = 
\underbrace{\frac{1/5}{\ell_2(x,y)\ell_3(x,y)^2}}_{F_1} + 
\underbrace{\frac{4/5}{\ell_1(x,y)\ell_3(x,y)^2}}_{F_2}.
$$
Both $F_1(x,y)$ and $F_2(x,y)$ are simple, so we can apply 
Theorem~\ref{thm:mainASM} to obtain dominant asymptotics. For instance, 
in the direction $\br = (1,1)$ the function $F_2$ admits the 
point $(1,1)$ as its contributing singularity determining dominant 
asymptotics, while $(1,1)$ is not a contributing singularity of $F_1$. 
The contributing singularity of $F_1$ which determines dominant 
asymptotics is $(5/6,5/4)$, meaning the diagonal coefficients of $F_1$ 
decay exponentially. Thus, Theorem~\ref{thm:mainASM} gives
\begin{align*} 
[x^ny^n]F(x,y) &= (1/5)[x^ny^n]F_1(x,y) + (4/5)[x^ny^n]F_2(x,y) \\[+2mm]
&= \frac{15n}{4}\left(1+O\left(\frac{1}{n}\right)\right).
\end{align*}
\end{example}

To generalize this approach, we introduce some definitions.  The {\bf support} 
of a rational function is the set of divisors appearing 
in the denominator when in lowest terms.  In other words, 
$$ \sup\left(\frac{G(\bz)}{\ell_1(\bz)^{q_1} \cdots \ell_m(\bz)^{q_m}}\right) 
   := \left\{\ell_j(\bz) : q_j > 0 \right\}$$
when no $\ell_j$ divides $G$.

As in the theory of matroids, we call any minimal linearly dependent 
set $\{\ell_{i_1}, \dots,\ell_{i_s}\}$ a {\bf circuit}. 
A {\bf broken circuit} is the independent collection obtained from any 
circuit by removing the element $\ell_j$ with largest index $j$.  
A collection is said to be {\bf $\chi$-independent} if it does not 
contain a broken circuit; note that any $\chi$-independent set is also 
linearly independent. 

In Example~\ref{ex:nonSimp} the set $\{\ell_1,\ell_2,\ell_3\}$ is the only 
circuit, so $\{\ell_1,\ell_2\}$ is the only broken circuit. Note that the 
supports of the two rational functions which we analyze, $\ell_1\ell_3^2$ and 
$\ell_2\ell_3^2$, correspond to $\chi$-independent sets $\{\ell_1,\ell_3\}$ and 
$\{\ell_2,\ell_3\}$. This is no coincidence: 
Proposition~\ref{prop:SimpDecompAlg} below shows that one can always 
decompose a rational function into a sum of rational functions whose 
supports are $\chi$-independent, and the following shows in general one 
can make no further simplifications.

\begin{proposition} \label{prop:SimpDecompAlg} 
Let $\ell_1, \ldots , \ell_m$ be any $m$ linear functions of 
$d$ variables.
\begin{enumerate}[(i)]
\item The set of rational functions 
$$\left\{ \frac{1}{\ell_{i_1}(\bz) \cdots\ell_{i_s}(\bz)} : 
   \{\ell_{i_1},\dots,\ell_{i_s}\} \text{ is $\chi$-independent} \right\} $$
is linearly independent over $\C$.
\item The span over $\C$ of the rational functions
$$\left\{ \frac{1}{\ell_{i_1}(\bz)^{p_1} \cdots \ell_{i_s}(\bz)^{p_s}} : 
   \{\ell_{i_1},\dots,\ell_{i_s}\} \text{ is $\chi$-independent} 
   \mbox{ and } \sum_{i=1}^s p_i = M \right\} $$
contains the inverses of all products of the $\ell_j$ over multisets 
of cardinality $M$.
\item
Algorithm~\ref{alg:SimpleDecomp} terminates in a finite number of steps 
and outputs a sum of rational functions whose supports contain no broken 
circuits.
\end{enumerate}
\end{proposition}

\begin{proof}
The first conclusion follows from standard results on hyperplane 
arrangements; see Orlik and 
Terao~\cite[Theorems 3.43, 3.126,  5.89]{OrlikTerao1992} or Pemantle 
and Wilson~\cite[Prop. 10.2.10]{PW-book}. 

To prove the second, given any linear dependence 
$$ 0 = a_1\ell_{i_1}(\bz) + \cdots + a_s\ell_{i_s}(\bz) $$
with each $a_i \neq 0$, one can divide by any $a_jl_{i_j}(\bz)$ and then 
by some product of all $\ell_k(\bz)$ to obtain
\begin{equation}
\frac{1}{\bl^{\bq}} = \frac{1}{\ell_1(\bz)^{q_1} \cdots\ell_m(\bz)^{q_m}} = 
\sum_{k \neq j} \frac{(-a_k/a_j)}{\bl^{\bq+\be^{(i_j)}-\be^{(i_k)}}}
\label{eq:linDependence}
\end{equation}
for any $\bq \in \N^m$, where $\be^{(\kappa)}$ is the elementary 
basis vector with a one in the $\kappa$th coordinate and all other 
coordinates zero. 

Given any rational function $F(\bz) = G(\bz)/\bl^{\bpt}$ whose support 
contains a broken circuit $\{i_1,\dots,i_s\}$ with indices in increasing 
order, one can apply the base exchange equation~\eqref{eq:linDependence} 
with any $i_j > i_s$ such that $\{i_1 , \dots , i_s , i_j\}$ 
is linearly dependent; such an $i_j$ exists by the definition of 
a broken circuit.  At each step, every vector $\bq+e_{i_j}-e_{i_k}$ on 
the right hand side of Equation~\eqref{eq:linDependence} has smaller 
lexicographical order than $\bq$.  Therefore, repeating the process 
formalized in Algorithm~\ref{alg:SimpleDecomp}, one must eventually 
arrive at an expression for $F$ as a sum of rational functions whose 
supports contain no broken circuits.   
The second while loop terminates because any time a summand 
$\tilde{F}$ is further decomposed into several summands, the supports of 
these new rational functions are subsets of the support of $\tilde{F}$, 
and the new functions have denominators of smaller total degree than the 
denominator of $\tilde{F}$. 
This proves part~$(iii)$, and thereby, part~$(ii)$.  
\end{proof}

\begin{example} \label{ex:nbc}
If
$$ \ell_1 = 1 - 2x-2y+2z, \qquad\ell_2 = 1-2x, \qquad\ell_3 = 1-2y, 
   \qquad\ell_4 =  1-2z, \qquad\ell_5 = 1-x-y  $$
then the circuits are the sets
$$ \{\ell_1,\ell_2,\ell_3,\ell_4\}, \quad \{\ell_1,\ell_4,\ell_5\}, 
   \quad \{\ell_2,\ell_3,\ell_5\}, $$
and the broken circuits are the sets
$$ \{\ell_1,\ell_2,\ell_3\}, \quad \{\ell_1,\ell_4\}, 
   \quad \{\ell_2,\ell_3\} \, . $$
The collection of $\chi$-independent sets is 
$$ \varnothing $$ 
$$ \{\ell_1\} \quad \{\ell_2\} \quad \{\ell_3\} \quad \{\ell_4\} 
   \quad \{\ell_5\} $$
$$ \{\ell_1,\ell_2\} \quad \{\ell_1,\ell_3\} \quad \{\ell_1,\ell_5\} 
   \quad \{\ell_2,\ell_4\} 
   \quad \{\ell_2,\ell_5\} \quad \{\ell_3,\ell_4\} \quad \{\ell_3,\ell_5\} 
   \quad \{\ell_4,\ell_5\} $$
$$ \{\ell_1,\ell_2,\ell_5\} \quad \{\ell_1,\ell_3,\ell_5\} 
    \quad \{\ell_2,\ell_4,\ell_5\} \quad \{\ell_3,\ell_4,\ell_5\} \, . $$
Consider the rational function
$$ F(x,y,z) = \frac{1}{\ell_1(x,y)l_2(x,y)l_3(x,y)l_4(x,y)l_5(x,y)^2} \, .$$ 
Running Algorithm~\ref{alg:SimpleDecomp} gives the decomposition
\begin{align*}
F(x,y,z) 
= &\quad \frac{1}{4\ell_2(x,y)\ell_4(x,y)\ell_5(x,y)^4} 
+ 
\frac{1}{4\ell_3(x,y)\ell_4(x,y)\ell_5(x,y)^4}  \\
&+ 
\frac{1}{4\ell_1(x,y)\ell_2(x,y)\ell_5(x,y)^4} 
+ 
\frac{1}{4\ell_1(x,y)\ell_3(x,y)\ell_5(x,y)^4}.
\end{align*}
Each of these summands is now simple, and we may use our Maple implementation 
of the results in Section~\ref{sec:simplegeneric} to find asymptotics in 
direction $\br=(1,2,3)$, say. The summands 
have dominant coefficient asymptotics 
\[\left(\frac{81\sqrt{42}}{224}\right) 54^n n^{5/2}, \qquad (2/3) \, 
64^n n^3 , \qquad \frac{-(-27648)^n}{5\pi n}, \qquad 
\frac{-(-27648)^n}{4\pi n},\]
respectively, so that 
\[ [x^ny^{2n}z^{3n}]F(x,y,z) = \frac{-9(-27648)^n}{20\pi n}
\left(1 + O\left(\frac{1}{n}\right)\right). \]
\end{example}

Part~$(i)$ of Theorem~\ref{th:ideal} continues to hold for 
non-simple arrangements.  In particular, there is an ideal $\mJ$ such that
$G \in \mJ$ implies $G / \prod_{j=1}^m \ell_j^{m_j}$ decomposes into
a sum of rational functions whose denominators have poles on
at most $k-1$ hyperplanes, $k$ being the co-dimension of the
flat containing a given critical point.  Hence for $G \in \mJ$, 
$G/\prod_{j=1}^m \ell_j^{m_j}$ will have neighbourhood exponential rate
$\obeta (\rhat) < h_{\rhat} (\bs)$, provided no other critical
point produces a contribution with exponential rate $h_{\rhat} (\bs)$.
The converse, however, fails for non-simple arrangements,
as shown by the following example.

\begin{example} \label{eg:4 lines}
Let $\ell_j = y - 1 - \lambda_j (x-1)$ for $1 \leq j \leq 4$ 
be four real lines through $\bs=(1,1)$, shown in Figure~\ref{fig:4 lines},
whose normals have positive slopes $\nu_j := -1/\lambda_j$ that 
decrease in $j$. 

\begin{figure}
\centering
\includegraphics[height=2in]{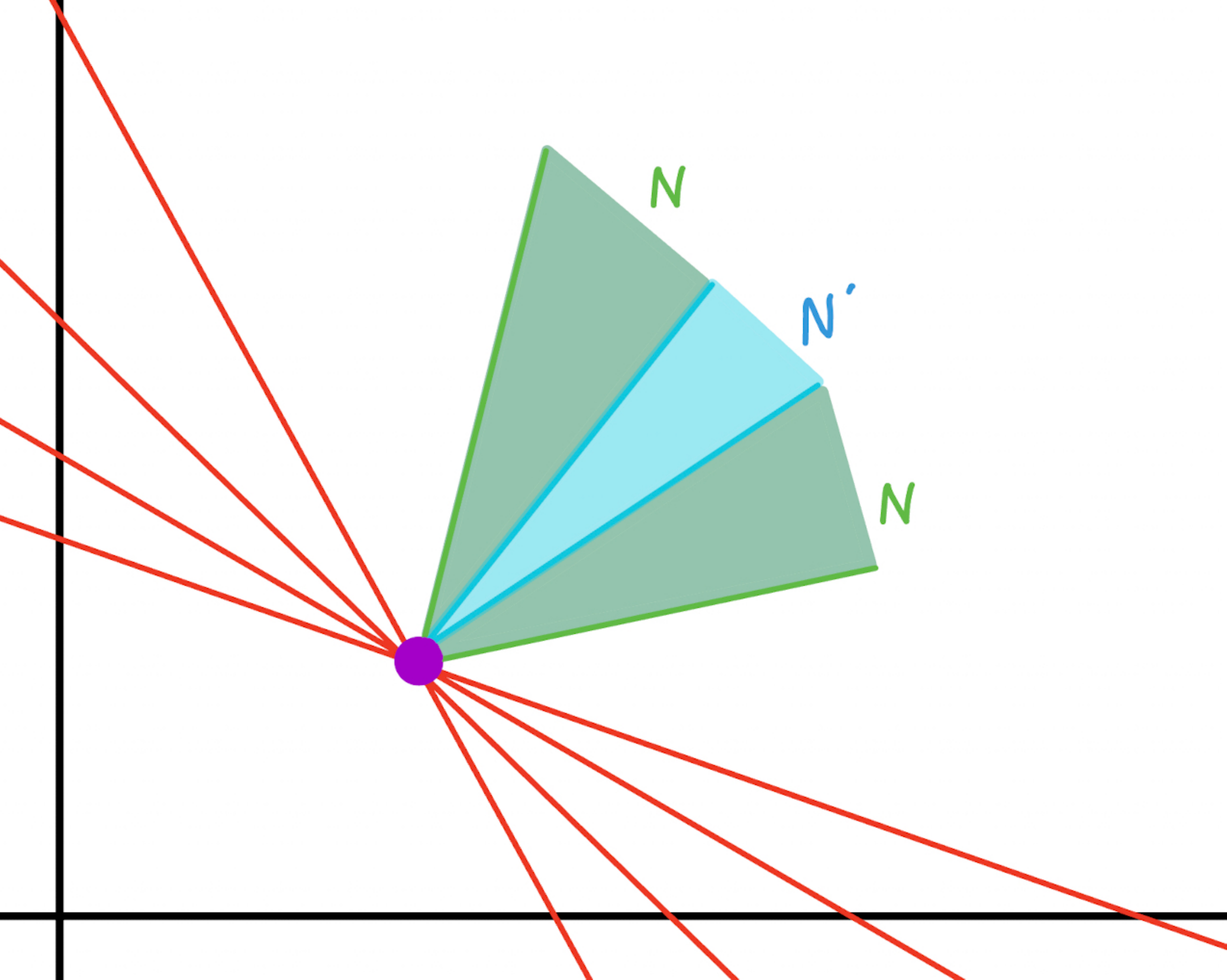}
\caption{A non-simple arrangement where the exponential rate drops within
a hole $N'$ of the normal cone~$N$ at $\bs=(1,1)$.}
\label{fig:4 lines}
\end{figure}

The normal cone at $\bs$ contains the lines with slopes from
$\nu_4$ to $\nu_1$.  Let $G(x,y) := \ell_1(x,y) 
\ell_2(x,y) + \ell_3(x,y) \ell_4(x,y)$, so
$$\frac{G(x,y)}{\prod_{j=1}^4 \ell_j(x,y)} = \frac{1}{\ell_1(x,y) \ell_2(x,y)} + 
   \frac{1}{\ell_3(x,y) \ell_4(x,y)} \, .$$
As we know from the simple case, this has an exponential rate of zero 
in directions with slopes in the two intervals $[\nu_1 , \nu_2]$ 
and $[\nu_3 , \nu_4]$, where $\bs$
is the contributing singularity determining asymptotics, and
a negative exponential rate elsewhere.  In particular, 
the exponential rate is negative along directions $\rhat$ with slopes in 
$(\nu_2 , \nu_3)$, which are generic and have 
$\obeta (\rhat) < 0 = h_{\rhat} (\bs)$.  
On the other hand, the numerator 
$\ell_1 \ell_3 + \ell_2 \ell_4$ leads to two non-canceling
contributions rather than to none.  Thus, any algebraic
test for the numerator to create a drop in the neighbourhood exponential
rate must take into account specific normal cones.  While this may
seem trivial in low dimensions, computation of the chamber decomposition 
of the normal cone at $\bs$ is in general a high complexity computation, 
and the description of the set of numerators for which the neighbourhood
exponential rate drops is correspondingly difficult to compute.
\end{example}

\section{Simple Arrangements in Non-Generic Directions}
\label{sec:simple-nongeneric}
We now discuss the analysis in non-generic directions.
First, we see how asymptotics behave in an exact non-generic
direction. Then, we study transitions in asymptotic
behaviour around non-generic directions. 

\subsection{Exact non-generic directions} \label{ss:exact}
Fix a direction $\brhat$ and suppose $F$ admits a unique contributing singularity 
$\bs$ of highest height, contained in the stratum $\mS_{1,\dots,t}$. 
Suppose $\bs$ is a non-generic direction such that
\begin{equation} (-\nabla h)(\bs) = a_1\bb^{(1)} + \cdots 
+ a_s\bb^{(s)} + 0\cdot\bb^{(s+1)} + \cdots + 0\cdot\bb^{(t)} 
\label{eq:nonGenGradSetup} \end{equation}
for some $s<t$, where each $a_j\neq0$ (and thus $a_j>0$ by the 
definition of contributing point). We work in the 
simple case, so $\bb^{(1)},\dots,\bb^{(t)}$ are linearly 
independent, and up to permuting coordinates we may assume that the matrix
\[ M = \begin{pmatrix} \bb^{(1)} \\ \vdots \\ \bb^{(t)} \\ 
\be^{(t+1)} \\ \vdots \\ \be^{(d)} \end{pmatrix} \]
has full rank, where $\be^{(j)}$ denotes the $j$th elementary basis vector. 

Because $\bs$ is minimal and has highest height, 
asymptotic growth of $[\bz^{\br}]F(\bz)$ is determined, up to an 
exponentially negligible error term, by the integral
\[ I = \frac{1}{(2\pi i)^d}\int_{\bs-\epsilon\bm+i\R^d} 
\frac{\tG(\bz)}{\ell_1(\bz)^{p_1}\cdots\ell_t(\bz)^{p_t}} 
\, \frac{d\bz}{\bz^{\br}}, \]
where $\bm = M^{-1}\left(\be^{(1)}+\cdots+\be^{(t)}\right)$ and 
\[ \tG(\bz) = \frac{G(\bz)}{(z_1\cdots z_d)
\prod_{j>t}\ell_j(\bz)^{p_j}}. \]
Making the substitution $\bw = M(\bs-\bz)$ gives
\begin{equation} I = \frac{1}{|\det M|(2\pi i)^d}
\int_{\epsilon(\bone_t,\bzer)+i\R^d} 
\frac{\tG(\bs-M^{-1}\bw)}{w_1^{p_1}\cdots w_t^{p_t}\left(\bs-M^{-1}\bw\right)^{\br}}
\, d\bw, \label{eq:nonGenOrigin} \end{equation}
where $\bone_t$ denotes the $t$-dimensional all ones vector. 
Since the $a_j$ in~\eqref{eq:nonGenGradSetup} are positive,
replacing the domain of integration in~\eqref{eq:nonGenOrigin}
by any of the $2^s-1$ imaginary fibers with 
basepoints $(\pm\bone_{s},\bone_{t-s},\bzer)$ not equal to 
$(\bone_t,\bzer)$ results in
an integral of exponentially smaller growth\footnote{If $\brhat$ were
a generic direction, we would be able to add all $2^t$ fibers 
and use univariate residues will get rid of all $w_j$ in the integrand 
denominator.}. Taking a signed sum of these integrals thus introduces
an exponentially negligible error and results 
in a $(d-s)$-dimensional integral obtained by taking
univariate residues in the variables $w_1,\dots,w_s$ at the origin.  

We can be most explicit when $\bpt=\bone$, when the residues of
$w_1,\dots,w_s$ at the origin are obtained by removing the factor
$(w_1\cdots w_s)$ from the denominator and setting these variables
equal to zero in the remaining expression. If 
$\bpt=\bone$ and we make the change of variables 
$y_j=i w_{s+j}$ for $j=1,\dots,d-s$, 
then dominant asymptotics of $[\bz^{\br}]F(\bz)$ are 
given by the integral
\begin{equation} 
I' := \frac{(-1)^{t-s}}{|\det M|(2\pi)^{d-s} \, i^{t-s}}
\int_{\R^{d-s}+i\epsilon(\bone_{t-s},\bzer)} 
\frac{\tG\left(\bs+iM^{-1}\begin{psmallmatrix}\bzer\\\by\end{psmallmatrix}\right)}
{y_1\cdots y_{t-s}}\;e^{-\br \cdot \log\left(\bs+iM^{-1}
\begin{psmallmatrix}\bzer\\\by\end{psmallmatrix}\right)}d\by, 
\label{eq:nonGenAfterRes} \end{equation}
where $\by=(y_1,\dots,y_{t-s})$. For general $\bpt$, 
repeated differentiation shows~\cite[Theorem 10.2.6]{PW-book}
that the residue has the form
\[ \left.\frac{\tG(\bs-M^{-1}\bw)}{w_{s+1}^{p_{s+1}}\cdots 
w_t^{p_t}\left(\bs-M^{-1}\bw\right)^{\br}} 
\; P(\br,\bw) \right|_{w_j=0,1\leq j\leq s}, \]
where $P$ is a multivariate polynomial in $\br$ 
of degree $p_1+\cdots+p_s-s$ having leading term
\[ R(\bw)=\prod_{k=1}^s \frac{1}{(p_k-1)!}\left( \sum_{j=1}^d 
\frac{r_jM^{-1}_{jk}}{(\bs-M^{-1}\bw)_j} \right)^{p_k-1} .\]
Thus,
\begin{align} 
I &=  \frac{1}{|\det M|(2\pi i)^{d-s}\prod_{k=1}^s(p_k-1)!}
\int_{\epsilon(\bone_{t-s},\bzer)+i\R^{d-s}} 
\left.\frac{\tG(\bs-M^{-1}\bw)R(\bw)}{w_{s+1}^{p_{s+1}}\cdots 
w_t^{p_t}\left(\bs-M^{-1}\bw\right)^{\br}} 
\; \right|_{w_j=0,1\leq j\leq s}d\bw
\left(1+O\left(\frac{1}{|\br|}\right)\right) \notag \\[+2mm]
&=  \frac{i^{p_{s+1}+\cdots+p_t}}{|\det M|(2\pi)^{d-s}\prod_{k=1}^s(p_k-1)!}
\int_{\R^{d-s}+i\epsilon(\bone_{t-s},\bzer)} 
\frac{\tG\left(\bs+iM^{-1}\begin{psmallmatrix}\bzer\\\by\end{psmallmatrix}\right)
R(\bzer,-i\by)}{y_{s+1}^{p_{s+1}}\cdots 
y_t^{p_t}
\left(\bs+iM^{-1}\begin{psmallmatrix}\bzer\\\by\end{psmallmatrix}\right)^{\br}}d\by 
\left(1+O\left(\frac{1}{|\br|}\right)\right),
\label{eq:nonGenAfterResP}
\end{align}
where $\by=(y_1,\dots,y_{t-s})$ and $|\br|=r_1+\cdots+r_d$.

\begin{example}
\label{ex:Ex1NonGen}
Consider again the function 
\[ F(x,y) = \frac{1}{\left(1 - \frac{2x+y}{3}\right)\left(1 - 
\frac{x+2y}{3}\right)}. \]
The sequence
\begin{equation} [x^{2n}y^{n}]F(x,y) = \frac{1}{(2\pi i)^2} \int_{\mT} 
\frac{1}{\left(1 - \frac{2x+y}{3}\right)\left(1 - \frac{x+2y}{3}\right)} 
\, \frac{dxdy}{x^{2n+1}y^{n+1}} \label{eq:nonGenEx} \end{equation}
is in the non-generic direction $\rr=(2,1)$, meaning its asymptotic 
behaviour can be quantitatively different than what happens when $1/3<\hat{r}_1 
< 2/3$ (limit to a constant with exponential error) and $2/3 < \hat{r}_1 < 1$ 
(limit to zero exponentially with polynomial error).  The highest contributing 
singularity is now $\bs=\bs_{1,2}=(1,1)$ on the stratum $\mS_{1,2}$, 
which now also coincides with the critical point $\bs_1$ on the 
flat $\sing_1$.

Since there are
no bounded components of $\mMR$ in any quadrant except the positive quadrant,
the Cauchy integral in~\eqref{eq:nonGenEx} equals 
\begin{align*} 
[x^{2n}y^n]F(x,y) &= \frac{1}{(2\pi i)^2} \int_{(\epsilon,\epsilon)+i\R^2} 
\frac{1}{\left(1 - \frac{2x+y}{3}\right)\left(1 - \frac{x+2y}{3}\right)} 
\, \frac{dxdy}{x^{2n+1}y^{n+1}} \\
&= \frac{1}{(2\pi i)^2} \int_{(1-\epsilon,1-\epsilon)+i\R^2} 
\frac{1}{\left(1 - \frac{2x+y}{3}\right)\left(1 - \frac{x+2y}{3}\right)} 
\, \frac{dxdy}{x^{2n+1}y^{n+1}}.
\end{align*}
Making the change of variables described above shows
\[
[x^{2n}y^{n}]F(x,y) = \frac{3}{(2\pi i)^2} \int_{(\epsilon,\epsilon) + 
i\R^2}  \frac{1}{w_1w_2} \, \frac{dw_1dw_2}{(1-2w_1+w_2)^{2n+1}(1+
w_1-2w_2)^{n+1}}.\]
If
\[ \tilde{\omega} = \frac{1}{w_1w_2} \, \frac{dw_1dw_2}{(1-2w_1+w_2)^{2n+1}(1+
w_1-2w_2)^{n+1}}, \]
then 
\[[x^{2n}y^{n}]F(x,y) = \frac{3}{(2\pi i)^2} \left[ 
\int_{(\epsilon,\epsilon) + i\R^2}  \tomega - \int_{(-
\epsilon,\epsilon) + i\R^2}  \tomega \right] + O(\tau^n)\]
for some $\tau \in (0,1)$. Taking the residue in $w_1$ at the origin then yields
\begin{align*} 
[x^{2n}y^n]F(x,y) &= \frac{3}{2\pi i} \int_{\epsilon + i\R}  
\frac{1}{w_2} \, \frac{1}{(1+w_2)^{2n+1}(1-2w_2)^{n+1}}dw_2 + O(\tau^n) \\[+2mm]
&= \frac{-3}{2\pi i} \int_{\R+i\epsilon}  \frac{1}{y(1+iy)(1-2iy)} 
e^{-n\left[2\log(1+iy) + \log(1-2iy)\right]}dy + O(\tau^n).
\end{align*}
\end{example}

The idea behind an asymptotic analysis is to replace the analytic functions
in the integrand under consideration with the leading terms of their
power series expansions at the origin. Doing this rigorously
requires extending the asymptotic bounds of~\cite[Chapter 5]{PW-book}
to integrals which are singular at the origin due to a division
by monomial terms. When $(-\nabla h_{\br})(\bs)$ lies on a $k$-dimensional facet 
of $N(\bs)$, asymptotics of the dominant term in the residue integral
are obtained by analyzing an integral of the form
\[ \int_{\R^k + i\bepsilon} \, \by^{\cc} \cdot e^{-\by^t A \by} 
d\by, \]
where $A$ is a $k \times k$ positive-definite matrix and $\cc \in 
\Z^k$.  This can be viewed as a Gaussian \emph{negative}-moment 
integral; although positive moments ($\cc \in \N^k$) of the 
Gaussian distribution can be determined through the Wick-Isserlis 
Theorem~\cite[Theorem 3.2.5]{LandoZvonkin2004}, such negative moments do 
not seem to have received previously attention in the literature.  

We postpone the general discussion of this theory to future
work and focus here on the $s=d-1$ codimension 1 case, where it is easy
to be explicit. In this setting, asymptotics boil down to 
an analysis of the integral
\[ \int_{\R+i\epsilon}  \frac{e^{-at^2}}{t^r} dt, \]
which is characterized by the following result.

\begin{proposition} \label{prop:1DNegGauss}
Let $r$ be a positive integer and $a\geq0$.  Then
\[ \int_{\R+i\epsilon}  \frac{e^{-at^2}}{t^r} dt = \frac{(-
i)^ra^{(r-1)/2} \pi}{\, \Gamma\left(\frac{r+1}{2}\right)} \, .\]
\end{proposition} 

\begin{proof}
Suppose that $r$ is odd and $a=1$.  Then as the integrand under consideration 
is odd, Cauchy's Integral Theorem implies
\[ \int_{\R+i\epsilon}  \frac{e^{-t^2}}{t^r} dt
= -\int_{-\R+i\epsilon}  \frac{e^{-t^2}}{t^r} dt 
= -\int_{\mC} \frac{e^{-t^2}}{t^r} dt, \]
where $\mC$ is the positively oriented unit half-circle above the 
$x$-axis.  As $\mC$ is a compact domain of integration, Fubini's Theorem 
implies that we can expand the integrand as a power series, integrate 
each term using an explicit parametrization, and sum to obtain  
\[ \int_{\mC}  \frac{e^{-at^2}}{t^r} dt = \frac{(\pi i)(-1)^{(r-
1)/2}}{\left((r-1)/2\right)!} = \frac{(-1)^{r+1}i^ra^{(r-1)/2} \pi}{\, 
\Gamma\left(\frac{r+1}{2}\right)} \, . \]
The case when $r$ is odd and $a$ is arbitrary is reduced to this case by 
a linear change of variables. 

Now suppose that $r$ is even, and let $J_r(a) := \int_{\R+i\epsilon}  
\frac{e^{-at^2}}{t^r} dt.$ Differentiating under the integral sign with 
respect to $a$ implies
\[ (\partial J_r/\partial a)(a) = \int_{\R+i\epsilon}  -
\frac{e^{-at^2}}{t^{r-2}} dt = -J_{r-2}(a) \]
for $r \geq 2$, while
\[  J_r(0) = \int_{\R+i\epsilon}  \frac{1}{t^r} dt = 0 \]
when $r \geq 2$ and  
\[ J_0(a) = \int_{\R+i\epsilon} e^{-at^2} dt = \sqrt{\pi/a}. \]
Solving this recurrence implies
\[  J_r(a) = \frac{(-1)^{r/2}a^{(r-1)/2}\sqrt{\pi}}{(r/2-1/2)(r/2-3/2) 
\cdots (1/2)} = \frac{(-1)^ri^ra^{(r-1)/2} \pi}{\, 
\Gamma\left(\frac{r+1}{2}\right)} \,, \]
as the difference equation satisfied by the gamma function shows 
\[\Gamma\left(n + \frac{1}{2}\right) = \frac{(2n-1)(2n-3) \cdots 
(3)(1)}{2^n}\Gamma\left(\frac{1}{2}\right) = \frac{(2n-1)(2n-3) \cdots 
(3)(1)}{2^n}\sqrt{\pi}\] 
for any non-negative integer $n$.
\end{proof}

This allows for the following asymptotic determination.

\begin{proposition}
\label{prop:nonGenDir}
Suppose $F(\bz)$ is simple and $\br$ is a non-generic direction with a 
unique contributing singularity of maximal height where $G(\bs) \neq 0$. 
If $(-\nabla h_{\br})(\bs)$ lies on the codimension 1 face of $N(\bs)$ 
with
\[ (-\nabla h_{\br})(\bs) = a_1 \cdot \bb^{(1)} + \cdots + a_{d-1} \cdot 
\bb^{(d-1)} + 0 \cdot \bb^{(d)} \]
for $\bs$ in the stratum defined by $\ell_1=\cdots=\ell_d=0$, then as 
$\br\rightarrow\infty$,
\[ [\bz^{\br}]F(\bz) = \bs^{-\br} \cdot \left(C + O\left(\frac{1}{n}\right)\right), \]
where 
\[ C = \prod_{j=1}^{d-1} \frac{\left(\begin{pmatrix}\frac{r_1}{\sigma_1} 
& \cdots & \frac{r_d}{\sigma_d} \end{pmatrix} \cdot M^{-1}\right)_j^{p_j-
1}}{(p_j-1)!} \cdot \frac{G(\bs)}{\prod_{j>d}\ell_j(\bs)^{p_j}} \cdot 
\frac{(\bq^T\bq/2)^{(p_d-1)/2}}
{2(\sigma_1 \cdots \sigma_d) \, |\det M| \, \Gamma\left(\frac{p_d+1}{2}\right)}, \]
$M$ is the matrix with rows $\bb^{(1)},\dots,\bb^{(d)}$, and $\bq$ is
the rightmost column of the matrix
\[ 
Q = \begin{pmatrix} 
\sqrt{r_1}/\sigma_1 & 0 & \bzer & 0 \\[+1mm]
0 & \sqrt{r_2}/\sigma_2 & \bzer & 0 \\[+1mm]
\bzer & \bzer & \ddots & \bzer \\[+1mm]
0 & 0 & \bzer & \sqrt{r_d}/\sigma_d
\end{pmatrix} M^{-1}.
\]

If $\br=n\brhat$ as $n\rightarrow\infty$ then the order
of growth of $[\bz^{\br}]F(\bz)$ is $\bs^{-\br}n^{p_1+\cdots+p_{d-1}+p_{d}/2-
d+1/2}$. In the simple pole case, when $\bpt=\bone$, the leading constant here is 
half what it would be if $\br$ was generic (i.e., if $(-\nabla 
h_{\br})(\bs)$ was in the interior of $N(\bs)$).
\end{proposition}

\begin{proof}
In this one-dimensional case,
the (now univariate) integral in~\eqref{eq:nonGenAfterResP} equals
\[ I' = 
\int_{\R+i\epsilon} 
\frac{\tG\left(\bs+iM^{-1}\begin{psmallmatrix}\bzer\\y\end{psmallmatrix}\right)}{y_d^{p_d}
\left(\bs+iM^{-1}\begin{psmallmatrix}\bzer\\y\end{psmallmatrix}\right)^{\br}}
\prod_{k=1}^{d-1} \left( \sum_{j=1}^d 
\frac{r_jM^{-1}_{jk}}{(\bs+M^{-1}\begin{psmallmatrix}\bzer\\y\end{psmallmatrix})_j} \right)^{p_k-1} dy.
\]
Note that 
\[ I' = 
\int_{\R+i\epsilon} \frac{A(y)}{y^{p_d}}e^{-\psi(y)}dy,
\]
where $A(y)$ and $\psi(y)$ are analytic functions at the origin with 
\begin{align*} 
A(y) &= \underbrace{\tG(\bs)\prod_{j=1}^{d-1} \left(\begin{pmatrix}\frac{r_1}{\sigma_1} 
& \cdots & \frac{r_d}{\sigma_d} \end{pmatrix} \cdot M^{-1}\right)_j^{p_j-
1}}_{C_1} \,+\, O(y) \\
\psi(y) &= \br\cdot\log(\bs+M^{-1}\begin{psmallmatrix}\bzer\\y\end{psmallmatrix})= 
\log(\bs) + \underbrace{(\bq^T\bq/2)}_{C_2}y^2 + O(y^3).
\end{align*}
The desired result is given by an application of the following lemma. 
\end{proof}

\begin{lemma}
\label{lem:BigONegGauss}
For some $\epsilon>0$ and $k\in\N$ let 
\[ I = \int_{\R+i\epsilon} y^{-k}A(y)e^{-n\psi(y)}, \]
where
\begin{itemize}
	\item $A,\psi$ are analytic in $\R+(-2\epsilon,2\epsilon)i \subset\C$;
	\item $A(y) = a_0 + O(y)$ and $\psi(y) = b_2y^2+O(y^3)$ at 0;
	\item $\psi'(y)=0$ implies $y=0$;
	\item $\Re(\psi)\geq0$ with equality only
	when $y=0$, and $|\Re(\psi)|\rightarrow\infty$ as $|y|\rightarrow\infty$.
\end{itemize}
Then 
\[ I \sim a_0\int_{\R+i\epsilon}y^{-k}e^{-nb_0y^2}dy = \frac{a_0(-i)^k(nb_0)^{(k-1)/2}}{\Gamma(\frac{k+1}{2})}. \]
\end{lemma}

\begin{proof}
For $n>0$ let $\mC_1 = (-n^{-2/5},n^{-2/5})$ and $\mC_2 = \R\setminus\mC_1$, so that
\[ I = \int_{\mC_1+i\epsilon} y^{-k}A(y)e^{-n\psi(y)}dy 
+ \int_{\mC_2+i\epsilon} y^{-k}A(y)e^{-n\psi(y)}dy. \]
The hypotheses on the 
real part of $\psi$ imply that the integrand under 
consideration decays exponentially as $y$ grows, meaning the value of
$I$ is independent of $\epsilon>0$. For the rest of this proof we 
take $\epsilon=\epsilon(n)=1/\sqrt{n}$.

\emph{Step 1 (Prune Tails).} 
Our first goal is to show that the integral over $\mC_2+i\epsilon$ is negligible.
To do this we bound the integral over $(n^{-2/5},\infty)+i\epsilon$, with the 
integral over $(-\infty,-n^{-2/5})+i\epsilon$ bounded with an analogous argument.
Since $\psi'(y)=0$ has no solution other than $y=0$, 
repeatedly integrating by parts shows that for any fixed $a>0$
\[ \int_{(a,\infty)+i\epsilon} y^{-k}A(y)e^{-n\psi(y)}dy \]
decays to zero faster than any negative integer power of $n$
(see, for instance, the argument in Stein~\cite[Section VIII.1.1]{Stein1993}).
Because of the power series expansions of $A$ and $\psi$ at the
origin, there exists $\sigma>0$ and $K_1,K_2>0$ such that for
all $|t|\leq\sigma$ and $n$ sufficiently large,
\[ |A(t+i\epsilon)|\leq K_1 \quad \text{and}\quad 
\Re(\psi(t+i\epsilon)) \geq K_2\Re((t+i\epsilon)^2) = K_2(t^2-\epsilon^2). \]
When $n$ is sufficiently large then $n^{-2/5}<\sigma$, and we see that
\[ \left| \int_{(n^{-2/5},\infty)+i\epsilon} y^{-k}A(y)e^{-n\psi(y)}dy 
- \int_{(n^{-2/5},\sigma)+i\epsilon} y^{-k}A(y)e^{-n\psi(y)}dy \right|\rightarrow0 \]
faster than any negative integer power of $n$.
Furthermore, with $\epsilon=1/\sqrt{n}$ we have
\[ \left|\int_{(n^{-2/5},\sigma)+i\epsilon} y^{-k}A(y)e^{-n\psi(y)}dy\right|
\leq K_1 \int_{n^{-2/5}}^{\sigma} \left|t+i\epsilon\right|^{-k}e^{-nK_2(t^2-\epsilon^2)}dt
\leq K_1e^{K_2} \int_{n^{-2/5}}^{\sigma} t^{-k}e^{-nK_2t^2}dt,
\]
and making the change of variables $w=nK_2t^2$ implies
\begin{align*}
\int_{n^{-2/5}}^{\sigma} t^{-k}e^{-nK_2t^2}dt 
&=O\left(n^{(k-1)/2}\int_{K_2n^{1/5}}^{\infty} w^{-(k+1)/2}e^{-w}dw \right)\\
&=O\left(n^{(k-1)/2 - (k+1)/10} \int_{K_2n^{1/5}}^{\infty} e^{-w}dw \right) \\
&= O\left(n^{(2k-3)/5} e^{-K_2n^{1/5}} \right).
\end{align*}
Putting everything together, and repeating the argument for the elements of
$\mC_2$ which are less than zero, we have shown that 
\[ I \sim \int_{\mC_1+i\epsilon} y^{-k}A(y)e^{-n\psi(y)}dy, \]
up to an error that decays faster than any negative integer power of $n$.

\emph{Step 2 (Approximate).} We now show that $A$ and $\psi$ can be replaced 
by their leading power series terms while introducing an error
which will not affect asymptotics. Note that 
\[ y^{-k}A(y)e^{-n\psi(y)} = a_0y^{-k}e^{-nb_0y^2}\left(1+O(y+ny^3)\right), \]
and $|y| \leq 2n^{-2/5}$ when $y\in\mC_1+i\epsilon$. Thus,
\begin{align*}
\left| \int_{\mC_1+i\epsilon}y^{-k}A(y)e^{-n\psi(y)}dy 
- a_0\int_{\mC_1+i\epsilon}y^{-k}e^{-nb_0y^2}dy \right|
&=
O\left(n^{-1/5}\int_{\mC_1+i\epsilon} 
\left|y^{-k}e^{-nb_0y^2}\right|dy\right) \\[+2mm]
&= O\left(n^{-1/5}\int_{\mC_1}\frac{1}{(t^2+\epsilon^2)^{k/2}}e^{-n(t^2-\epsilon^2)}dt\right) \\[+2mm]
&= O\left(n^{k/2-1/5}\int_{\R}\frac{1}{(nt^2+1)^{k/2}}e^{-nt^2}dt\right),
\end{align*}
where we have made the substitution $y=t+i\epsilon$ and 
used $\epsilon=1/\sqrt{n}$. Substituting
$w=nt^2$ then implies 
\[ O\left(n^{k/2-1/5}\int_{\R}\frac{1}{(nt^2+1)^{k/2}}e^{-nt^2}dt\right)
=  O\left(n^{(k-1)/2-1/5}\int_0^{\infty}\frac{1}{\sqrt{w}(w+1)^{k/2}}e^{-w}dw \right)
= O\left(n^{(k-1)/2-1/5}\right),
\]
since
\[ \int_0^{\infty}\frac{1}{\sqrt{w}(w+1)^{k/2}}e^{-w}dw \leq \int_0^{\infty}\frac{1}{\sqrt{w}}e^{-w}dw = \sqrt{\pi} \]
is finite for all $k>0$.

\emph{Step 3 (Add Tails Back).} Following the reasoning of Step 1, 
the integral 
\[ a_0\int_{\mC_2+i\epsilon} y^{-k}e^{-nb_0y^2}dy \]
is asymptotically negligible, meaning
\[ 
I = a_0\int_{\R+i\epsilon} y^{-k}e^{-nb_0y^2}dy + O\left(n^{(k-1)/2-1/5}\right).
\]
Proposition~\ref{prop:1DNegGauss} gives the
value of this integral, showing it grows as a constant
times $n^{(k-1)/2}$.
\end{proof}

\begin{rem}
Our proof of Lemma~\ref{lem:BigONegGauss} mirrors the
standard presentation of the saddle-point method for 
integrals of the form $\int_{\R} A(y)e^{-n\psi(y)}dy$,
where $A$ and $\psi$ are analytic (or, more generally, smooth). 
Our arguments are slightly more involved 
because the negative powers of $y$ introduced in 
our situation necessitates working off the real line.
\end{rem}

\begin{example}
\label{ex:Ex2NonGen}
Returning to Example~\ref{ex:Ex1NonGen}, 
we see that
\begin{align*}
\frac{-3}{2\pi i} \int_{\R+i\epsilon}  \frac{1}{y(1+iy)(1-2iy)} 
e^{-n\left[2\log(1+iy) + \log(1-2iy)\right]}dy
&= \frac{-3}{2\pi i} \int_{\R+i\epsilon}  \frac{e^{-n\left[3y^2 
+ O(y^3)\right]}}{y}\left(1+O(y)\right) dy + O(\tau^n) \\
&\sim\frac{-3}{2\pi i} \left(\int_{\R+i\epsilon}  
\frac{e^{-3ny^2}}{y} dy\right),
\end{align*}
so Proposition~\ref{prop:1DNegGauss} implies
\[ [x^{2n}y^n]F(x,y) \sim \frac{3}{2}. \]
\end{example}

Explicit expressions for 
non-generic directions where $(-\nabla h_{\br})(\bs)$ lies on 
higher codimensional faces of $N(\bs)$ follow from a similar 
approximation procedure, however deriving the necessary
bounds is slightly harder due to the presence of additional
variables. A rigorous
investigation of these cases is in progress.

\subsection{Bridging the Exponential Gaps} \label{ss:bridge}
Now we consider asymptotic transitions around non-generic directions.
Suppose we again have a contributing point $\bs$ for the non-generic direction
$\brhat$ under the assumptions of Section~\ref{ss:exact} and $\bpt = \bone$, 
so we have only simple poles.

To study such transitions, we write 
$\br = n\rhat + \sqrt{n}\theta_1\bv_1 + \cdots + 
\sqrt{n}\theta_{t-s}\bv_{t-s}$ where
\[ \bv_j = \sigma \odot \bb^{(j)} = 
\left(\sigma_1\bb^{(j)}_1,\dots,\sigma_d\bb^{(j)}_d\right) \]
and work from~\eqref{eq:nonGenAfterRes}. Then 
\[ \br \cdot \log\left(\bs-iM^{-1}
\begin{psmallmatrix}\bzer\\\br\end{psmallmatrix}\right) 
= n\phi(\by) + \sqrt{n}\theta_1\psi_1(\by) + \cdots + \sqrt{n}\theta_{t-s}\psi_{t-s}(\by) \]
where
\begin{align*}
\phi(\by) &= \rhat \cdot \log\left(\bs+iM^{-1}
\begin{psmallmatrix}\bzer\\\by\end{psmallmatrix}\right) \\
\psi_j(\by) &= \bv_j \cdot \log\left(\bs+iM^{-1}
\begin{psmallmatrix}\bzer\\\by\end{psmallmatrix}\right).
\end{align*}
Since the derivative of the $j$th coordinate of 
$\log\left(\bs+iM^{-1}\begin{psmallmatrix}\bzer\\\by
\end{psmallmatrix}\right)$ with respect to $y_k$ is $\sigma_j^{-1}iM^{-1}_{j,k+s}$,
\begin{align*} 
(\nabla \phi)(\bzer) &= (-i)\sum_{j=1}^d \frac{r_j}{\sigma_j}M^{-1}_{j,k+s} \\
&= (-i)\left[ \left(\frac{r_1}{\sigma_1},\dots,\frac{r_d}{\sigma_d}\right) 
\cdot M^{-1} \right]_{[s+1,\dots,t]} \\
&= i\left[ \left(a_1\bb^{(1)} + \cdots + a_s\bb^{(s)}\right) \cdot M^{-1} 
\right]_{[s+1,\dots,t]} \\
&= 0
\end{align*}
as the first $t$ rows of $M$ are $\bb^{(1)},\dots,\bb^{(t)}$, 
where the notation $P_{[a,\dots,b]}$ refers to the $a$ through $b$th 
columns of a matrix $P$. Similarly,
\[ (\nabla\psi_j)(\bzer) = (-i)\left[\bb^{(s+j)}\cdot 
M^{-1}\right]_{[s+1,\dots,t]} = (-i)\be^{(s+j)}. \]

In the same codimension~1 case from the last section, the modified 
saddle-point method described in the proof of Proposition~\ref{prop:nonGenDir}
shows that when $\theta=O(n^c)$ for some $c<1/2$ then
\[ I_{\theta} \sim \frac{\bs^{-\br}i\tG(\bs)}{|\det M|(2\pi)}
\underbrace{\int_{\R+i\epsilon} \frac{1}{y} e^{-n(\bq^T\bq/2)y^2  
- i\sqrt{n}\theta y} \; d\by}_{J(\theta)},
\]
where $M$ and $\bq$ are the same as above.
This implies the following result, where 
we recall the Gaussian error function 
$\erf (z) = \frac{2}{\pi} \int_{0}^z e^{-u^2/2} \, du$.
\begin{proposition}
\label{prop:nonGenErf}
Suppose
\[ F(\bz) = \frac{G(\bz)}{\ell_1(\bz) \cdots\ell_d(\bz)} \]
is simple and $\br$ is a non-generic direction with a unique contributing 
singularity $\bs$ of maximal height, which lies on the stratum 
$\ell_1=\cdots=\ell_d=0$. Suppose also that $G(\bs) \neq 0$ and $(-\nabla 
h_{\br})(\bs)$ lies on a codimension 1 face of $N(\bs)$, in such a way that
\[ (-\nabla h_{\br})(\bs) = \lambda_1 \cdot \bb^{(1)} + \cdots + 
\lambda_{d-1} \cdot \bb^{(d-1)} + 0 \cdot \bb^{(d)}\]
for some $\lambda_j>0$.  If $\bv = \left(\sigma_1 b^{(d)}_1, \dots,
\sigma_n b^{(d)}_n \right)$ and $\theta = O(n^c)$ for some $c<1/2$, then
\[ \left[\bz^{n \br + \theta\sqrt{n}\bv}\right]F(\bz) 
\sim 
\bs^{-n \br - \theta\sqrt{n}\bv} \, \frac{G(\bs)}{2\sigma_1\cdots \sigma_d |\det 
M|}
\left(\Psi\left(\frac{\theta}{\sqrt{2\bq^T\bq}}\right) + 1\right), \]
where $M$ is the matrix whose rows are the $\bb^{(j)}$ and $\bq$ is the
right-most column of
\[ 
Q = \begin{pmatrix} 
\sqrt{r_1}/\sigma_1 & \bzer & \bzer & \bzer \\[+1mm]
\bzer & \sqrt{r_2}/\sigma_2 & \bzer & \bzer \\[+1mm]
\bzer & \bzer & \ddots & \bzer \\[+1mm]
\bzer & \bzer & \bzer & \sqrt{r_d}/\sigma_d
\end{pmatrix} M^{-1}.
\]
\end{proposition}

\begin{proof}
The integral $J(\theta)$ is actually easier to analyze than the 
integrals from the previous section. In fact, differentiating under 
the integral sign with respect to $\theta$ simplifies to remove the
factor of $y$ in the denominator, giving
\[
J'(\theta) 
= i\sqrt{n} \int_{\R+i\epsilon} e^{-n(\bq^T\bq/2)y^2  
- i\sqrt{n}\theta y} \; d\by 
= i\sqrt{n}\int_{\R} e^{-n(\bq^T\bq/2)y^2  
- i\sqrt{n}\theta y} \; d\by 
= i\sqrt{\frac{2\pi}{\bq^T\bq}}e^{-\frac{\theta^2}{2\bq^T\bq}}.
\]
Integrating with respect to $\theta$ then yields
\[ J(\theta) = -\pi i \Psi\left(\frac{\theta}{\sqrt{2\bq^T\bq}}\right) + J(0), \]
where
\[ J(0) = \int_{\R+i\epsilon} \frac{e^{-n(\bq^T\bq/2)y^2}}{y} \; d\by = -\pi i \]
by Proposition~\ref{prop:1DNegGauss}.
\end{proof}

The general codimension case reduces to an analysis of integrals of the form 
\[ J(\bt) = \int_{\R^b+i\epsilon\bone} \frac{1}{y_1\cdots y_a} 
e^{-n(\by^T\mH\by)+i\sqrt{n}(\theta_1y_1+\cdots+\theta_ay_a)}d\by, \]
where $\mH$ is a positive definite matrix. A suitable change of variables 
splits this integral into a product of integrals, allowing one to reduce to
the case $a=b$ where every variable appears in the denominator.
This results in a so-called \emph{hydrodynamic} scaling limit, 
which will be rigorously covered in future work.

\bibliographystyle{plain}
\bibliography{bibl}

\end{document}